\documentclass[11pt]{amsart}

\usepackage[margin=25truemm]{geometry}

\usepackage{amssymb}
\usepackage{amsfonts}
\usepackage{mathrsfs}
\usepackage{bbm}
\usepackage{xcolor}
\hypersetup{% option list for hyperref
 setpagesize=false,
 hypertexnames=false,
 hyperfootnotes=true,
 bookmarksnumbered=true,%
 bookmarksopen=true,%
 colorlinks=true,%
 linkcolor=blue,
 citecolor=blue,
}
\theoremstyle{plain}
    \newtheorem{theorem}{Theorem}[section]
    \newtheorem*{theorem*}{Theorem}
    \newtheorem{lemma}[theorem]{Lemma}
    \newtheorem{proposition}[theorem]{Proposition}

\theoremstyle{definition}
    
    \newtheorem{remark}{Remark}[section]
    
\theoremstyle{remark}
    
\numberwithin{equation}{section}
\renewcommand{\l}{\left}
\renewcommand{\r}{\right}
\newcommand{\p}{\partial}

\newcommand{\lec}{\lesssim}
\newcommand{\gec}{\gtrsim}

\newcommand{\lan}{\langle}
\newcommand{\ran}{\rangle}
\newcommand{\R}{\mathbb{R}}

\newcommand{\Y}{\mathbb{Y}}

\newcommand{\D}{\mathcal{D}}
\newcommand{\cE}{\mathcal{E}}
\newcommand{\J}{\mathcal{J}}
\newcommand{\M}{\mathcal{M}}
\newcommand{\W}{\mathcal{W}}
\renewcommand{\Y}{\mathcal{Y}}
\newcommand{\de}{\delta}
\newcommand{\la}{\lambda}
\newcommand{\La}{\Lambda}
\newcommand{\Ga}{\Gamma}

\newcommand{\al}{\alpha}
\newcommand{\be}{\beta}
\renewcommand{\th}{\theta}
\newcommand{\si}{\sigma}
\newcommand{\Om}{\Omega}
\newcommand{\z}{\zeta}
\newcommand{\gat}{{\gamma}}
\newcommand{\Tt}{\tilde{T}}

\newcommand{\gt}{\tilde{g}}
\newcommand{\chit}{\tilde{\chi}}

\newcommand{\zz}{\tilde{z}}
\newcommand{\Ht}{\tilde{H}}
\newcommand{\mv}{\vec{m}}
\newcommand{\Mv}{\vec{M}}
\newcommand{\bv}{\vec{b}}
\DeclareMathOperator{\im}{Im}
\DeclareMathOperator{\re}{Re}

\DeclareMathOperator{\supp}{supp}

\DeclareMathOperator{\esssup}{ess\sup}

\begin{document}
%%%%%%%%%%%%%%%%%%%%%%%%%%%%%%%%%%%%%%%%%%%%%%%%%%%%%%%%

\title[TwoSolitonDelta]{Two-solitons with logarithmic separation for 1D NLS with repulsive delta potential}
\author[S. Gustafson]{Stephen Gustafson}
\address[S. Gustafson]{University of British Columbia, 1984 Mathematics Rd., Vancouver, Canada V6T1Z2.}
\email{gustaf@math.ubc.ca}
\author[T. Inui]{Takahisa Inui}
\address[T. Inui]{Department of Mathematics, Graduate School of Science, Osaka University, Toyonaka, Osaka, Japan 560-0043.
\newline 
University of British Columbia, 1984 Mathematics Rd., Vancouver, Canada V6T1Z2.}
\email{inui@math.sci.osaka-u.ac.jp 
%inui@math.ubc.ca
}
\date{\today}
%\keywords{nonlinear Schr\"{o}dinger equation, odd functions, global dynamics, threshold}
%\subjclass[2020]{35Q55,37K40 etc.}
\maketitle

\begin{abstract}
We consider the one-dimensional nonlinear Schr\"odinger equation with
focusing, power nonlinearity, and a repulsive delta potential.
We show that if the potential is not too strong, the construction
by \cite{Ngu19} of solutions converging strongly at time infinity to a
pair of logarithmically separating solitons can be adapted to 
accommodate the effect of the potential. On the other hand,
we show that if the potential is stronger, no such solutions exist.
\end{abstract}

\setcounter{tocdepth}{1}
\tableofcontents

%%%%%%%%%%%%%%%%%%%%%%%%%%%%%%%%%%%%%%%%%%
\section{Introduction}

%-----------------------------------------
\subsection{Background}

We consider the nonlinear Schr\"{o}dinger equation, in one space dimension, with focusing power nonlinearity, and with a delta potential at the origin $x=0$:
\begin{equation} \label{NLS}
\begin{split}
  & i\partial_t u + \partial_x^2 u - \Ga \de u + |u|^{p-1}u=0, \qquad (t,x) \in I \times \mathbb{R} \\
  & u(t_0,x) = u_0(x)
\end{split}
\end{equation}
where $I \ni t_0$ denotes an open interval, and we assume
\[
  \Ga \geq 0, \qquad  
\]
which means the delta potential is repulsive (rather than attractive), and
\begin{equation} \label{pass}
  p > 2, \quad p \not= 5.
\end{equation}
Equation~\eqref{NLS} with $\Ga=0$ is a standard model in quantum physics
and optics, and is well-studied as a model of nonlinear dispersive waves more
generally (see, eg, \cite{SS,Fib}). 
The delta potential describes an additional
(repulsive) point force on the system.

The (non-negative) self-adjoint (on $L^2(\R)$) 
operator $-\p_x^2 + \Ga \de$ appearing here may be defined by
\begin{equation} \label{delta}
\begin{split}
  & (-\p_x^2 + \Ga \de) f (x) = - f''(x), \\
  f \in \D^\Ga := &\{ f \in H^1(\R) \cap H^2(\R \backslash \{0\}) \; | \; 
  f'(0+)-f'(0-) = \Ga f(0) \}, 
\end{split}
\end{equation}
or via the quadratic form
\[
  q^\Ga (g, f) = 
  \int_\R \bar{g'}(y) f'(y) dy + \Ga \bar{g}(0)f(0), \qquad 
  f, g \in H^1(\R). 
\]
Note that
\begin{equation} \label{IBP}
  f \in \D^\Ga, \; g \in H^1 \; \implies \; 
  q^\Ga(g,f) = ( g, -f'' ).
\end{equation}

It is known (\cite{GiVe79,Caz03,LiPo15}) that the Cauchy problem~\eqref{NLS} is locally well-posed in the energy space $H^1(\mathbb{R})$: 
if $u_0 \in H^1(\mathbb{R})$, there exists 
a maximal time interval $I = (T_-,T_+) \ni t_0$,
$T_{\pm} = T_{\pm}(u_0,t_0)$,
and a unique solution $u \in C(I;H^1(\mathbb{R}))$ of~\eqref{NLS},
such that if $T_+ < \infty$, then $\limsup_{t \to T_+-} \| u(t,\cdot) \|_{H^1} = \infty$ (and similarly for $T_-$),
and whose energy and mass,
\[
  E^\Ga(u) = \frac{1}{2} q^\Ga(u,u)
  - \frac{1}{p+1} \|u\|_{p+1}^{p+1}, \qquad 
  M(u) =  \frac{1}{2} \| u\|_{2}^2,
\]
are conserved:
\begin{equation} \label{conservation}
  E^\Ga(u(t,\cdot)) \equiv E^\Ga(u_0),
  \qquad M(u(t,\cdot)) \equiv M(u_0).
\end{equation}
Additionally,
\[
  u_0 \in D^\Ga \; \implies \; 
  u \in C(I;D^\Ga) \cap C^1(I,L^2(\R)).
\]

When $\Ga = 0$,~\eqref{NLS} admits the {\it soliton} solution 
\begin{equation} \label{soliton}
  u(t,x) = e^{it} Q(x), 
\end{equation}
where 
\begin{equation}  \label{Qform}
  Q(x)= c_p 
  \left[ 2 \cosh \left(\frac{p-1}{2} x \right) \right]^{-\frac{2}{p-1}},
  \qquad  c_p = \left[ 2(p+1) \right]^{\frac{1}{p-1}}
\end{equation}
is positive, even, and solves the ODE
\begin{equation} \label{ode}
  -Q'' + Q -Q^{p} = 0.
\end{equation}
The exponential decay of $Q(x)$ as $|x| \to \infty$, 
\begin{equation} \label{asy}
  |Q(x) - c_p e^{-|x|}| + 
  |Q'(x) + c_p \frac{x}{|x|} e^{-|x|}|
  \lec e^{-p|x|},
\end{equation}
plays an important role in this paper,
and we will also make occasional use of the first-order relation
\begin{equation} \label{relation}
  (Q'(x))^2 + \frac{2}{p+1} Q^{p+1}(x) = Q^2(x) 
\end{equation}
With $\Ga = 0$, equation~\eqref{NLS} enjoys several invariances:
\begin{equation} \label{fullinvar}
\begin{split}
  &\mbox{phase } \; u(t,x) \mapsto e^{i \gat_0} u(t,x)  \;\; (\gat_0 \in \R); \quad
  \mbox{scale } \; u(t,x) \mapsto 
  \omega^{\frac{1}{p-1}} u(\omega t, \sqrt{\omega} x) \;\; (\omega > 0) \\
  &\mbox{translation } \; u(t,x) \mapsto u(t, x-z_0) \;\; (z_0 \in \R); \;\; 
  \mbox{Galilean } \;  u(t,x) \mapsto
  e^{i(vx-v^2 t)} u(t,x-2vt) \;\; (v \in \R)
\end{split}
\end{equation}
We note here that the addition of a potential, $\Ga \not= 0$, preserves phase invariance, but destroys the translation, Galilean and scale
invariances. A feature of the delta potential, however, is that a version of scale invariance can be recovered 
by adjusting the potential strength:
\begin{equation} \label{scaleinv}
  u(t,x) \mbox{ solves~\eqref{NLS} }
  \implies \omega^{\frac{1}{p-1}} u(\omega t, \sqrt{\omega} x)
  \mbox{ solves~\eqref{NLS} with }
  \Ga \mapsto \frac{\Ga}{\sqrt{\omega}}.
\end{equation}
Applying the symmetries~\eqref{fullinvar} to~\eqref{soliton} generates a 4-parameter family of moving soliton solutions of~\eqref{NLS} with $\Ga=0$:
\begin{equation} \label{single}
  u(t,x) = e^{i (v x + \gat(t))} 
  Q_{\omega}(x-z(t)),
  \quad Q_{\omega} (\cdot) = \omega^{\frac{1}{p-1}} Q(\sqrt{\omega} \cdot),
  \quad z(t) = 2v t + z_0, \; 
  \gat(t) = (\omega - v^2) t + \gat_0.
\end{equation}

In addition to the single soliton solutions~\eqref{single}, {\it multi-soliton}
solutions of~\eqref{NLS} with $\Ga=0$ have been constructed
(eg \cite{CMM,Com,MM}). These converge in $H^1$ as $t \to \infty$ to sums of single solitons with distinct velocities, hence are separating linearly in time and so are 
(asymptotically) non-interacting.
More recently, \cite{Ngu19} constructed rather 
special solutions of~\eqref{NLS} with $\Ga=0$
(and its higher-dimensional versions) which satisfy
\[
  \left\| u(t,\cdot) - e^{i \gat(t)} \left[ Q \left( \cdot -\frac{z(t)}{2} \right) + Q \left( \cdot +\frac{z(t)}{2} \right)
  \right] \right\|_{H^1} \lec \frac{1}{t}, \qquad
  z(t) = 2 \log t (1 + o(1))
\]
as $t \to \infty$.
Here the two solitons are {\it not} asymptotically non-interacting. 
Rather, their delicate interaction (through their exponentially small tails) 
provides an attraction which slows their escape and produces 
the logarithmic (rather than linear) separation rate. In the completely integrable case $p=3$ (and one dimension), such solutions, called double-pole solutions, were already known via integrable machinery \cite{ZS,Olm}.
For related constructions involving strong interactions between solitons or with potentials, see \cite{Jen,KMR,MM2,NR,Ary,IN}, as well as
the introduction to~\cite{Ngu19} for further references.

%--------------------------------------------------------
\subsection{Main results}

The primary motivation for the present work is to investigate
whether such logarithmically separating multi-solitons persist
in the presence of a potential forcing. Our main result is that
for a repulsive delta potential which is not too strong, they do:
\begin{theorem} \label{existence}
Assume~\eqref{pass} and $\Ga < \frac{3}{2}$. Then there exists an even solution
$u$ of~\eqref{NLS}, defined for (at least) $t \geq 0$,
and satisfying, for some $\gat(t) \in \R$,
\begin{equation} \label{logsep}
\begin{split}
  &\left\| u(t,\cdot) - e^{i \gat(t)} \left[ Q \left(\cdot-\frac{z(t)}{2} \right) + Q\left(\cdot+\frac{z(t)}{2}\right)
  \right] \right\|_{H^1} \lec \frac{1}{t}, \\
  & \qquad \qquad |z(t) - 2 \log t| \lec 1.
\end{split}
\end{equation}
\end{theorem}
On the other hand, we show that if the potential repulsion is
slightly stronger, we do not have such solutions:
\begin{theorem} \label{nonexistence}
Assume~\eqref{pass} and $\Ga > 2$. Then there exists no solution of~\eqref{NLS}
satisfying~\eqref{logsep}.
\end{theorem}

\begin{remark}
 These results show that the interplay between the soliton-soliton
 attraction and the potential-soliton repulsion is somewhat delicate.
 We conjecture that the existence/non-existence threshold 
 potential strength should actually be $\Ga = 2$.
\end{remark}
\begin{remark}
By using the scaling~\eqref{scaleinv}, for any $\omega > 0$, assuming
$\Ga < \frac{3}{2} \sqrt{\omega}$,
the conclusion of Theorem~\ref{existence} holds with $Q \mapsto Q_\omega$ and with $|z(t) - \frac{2}{\sqrt{\omega}} \log t| \lec 1$. And a non-existence analogue of Theorem~\ref{nonexistence} holds for
$\Ga > 2 \sqrt{\omega}$.
\end{remark}
\begin{remark}
By the conservation laws~\eqref{conservation} and the convergence~\eqref{logsep}, 
the action of a solution satisfying~\eqref{logsep} is
\begin{equation} \label{action}
  S^\Ga(u) = E^\Ga(u) + M(u) = 2 E^0(Q) + 2 M(Q) = 2S^0(Q)
\end{equation}
which, for $\Ga \geq 2$ is a threshold level of the action for even functions in the sense
\[
  \Ga \geq 2 \; \implies \; 
  \inf \{ S^\Ga(f) \; | \; f \in H^1_{even}(\R) \backslash \{0\}, 
  \; N^\Ga(f) = 0 \} = 2 S^0(Q),
\]
where
\[
  N^\Ga(f) = \| f' \|_2^2 + \Ga |f(0)|^2  -\|u\|_{p+1}^{p+1}
\]
is the Nehari functional (see \cite{FJ}).
For $\Ga < 2$, however, this threshold action is lower, and attained by an even ground state soliton $Q_\Ga$:
\[
  S^\Ga(Q_\Ga) < 2S^0(Q). 
\]
When $p > 5$, even solutions with action below this threshold were shown in~\cite{II}
either to scatter (to zero) or blow-up in
finite time,
while those at the threshold are classified in~\cite{GI1} for $\Ga < 2$
(where the ground state $Q_\Ga$ soliton and its stable manifold are additional dynamical possibilities), and in~\cite{GI2} for $\Ga \geq 2$; the latter result considerably strengthens Theorem~\ref{nonexistence} for $p > 5$, showing that all solutions with action satisfying~\eqref{action} either scatter 
or blow-up in finite time.
\end{remark}
\begin{remark}
In~\eqref{pass}, we have assumed $p > 2$ in order to simplify computations in many places
and to avoid having to further modify our approximate
solution -- see~\cite[Section 5]{Ngu19} for a discussion of
the complications that arise for $1 < p \leq 2$, and how
to address them. We also avoid  the mass-critical case $p=5$, for which solutions
satisfying~\eqref{logsep} are inconsistent with the virial identity (see~\cite[Remark 2]{Ngu19}), and for which different phenomena occur (see~\cite{MR}).
\end{remark}

%-----------------------------------------------------------------
\subsection{Strategy and guide}

Our proof of Theorem~\ref{existence} closely follows~\cite{Ngu19} in its outline, the steps of which we now summarize.
The PDE~\eqref{NLS} is solved backwards from a
sequence of final times tending to infinity, with final data
given by an approximate $2$-soliton solution defined in
Sections~\ref{parameters}-\ref{ourapprox}. Each of these solutions is expressed, using standard modulation theory (Section~\ref{parmod}), as a remainder from an approximate solution whose parameters (scale, phase, position, velocity) are time-modulated to ensure that the remainder
satisfies desired orthogonality conditions, and that the solitons are accelerating according to a carefully chosen force law, derived in Section~\ref{law}, which captures the main inter-soliton (and, in our case, potential-soliton) 
interactions, and is consistent with logarithmic separation. 
The heart of the argument is to show that the final soliton positions and velocities can be chosen so that uniformly, on intervals with fixed initial time, the remainders decay (in $H^1$) like $1/t$ and the parameters satisfy appropriate
estimates (consistent with logarithmic separation) -- these
are Propositions~\ref{uniform} (for $p<5$) and~\ref{uniform2} (for $p > 5$). Given these uniform estimates, the construction of Theorem~\ref{existence} follows from a compactness argument given in Sections~\ref{subcrit} ($p < 5$) and~\ref{supercrit} ($p > 5$). The uniform 
estimates are proved in Section~\ref{uniformsec},
again following the strategy of~\cite{Ngu19}:
estimates of the modulation parameters are 
obtained by taking suitable inner-products with 
the PDE in Section~\ref{modparest};
the inner-product of the remainder with the translational zero-mode (which is not fixed
by modulation) is controlled using a localized
momentum functional, together with the force law, in Section~\ref{momentum}. The remainder
is controlled by an energy estimate in 
Sections~\ref{enest}-\ref{coercive}; lastly,
the final soliton positions (and velocities, via
the force law) are chosen by a topological 
argument in Section~\ref{topology}. Section~\ref{super} shows how to adapt these
arguments for $p > 5$.

However, our proof requires a major adaptation of the arguments
of~\cite{Ngu19} to accommodate the effect of the potential -- essentially because
the potential-soliton interaction and
the soliton-soliton interaction are the
same size. In~\cite{Ngu19}, to control the remainder by a coercive energy functional, it is necessary to control an inner-product
\begin{equation} \label{oldinner}
  \lan \eta, \; i Q' \ran
\end{equation}
(which is not fixed by parameter modulation),
where $\eta$ is the remainder (the difference between the true and approximate solutions, expressed in the
frame of one of the solitons), and we note
that $Q'$ is the translational zero-eigenfunction:
\[
  L^+ Q' = 0, \qquad L^+ = -\p_x^2 +1 - pQ^{p-1}.
\]
Quantity~\eqref{oldinner} is controlled by a localized momentum functional, provided the leading term in the equation for its time derivative is removed by imposing a motion law
\begin{equation} \label{oldlaw}
  \frac{1}{2} M(Q) \;  \dot v = -H(z),
\end{equation}
where $z$ is the soliton position,
$v$ the velocity, and 
\begin{equation} \label{oldforce}
  H(z) \approx 2 c_p^2 e^{-z}
\end{equation}
is an integral (see~\eqref{Hdef}) capturing the interaction between the solitons.
The potential, though, interacts with the remainder $\eta$ to produce
an additional term of the same
size as the leading order~\eqref{oldforce}.
To remove this contribution, we replace
$Q'$ in~\eqref{oldinner} by the perturbed
translational eigenfunction
\[
  \left( L^+ + \Ga \de_{-\frac{z}{2}} \right) T_z = \nu_z T_z
\]
(the delta here is shifted to $x = -\frac{z}{2}$ since we are in the 
frame of one of the solitons),
and replace the quantity~\eqref{oldinner} with
\begin{equation} \label{newinner}
  \lan \eta, \; i T_z \ran,
\end{equation}
which may be controlled -- see Remark~\ref{replacement} for the exact
place this replacement is used.
In controlling~\eqref{newinner}, however,
interaction between $T_z$ and the soliton forces us to replace the motion law~\eqref{oldlaw} with one that captures
the additional effect of the potential:
\begin{equation} \label{newforce}
  \frac{1}{2} M(Q) \;  \dot v = -H(z)  
  + 2 \Ga c_p e^{-\frac{z}{2}} T_{z}\left(-\frac{z}{2}\right).
\end{equation}
That~\eqref{newforce} is the correct choice is encapsulated in the key estimate
of Proposition~\ref{Einnerest2}, and its
use in the localized momentum argument -- see Remark~\ref{modification}.

To carry out the program of estimates outlined above with these modifications, we need quite precise estimates of
the perturbed eigenfunction $T_z$ and its eigenvalue
$\nu_z$, which are given in Section~\ref{perturbed}. Their proofs, fairly involved, occupy Section~\ref{properties}.
Though the perturbing potential is of size $\Ga$, hence not small,
the fact that the unperturbed eigenfunction $Q'$ is small on its support allows us to estimate to the necessary precision.
One implication -- see~\eqref{Ttight} -- is that $T_z(-z/2)$ is of size $e^{-\frac{z}{2}}$. This shows that the second term (representing the soliton-potential interaction) on the right side of~\eqref{newforce} is the same size as the
first (representing the soliton-soliton interaction). Another implication is that
the net interaction in~\eqref{newforce} may be repulsive or attractive depending on the strength $\Ga$ of the potential -- see~\eqref{sign}. This is where the conditions $\Ga < \frac{3}{2}$ of Theorem~\ref{existence} and $\Ga > 2$ of
Theorem~\ref{nonexistence} enter.
The latter is proved in Section~\ref{non} by combining similar estimates to those
outlined above for the construction, with
the net attractivity of the force law~\eqref{newforce}, to derive a contradiction.

Though the fact that our potential is a delta function provides certain advantages, notably 
in establishing the precise eigenfunction estimates of Proposition~\ref{efprop}, 
it also presents several technical obstacles. 
First, we must cut-off our approximate $2$-soliton solution near the origin (see
Section~\ref{ourapprox}).
This both ensures that it lies in the domain of the linear operator in the PDE,
and, at a technical level, avoids a problematic singular contribution to the energy estimates of Section~\ref{enest} (see Remark~\ref{cutoffrem}). This cut-off creates additional terms to deal with in most of the estimates. Indeed, it plays a key role in the modified force law (see in particular Lemma~\ref{deltacomp} and its proof). 
As a knock-on effect of the cut-off, we have to modify the energy function of~\cite{Ngu19} -- see~\eqref{en} and Remarks~\ref{endef} and~\ref{cutpoint}.
Second, potential interactions with soliton tangent directions (such as $y Q$, $\La Q$) have to be carefully tracked, and in places create logarithmic losses compared to~\cite{Ngu19} -- see, eg, Remarks~\ref{logloss} and~\ref{logloss2}.
Third, the regularity of solutions to~\eqref{NLS} with $\Ga \not= 0$ seems
insufficient to use~\cite{Ngu19}'s proof of the modulation Lemma~\ref{modulation}
-- see Remark~\ref{modproofrem} -- 
so we provide a different proof in Section~\ref{modproof}. 
Finally, the compactness argument of Section~\ref{subcrit} requires continuous
dependence in fractional Sobolev spaces for solutions of~\eqref{NLS} with $\Ga \not= 0$ -- see equation~\eqref{Hsigmaconv} -- for which we
refer to~\cite{GIS}.

%-------------------------------------------------------------
\subsection{Notation}

As usual, we use the notation 
\[
  A \lec B
\]
if there is a constant $C > 0$, independent of any 
relevant parameters, such that $A \leq C B$. 

Lebesgue norms are denoted by
\[
  \| f \|_p = \| f \|_{L^p(\R)} = \left(\int_{\R} |f(x)|^p dx \right)^{\frac{1}{p}}, \qquad
  \| f \|_\infty = \esssup\limits_{x \in \R} |f(x)|,
\]
and Sobolev norms by
\[
  \| f \|_{H^1} = \| f \|_2 + \| f' \|_2,
  \qquad
  \| f \|_{W^{1,\infty}} = \| f \|_\infty
  + \| f ' \|_\infty, \quad
  \mbox{etc.}
\]
We will routinely make use of the Sobolev 
embedding inequality
\[
  \| f \|_{\infty} \lec \| f \|_{H^1}.
\]
We denote the complex and real $L^2(\R)$ inner-products by
\[
  ( f , g ) = \int_{\mathbb{R}} \bar f(x) g(x) dx, \qquad 
  \lan f , g \ran = \re \int_{\mathbb{R}} \bar f(x) g(x) dx.
\]
The generator of $L^2$-scaling is denoted
\[
  (\La f )(x) = x f'(x) + \frac{2}{p-1} f(x).
\]

With $\D^\Ga$ defined as in~\eqref{delta}, we denote
\[
  \D^\Ga_z = \left\{ f \left( \cdot - z \right) \; | \; f \in D^\Ga \right\}
\]
for $z \in \R$.

We denote the nonlinear term of~\eqref{NLS} by
\[
  N(u) = |u|^{p-1} u,
\]
express its Taylor expansion as
\begin{equation} \label{taylor}
  N(u + \eta) = N(u) + N'(u) \cdot \eta + \frac{1}{2} \bar \eta \cdot
  N''(u) \cdot \eta + \cdots,
\end{equation}
with
\[
   N'(u) \cdot \eta = \frac{p+1}{2} |u|^{p-1} \eta + 
   \frac{p-1}{2} |u|^{p-1} \bar{\eta},
\]
and
\[
  \frac{1}{2} \bar \eta \cdot N''(u) \cdot \eta
  = \frac{1}{2}(p-1)|u|^{p-3} \bar{u} \eta^2 +
  (p-1)|u|^{p-3} u |\eta|^2 +
  \frac{1}{2}(p-1)(p-3)|u|^{p-5} u (\re  \bar{u} \eta)^2,
\]
and use the pointwise estimates
\begin{equation} \label{tay1}
  |N(u + \eta) - N(u) -  N'(u) \cdot \eta| \lec
  |u|^{p-2}|\eta|^2 + |\eta|^p,
\end{equation}
\begin{equation} \label{tay2}
  |N(u + \eta) - N(u) - N'(u) \cdot \eta - \frac{1}{2} \bar \eta \cdot
  N''(u) \cdot \eta| \lec
  |u|^{\max(p-3,0)} |\eta|^3 + |\eta|^p,
\end{equation}
which hold for $p > 2$.

%%%%%%%%%%%%%%%%%%%%%%%%%%%%%%%%%%%%%%%%%%
\section{Approximate solution, perturbed eigenfunction, force law}

We begin by constructing an appropriate approximate solution, which is 
a variant of the one used in~\cite{Ngu19}.

%---------------------------------
\subsection{Parameters} \label{parameters}

We first introduce time-dependent scale and phase parameters, 
\[
  \la(s) > 0, \qquad \gat(s) \in \R,
\]
assumed to be $C^1$ on some interval $s \in \tilde I$,
and transform a solution $u$ of~\eqref{NLS} via
\begin{equation} \label{scalephase}
  e^{-i \gat(s)} \la(s)^{\frac{2}{p-1}} u(t(s),\la(s) y) = 
  w(s,y), \qquad \frac{d}{ds}t(s) = \la^2(s).
\end{equation}
In terms of $w$,~\eqref{NLS} becomes
\begin{equation} \label{NLS2}
  i \p_s w + \p_y^2 w + \la \Ga \de w - w + |w|^{p-1} w
  - i \frac{\dot \la}{\la} \La w + (1 - \dot \gat) w = 0,
\end{equation}
where $\cdot$ denotes $\frac{d}{ds}$.
\begin{remark}
\eqref{NLS} retains the phase invariance of its $\Ga=0$
counterpart, but not the scale invariance, resulting in 
a factor of $\la$ appearing in front of the potential.
In terms of domains: 
\[
  u(t,\cdot) \in \D^\Ga \;\; \iff \;\;
  w(s,\cdot) \in \D^{\la(s) \Ga}.
\]
\end{remark}
Next we introduce time-dependent (again, $C^1$) centre and velocity parameters
\[
  z(s) > 0, \qquad v(s) \in \R.
\]
As a result of how our parameters will be chosen below -- via modulation Lemma~\ref{modulation} -- they will always satisfy 
\begin{equation} \label{sizes}
  z \geq 1, \quad |v| \leq 1, \quad \frac{1}{2} \leq \la \leq \frac{3}{2},
\end{equation}
which we will routinely use in our estimates below.

%-------------------------------------------------
\subsection{The free approximate solution}

Given $z=z(s)$ and $v=v(s)$ as in~\eqref{sizes}, 
the construction of~\cite{Ngu19} begins with an approximate solution
\begin{equation} \label{P}
\begin{split}
  P^0(y;z,v)  &=
  e^{i \frac{v}{2}(y-\frac{z}{2})} Q(y - \frac{z}{2})
  + e^{-i \frac{v}{2}(y+\frac{z}{2})} Q(y + \frac{z}{2}) \\
  & =: \;\; P_+ \;\; + \;\; P_- 
\end{split}
\end{equation}
a superposition of solitons with positions $\pm \frac{z}{2}$
and velocities $\pm v$. As shown in~\cite[Lemma 6]{Ngu19}, 
the error produced when $P^0(y;z(s),v(s))$ is inserted into~\eqref{NLS2} with $\Ga=0$ is
\begin{equation} \label{error}
\begin{split}
  \cE_{P^0}^{\Ga=0} &= i \p_s P^0 + \p_y^2 P^0 - P^0
  + |P^0|^{p-1} P^0 - i \frac{\dot \la}{\la} \La P^0 + (1 - \dot \gat) P^0 \\
  & \; = -\left( e^{i \frac{v}{2}(\cdot)} \vec{m} \cdot \vec{M} Q \right)(y - \frac{z}{2}) +
  \left( e^{-i \frac{v}{2}(\cdot)} \Om\vec{m} \cdot \vec{M} Q \right)(y + \frac{z}{2}) + G
\end{split}
\end{equation}
where
\begin{equation} \label{vector}
  \vec{m} =
  \left[ \begin{array}{c} 
  m_1 \\ m_2 \\ m_3 \\ m_4 \end{array} \right] =
  \left[ \begin{array}{c} 
  \frac{\dot \la}{\la} \\ \frac{\dot z}{2} - v + \frac{\dot \la}{\la}\frac{z}{2}
  \\ \dot \gat - 1 + \frac{v^2}{4} - \frac{\dot \la}{\la} \frac{v}{2} \frac{z}{2} 
  - \frac{v}{2} \frac{\dot z}{2} \\ \frac{\dot v}{2} - \frac{\dot \la}{\la} \frac{v}{2} \end{array} \right], \quad
  \vec{M} = \left[ \begin{array}{c} i \La \\ i \p_y \\
  1 \\ y \end{array} \right], \quad
  \Om = \left[ \begin{array}{rrrr}  -1 & 0 & 0 & 0 \\
  0 & 1 & 0 & 0 \\ 0 & 0 & -1 & 0 \\ 0 & 0 & 0 & 1
  \end{array} \right]
\end{equation}
and the nonlinear interaction term
\begin{equation} \label{inter}
  G = |P^0|^{p-1} P^0 - |P_+|^{p-1}P_+ - |P_-|^{p-1}P_-
\end{equation}
satisfies
\begin{equation} \label{intersize}
  \| G \|_{W^{1,\infty}} \lec e^{-z}.
\end{equation}
We note also that
\begin{equation} \label{P0sum}
  \left\| P^0(\cdot;z,v) - \left[  Q(\cdot - \frac{z}{2}) +  
  Q(\cdot + \frac{z}{2})\right] \right\|_{H^1} \lec |v|.
\end{equation}

A key computation from~\cite[Lemma 7]{Ngu19} shows that
\begin{equation} \label{interform}
  \l|\lan G, (e^{i \frac{v}{2}(\cdot)} Q')(\cdot- \frac{z}{2}) \ran
  - H(z)\r| \lec e^{-z} (v^2 z^2 + e^{-\frac{z}{2}})
\end{equation}
where
\begin{equation} \label{Hdef}
  H(z) = p \l[ \int_{-\frac{z}{2}}^\infty Q^{p-1}(y) Q'(y) Q(y+z) dy + \int_{-\infty}^{-\frac{z}{2}} Q^{p-1}(y+z) Q'(y) Q(y) dy \r].
\end{equation}
Using~\eqref{asy} and~\eqref{Qform}, we can easily obtain 
\begin{equation} \label{interform2}  
  |H(z) - 2 c_p^2 e^{-z}| \lec e^{-2z} + e^{-\frac{p+1}{2} z}.
\end{equation}
This is shown in Lemma~\ref{Ip}.
The significance of~\eqref{interform} is its role in
determining the attraction force between the solitons. 
More precisely, an estimate of the form
\begin{equation} \label{Einnerest}
\begin{split}
  & \left|  \left\lan \cE_{P^0}^{\Ga=0}, \; 
  \left( e^{i \frac{v}{2}(\cdot) } Q' \right) 
  \left(\cdot- \frac{z}{2} \right) \right \ran  - \left[
   \frac{M(Q)}{2} \left( \dot v - \frac{\dot \la}{\la} v \right)
  + H(z) \right] \right| \\
  & \qquad \lec e^{-z} \left( |\vec{m}| z^2 + v^2 z^2 + e^{-\frac{z}{2}} \right)
\end{split}
\end{equation}
is used in~\cite[Proof of Lemma 13]{Ngu19}, to determine
that the 2-soliton system motion law should be
\begin{equation} \label{law0}
  \dot v(s) = -\frac{2}{M(Q)} H(z(s)).
\end{equation}
We will see below
how this force law is modified in the presence of the potential.

%---------------------------------
\subsection{Our approximate solution} \label{ourapprox}

In the presence of the delta potential, we modify the
approximate solution slightly by cutting off the origin. 
To this end, fix a cut-off function 
\[
  \chi \in C^\infty(\R) \mbox{ even }, \;\;
  0 \leq \chi \leq 1, \;\;
  \chi \equiv 0 \mbox{ in } \{ |x| \leq 1 \}, \;\;
  \chi \equiv 1 \mbox{ in } \{ |x| \geq 2 \},
\]
and define
\begin{equation} \label{P3}
  P(y; z,v) = \chi(y) P^0(y; z,v).
\end{equation}
We note that $P$ is an even and smooth function of $y$, 
and by virtue of vanishing near the origin
\begin{equation} \label{approxdomain}
  P(\cdot; z,v)\in \D^\al \qquad \forall \;\; \al \in \R.
\end{equation}
Moreover,
\[
  \| P(\cdot;z,v) - P^0(\cdot;z,v) \|_{H^1} \lec
  \| P^0(\cdot;z,v) \|_{H^1[-2,2]} 
  \lec (1 + |v|) e^{-\frac{z}{2}},
\]
and so by~\eqref{P0sum},
\begin{equation} \label{Psum}
  \left\| P(\cdot;z,v) - \left[  Q(\cdot - \frac{z}{2}) +  
  Q(\cdot + \frac{z}{2})\right] \right\|_{H^1} \lec 
  |v| + e^{-\frac {z}{2}}.
\end{equation}

The error produced by substituting $P(y;z(s),v(s))$
into~\eqref{NLS2} is
\begin{equation} \label{error3}
\begin{split}
  \cE_{P} &= i \p_s P + \p_y^2 P - \la \Ga \de P - P
  + |P|^{p-1} P - i \frac{\dot \la}{\la} \La P + (1 - \dot \gat) P \\
  & \; = \chi \cE_{P^0}^{\Ga=0} + \chi (\chi^{p-1}-1) |P^0|^{p-1} P^0
  + 2 (\p_y \chi) \p_y P^0 + (\p_y^2 \chi) P^0 
  - i \frac{\dot \la}{\la} (y \p_y \chi) P^0. 
\end{split}
\end{equation}
Note that the delta term contributes nothing to the error, 
since $P$ vanishes near the origin.

%%%%%%%%%%%%%%%%%%%%%%%%%%%%%%%%%%%%%%%%%%
\subsection{The perturbed translational eigenfunction} \label{perturbed}

In~\cite{Ngu19}, the attractive force between the solitons
is determined, as in~\eqref{interform}, 
through an inner product between the
nonlinear interaction term $G$ and the function $Q'$.
Because~\eqref{NLS} with $\Ga=0$ is invariant under spatial translation
(and as follows directly from differentiating~\eqref{ode}),
$Q'$ is a zero-eigenfunction of the Schr\"odinger operator
\[
  L^+ Q' = 0, \qquad L^+ = -\p_y^2 + 1 - p Q^{p-1},
\]
which appears in the linearization of~\eqref{NLS} with $\Ga=0$
around its soliton solution~\eqref{soliton}.
Including the potential, this operator becomes 
\begin{equation} \label{L+z}
  L^+_z = -\p_y^2 + 1 - p Q^{p-1} + \Ga \de_{-\frac{z}{2}}
\end{equation}
(here we have shifted $y$ so as to centre the right-hand soliton 
at the origin, so the delta potential is shifted to $y = -\frac{z}{2}$),
and so the eigenfunction is perturbed, and its eigenvalue is shifted.
This perturbed eigenfunction plays a key role in our analysis,
replacing $Q'$ in our analogue of~\eqref{interform}.
We need precise estimates:
\begin{proposition} \label{efprop}
Fix $\Ga > 0$. For all $0 < z$ sufficiently large, there exist
\[
  T_{z} \in \D(L^+_z) = \D^\Ga_{-\frac{z}{2}} \; \mbox{ and } \; \nu_{z} >0
\]
such that
\begin{equation} \label{eigen}
  L^+_{z} T_{z} = \nu_{z} T_{z}.
\end{equation}
The eigenvalue $\nu_z$ satisfies
\begin{equation} \label{nutight}
   \frac{1}{\Ga} \left[1 + \Ga - 
  (1 + 2\Ga)^{\frac{1}{2}} \right] - O(e^{-\frac{z}{2}})
  \leq \frac{\| Q' \|_2^2}{2 c_p^2} \frac{\nu_z}{e^{-z}} 
  \leq   \frac{\Ga}{\Ga+2} 
  + O(e^{-\min(\frac{1}{2},\frac{p-1}{p+3}) z}),
\end{equation}
and the derivative estimate
\begin{equation} \label{evz}
  |\p_z \nu_{z}| \lec \sqrt{\Ga}e^{-z}.
\end{equation}
The eigenfunction $T_z$ is real-valued, normalized as
\begin{equation} \label{normal}
  \| T_{z} \|_2 = \| Q' \|_2,
\end{equation}
and satisfies the norm estimate
\begin{equation} \label{esefest}
  \| T_{z} - Q' \|_{H^1 \cap L^1} \lec 
  \sqrt{\Ga} e^{-\frac{z}{2}},
\end{equation}
the $z$-dervative estimate
\begin{equation} \label{efz}
  \| \p_z T_{z} \|_{L^2} \lec \sqrt{\Ga} e^{-\frac{z}{2}},
\end{equation}
the pointwise estimates
\begin{equation} \label{pointwise}
  | T_{z}(y) - Q'(y) | \lec \sqrt{\Ga} \left\{ 
  \begin{array}{cc} 
  e^{-\sqrt{1-\nu_{z}}|y|} & |y| \geq \frac{z}{2} \\
  e^{-z} e^{-y} & -\frac{z}{2} \leq y \leq 0 \\
  e^{-z} \left( (y+z) e^{-y} + e^{-\frac{z}{2}} \right) & y \geq 0 
  \end{array} \right.,
\end{equation}
and the following refined estimate at $y=-\frac{z}{2}$:
\begin{equation} \label{Ttight}
  \frac{1}{\Ga}\left[ 1 + \Ga 
  -(1 +2\Ga)^{\frac{1}{2}} \right] - 
  O(e^{-\frac{z}{2}})
  \leq \left[1 - \frac{T_z(-\frac{z}{2})}{c_p e^{-\frac{z}{2}}} \right]
  \leq \frac{\Ga}{\Ga+2} + O(
  e^{-\min(\frac{1}{2},\frac{p-1}{p+3}) z}).
\end{equation}
\end{proposition}
The proofs of these properties, which are somewhat involved,  
are given in Section~\ref{properties}.

%%%%%%%%%%%%%%%%%%%%%%%%%%%%%%%%%%%%%%%%%%
\subsection{The modified force law} \label{law}

The main idea to capture the effect of the potential on the 
soliton dynamics is to replace $Q'$ in~\eqref{Einnerest}
with the perturbed eigenfunction $T_z$. The analogous estimate
is central to our analysis:
\begin{proposition} \label{Einnerest2}
\[
\begin{split}
  & \left|  \left\lan \cE_P, \; \left( e^{i \frac{v}{2}(\cdot)}  T_{z} \right) 
  \left(\cdot- \frac{z}{2} \right) \right\ran  - \left[
   \frac{M(Q)}{2} \left( \dot v - \frac{\dot \la}{\la} v \right)
  + H(z) - 2 \Ga c_p e^{-\frac{z}{2}} T_{z}(-\frac{z}{2}) \right] \right| \\
  & \qquad \lec e^{-z} \left( |\vec{m}| z^2 + v^2 z^2 + e^{-\frac{z}{2}} \right).
\end{split}
\]
\end{proposition}
Though this estimate captures the essential effect of the delta potential on the soliton
dynamics -- mediated through the modified
translational eigenfunction $T_z$ -- its proof
is somewhat involved, so to preserve the flow
of the main argument, we postpone it 
to Section~\ref{motionproof}.

Based on the estimate Proposition~\ref{Einnerest2}, we replace  
the motion law~\eqref{law0} for $\Ga=0$ by the modified one
\begin{equation} \label{vdyn2}
  \dot v(s) = -\Ht(z(s)), 
  \qquad \Ht(z) = \frac{2}{M(Q)} \left(  H(z) - 2 \Ga c_p e^{-\frac{z}{2}} T_{z}\left(-\frac{z}{2}\right) \right).
\end{equation}
Using~\eqref{interform2} and the bounds~\eqref{Ttight}, we see that
\begin{equation} \label{expapprox}
  |\Ht(z) - f(z) e^{-z}| \lec 
  e^{-\min \left(\frac{3}{2},\frac{2(p+1)}{p+3} \right)z}, 
\end{equation}
where
\begin{equation} \label{fbounds}
  2 - (1 + 2\Ga)^{\frac{1}{2}} \leq \frac{f(z)}{\si^2} \leq \frac{2 - \Ga}{2 + \Ga}, \qquad
  \si = \frac{2 c_p}{\sqrt{M(Q)}}.
\end{equation}
In particular, we have:
\begin{equation} \label{sign}
\begin{split}
  &\Ga < \frac{3}{2} \; \implies \; \exists \; 
  f_* > 0 \; \mbox{ s.t. } \; f(z) \geq f_* \\
  &\Ga > 2 \; \implies \;  \exists \; 
  f^* > 0 \; \mbox{ s.t. } \; f(z) \leq -f^*.
\end{split}
\end{equation}
The Newtonian dynamics governing the soliton motion via~\eqref{vdyn2} are given by
\[
  \dot z(s) = 2v(s), \qquad  \dot v(s) = -\Ht(z(s)),
\]
with conserved energy
\[
  E(z,v) = v^2 - F(z), \qquad
  F(z) := \int_z^\infty \Ht(z') d z'.
\]
By~\eqref{expapprox} and~\eqref{fbounds}, we have
\begin{equation} \label{Fbounds}
  2 - (1 + 2\Ga)^{\frac{1}{2}} - O(e^{-\min\left(\frac{1}{2},\frac{p-1}{p+3}\right)z})
  \leq \frac{F(z)}{\si^2 e^{-z}} \leq
  \frac{2 - \Ga}{2 + \Ga} +  O(e^{-\min\left(\frac{1}{2},\frac{p-1}{p+3}\right)z}).
\end{equation}
Assuming $\Ga < \frac{3}{2}$, so $F(z) > 0$, our desired trajectory has $E=0$, that is
\begin{equation} \label{vclass}
  v = \sqrt{F(z)},
\end{equation}
and
\[
  \frac{dz}{ds} = 2 v = 2 \sqrt{F(z)},
\]
so defining
\[
  \z(z) := \int_{z_0}^z \frac{d z'}{2 \sqrt{F(z')}},
\]
($z_0$ large enough that Proposition~\ref{efprop} applies for $z \geq z_0$),
the desired classical trajectory is given by
\begin{equation} \label{zclass}
  \z(z(s)) - s = c 
\end{equation}
for a constant $c$. By~\eqref{Fbounds}, we have,
if $\Ga < \frac{3}{2}$,
\begin{equation} \label{zetabounds}
  \sqrt{\frac{2+\Ga}{2-\Ga}} - 
  O(z e^{-\min\left(\frac{1}{2},\frac{p-1}{p+3} \right) z})
  \leq \frac{\si \z(z)}{e^{\frac{z}{2}}} \leq \frac{1}{\sqrt{2 - \sqrt{1+2\Ga}}} 
  + O(z e^{-\min\left(\frac{1}{2},\frac{p-1}{p+3} \right) z})
\end{equation}
so in particular
\begin{equation} \label{zetabounds2}
    |z - 2 \log \z(z) | \lec 1.
\end{equation}

%%%%%%%%%%%%%%%%%%%%%%%%%%%%%%%%%%%%%%%%%%
\subsection{Parameter modulation} \label{parmod}

For solutions which are close to the approximate solution family, we would like to modulate the phase, scale and position parameters, writing
\begin{equation} \label{xi}
  e^{-i \gat(s)} \la(s)^{\frac{2}{p-1}} u(t(s), \la(s) y)  = P(y;z(s),v(s)) + \xi(s,y),
\end{equation}
so that the remainder $\xi$, re-expressed in the frame
of the right-hand soliton,
\begin{equation} \label{eta}
  \eta(s,y) = e^{-i \frac{v(s)}{2} y} \xi \left( s, y + \frac{z(s)}{2} \right),
\end{equation}
satisfies some desired orthogonality conditions.
Proofs that this can be done locally in time,
in various contexts, are by now quite standard (going back at least to~\cite{Wein})
-- the one below is an analogue of \cite[Lemma 9]{Ngu19}.
Before stating it precisely, we note that under the
decomposition~\eqref{xi}-\eqref{eta},
the equation~\eqref{NLS} becomes the following
equation for $\eta$,
\begin{equation}  \label{etaeq}
\begin{split}
  0 &= i \p_s \eta + (\p_y^2 - \Ga \delta_{-\frac{z}{2}} -1 ) \eta
  + |P_1 + \eta|^{p-1}(P_1 + \eta) - |P_1|^{p-1} P_1 
  \\ & \qquad \qquad +
  \mv \cdot \Mv \eta + \cE_{P_1} + (\la-1) \Ga \de_{-\frac{z}{2}} \eta,
\end{split}
\end{equation}
where
\[
  P_1(s,y) = e^{-i \frac{v}{2} y} P(s,y+\frac{z}{2}), \qquad
  \cE_{P_1}(s,y) = e^{-i \frac{v}{2} y} \cE_P(s,y+\frac{z}{2}),
\]
and that
\begin{equation} \label{domains}
  u(t,\cdot) \in \D^\Ga \;\; \implies \;\; \xi(s,\cdot) \in \D^{\la(s) \Ga}, \quad
  \eta(s,\cdot) \in \D^{\la(s) \Ga}_{-\frac{z}{2}}.
\end{equation}
\begin{lemma} \label{modulation}
There is $\epsilon_0 > 0$ such that given $s_{\mathrm f} \in \R$, and
\begin{equation} \label{endparam}
  \la_{\mathrm f},  \gat_{\mathrm f}, z_{\mathrm f}, v_{\mathrm f}
  \in (0,\infty) \times \R \times (0,\infty) \times \R, \;\; \mbox{ with } \;\;
  |\la_{\mathrm f}-1| + \frac{1}{z_{\mathrm f}} + |v_{\mathrm f}| z_{\mathrm f} < \epsilon_0, 
\end{equation}
and an even solution 
\begin{equation} \label{givensol}
  u \in C(J;\D^\Ga) \cap C^1(J;L^2(\R))
\end{equation}
of~\eqref{NLS} on an open interval $J \ni s_{\mathrm f}$ with
\begin{equation} \label{endclose}
  \| e^{-i \gat_{\mathrm f}} (\la_{\mathrm f})^{\frac{2}{p-1}} u(s_{\mathrm f}, \la_{\mathrm f} \cdot)  - P( \cdot ; z_{\mathrm f}, v_{\mathrm f}) \|_{H^1} < \epsilon_0,
\end{equation}
there exist, on an open interval $I \ni s_{\mathrm f}$, unique $C^1$ functions
\begin{equation} \label{paramexist}
  \vec{q}(s) = \left( \la(s), \gat(s), z(s) , v(s)  \right) \in
  (0,\infty) \times \R \times (0, \infty) \times \R
\end{equation}
with
\[
  \vec{q}(s_{\mathrm f}) =  \left( \la_{\mathrm{f}}, \gat_{\mathrm f}, z_{\mathrm f}, v_{\mathrm f} \right),
\]
such that setting
\begin{equation} \label{ts}
  t(s) = s_{\mathrm f} - \int_s^{s_{\mathrm f}} \la^2(s) ds \;\; \in J,
\end{equation}
the motion law~\eqref{vdyn2}, and the orthogonality relations
\begin{equation} \label{difforthos}
  \frac{d}{ds} \lan \eta(s, \cdot), Q \ran = 
  \frac{d}{ds} \lan \eta(s, \cdot), yQ \ran = 
  \frac{d}{ds} \lan \eta(s, \cdot), i \La Q \ran = 0
\end{equation}
hold for all $s \in I$. 
In particular, if the orthogonality conditions
\begin{equation} \label{orthos}
  \lan \eta(s, \cdot), Q \ran = \lan \eta(s, \cdot), yQ \ran = \lan \eta(s, \cdot), i \La Q \ran = 0
\end{equation}
hold at $s=t=s_{\mathrm f}$, then they hold for $s \in I$.
\end{lemma}
Unfortunately, the proof of~\cite[Lemma 9]{Ngu19} does
not directly carry over, as the delta potential 
creates a technical complication. So we provide a 
proof in Section~\ref{modproof}.
\begin{remark}
\label{fullmod}
For the non-existence argument in Section~\ref{non}, we will also need a variant of this lemma,
where we replace the motion law~\eqref{vdyn2} 
(allowing the velocity $v(s)$ to modulate)
with a fourth orthogonality relation
\[
  \frac{d}{ds} \lan \eta(s,\cdot), i T_{z(s)} \ran = 0.
\]
The proof is essentially the same.
\end{remark}

%%%%%%%%%%%%%%%%%%%%%%%%%%%%%%%%%%%%%%%%%%
\section{Construction of the logarithmically separating two-soliton}

In this section we will give the proof of Theorem~\ref{existence},
modulo the uniform estimates on approximate solutions which 
are shown in the next section.
As in~\cite{Ngu19}, we treat the subcritical ($p < 5$) and
supercritical ($p > 5$) cases separately.

%----------------------------------------------------------
\subsection{The construction for subcritical nonlinearities}
\label{subcrit}

Here we give the proof of Theorem~\ref{existence} for 
the subcritical powers $p < 5$.

{\it Proof of Theorem~\ref{existence} for $p < 5$}:
let $s_{\mathrm f} \geq 1$ be sufficiently large. 
Let $u$ be the solution of~\eqref{NLS} with
\[
  u(s_{\mathrm f},x) = P(x; z_{\mathrm f}, v_{\mathrm f}), \qquad
  z_{\mathrm f} \geq \frac{1}{\epsilon_0}, \quad
  v_{\mathrm f} = \sqrt{F(z_{\mathrm f})},
\]
where $z_{\mathrm f}$ will be precisely chosen later, and then the choice of
$v_{\mathrm f}$ is determined by~\eqref{vclass}.
Then~\eqref{endparam} and~\eqref{endclose} hold with
\[
  \la_{\mathrm f} = 1, \qquad \gat_{\mathrm f} = 0,
\]
and by~\eqref{approxdomain},~\eqref{givensol} holds, so we may invoke Proposition~\ref{modulation} to provide 
an interval $I \ni s_{\mathrm f}$, and parameters as in~\eqref{paramexist}
satisfying
\begin{equation} \label{ICs}
  \left( \la(s_{\mathrm f}), \gat(s_{\mathrm f}), z(s_{\mathrm f}), v(s_{\mathrm f}) \right) = 
  \left( 1, 0, z_{\mathrm f}, v_{\mathrm f} \right)
\end{equation}
such that the decomposition~\eqref{xi}-\eqref{eta} holds, with
\begin{equation} \label{ICxi}
  \xi(s_{\mathrm f},\cdot) = \eta(s_{\mathrm f},\cdot) = 0,
\end{equation}
as do the orthogonality conditions~\eqref{orthos} 
and the motion law~\eqref{vdyn2}.

The key to the construction is to prove estimates of this 
solution on a time interval $[s_0, s_{\mathrm f}]$,
with $s_0$ fixed, which are uniform in $s_{\mathrm f}$ as $s_{\mathrm f} \to \infty$
(this is the analogue of~\cite[Prop 10]{Ngu19}):
\begin{proposition} \label{uniform}
There is $s_0 > 0$ such that for all $s_{\mathrm f} > s_0$ there is
$z_{\mathrm f} > 0$ such that for all $s \in [s_0,s_{\mathrm f}]$,
the parameters satisfy
\begin{equation} \label{paramests}
  |z(s) - 2 \log s| \lec 1, \quad
  |\z(z) - s| \lec \frac{s}{\log^{\frac{1}{2}} s}, \quad
  |v(s)| + |\la(s) - 1| \lec \frac{1}{s},
\end{equation}
and the remainder satisfies
\begin{equation} \label{stayclose}
  \| \xi(s,\cdot) \|_{H^1} \lec \frac{1}{s}.
\end{equation}
\end{proposition}
\begin{remark}
While a priori the existence of parameters~\eqref{paramexist} and decomposition~\eqref{xi} is only guaranteed in a neighbourhood
of $s = s_{\mathrm f}$, the estimates~\eqref{paramests} and~\eqref{stayclose}
ensure, if $s_0$ is chosen large enough, that~\eqref{endparam}
and~\eqref{endclose} continue to hold (with $s_{\mathrm f}$ replaced by $s$),
so that the existence of parameters~\eqref{paramexist} and decomposition persists.
\end{remark}

Proposition~\ref{uniform} is the main ingredient in the construction
-- we prove it in the next section.
Given these uniform estimates, the proof of Theorem~\ref{existence} is 
completed via a compactness argument, which, since it proceeds essentially as in~\cite[Sec 4]{Ngu19}, we 
give only a rough sketch of.

First, using the relations~\eqref{scalephase} and the estimate for $\la(s)$
in~\eqref{paramests},
it is straightforward to convert Proposition~\ref{uniform}
into the following statement:
there exist $t_0 \geq 1$, a sequence $T_n \to \infty$, and a sequence of
even solutions $u_n \in C([t_0,T_n]; H^1)$ to~\eqref{NLS} satisfying
\[
  u_n(t,x) = e^{i \gat_n(t)} \la_n^{-\frac{2}{p-1}}(t) 
  \left[ P \left( \frac{x}{\la_n(t)} ; z_n(t),v_n(t) \right) 
  + \xi_n \left(t, \frac{x}{\la_n(t)} \right) \right],
\]
with parameters satisfying the uniform (in $n$) estimates
\begin{equation} \label{paramt}
  |z_n(t) - 2 \log t |\lec 1, \quad
  |\z(z_n(t)) - t| \lec \frac{t}{\log^{\frac{1}{2}} t}, \quad
  |v_n(t)| + |\la_n(t) - 1| \lec \frac{1}{t},
\end{equation}
and remainders satisfying the uniform estimate
\begin{equation} \label{remt}
  \| \xi_n(t,\cdot) \|_{H^1} \lec \frac{1}{t}.
\end{equation}

Next, the compactness assertion: there exists even $u_0 \in H^1$ such that up to subsequence,
\[ 
  u_n(t_0,\cdot) \to u_0 \; \mbox{ weakly in } H^1 \mbox{ and strongly in } L^2.
\]
The proof of this is exactly as in~\cite[Lemma 15]{Ngu19}: 
since~\eqref{remt} implies a uniform in $n$ (and also $t$) $H^1$ bound,
$\| u_n(t,\cdot) \|_{H^1} \leq C$, it suffices to show that the tail
of the mass density is uniformly small -- that is,
\[
  \forall \; \de_1 > 0, \; \exists \; n_0, K_1, s.t. \;
  n \geq n_0 \implies 
  \int_{|x| > K_1} |u(t_0,x)|^2 dx < \de_1.
\]
This follows from using~\eqref{paramt} and~\eqref{remt}
to conclude that the tail of $u(t_1,\cdot)$ can be made sufficiently small
at a later time $t_1$, then using a finite speed of propagation of mass-type
estimate (using the uniform $H^1$ bound) to pull the conclusion back to time $t_0$. 

Next, let $u$ be the solution of~\eqref{NLS} with $u(t_0,\cdot) = u_0$.
By interpolation, the above compactness assertion implies
\[ 
  u_n(t_0,\cdot) \to u_0 \; \mbox{ strongly in } H^\si, \quad \si \in [0,1),
\]
and so, fixing a $\si \in (\frac{1}{2},1)$, and invoking 
continuous dependence in $H^\si$ of solutions of~\eqref{NLS} on initial data
(see~\cite{GIS}), we have
\begin{equation} \label{Hsigmaconv}
  \| u_n(t,\cdot) - u(t,\cdot) \|_{H^\si} \to 0 \quad
  \mbox{ for each } t \in [t_0,\infty).
\end{equation}
It then follows, just as in~\cite{Ngu19} that
\[
  \gat_n(t) \to \gat(t), \quad \la_n(t) \to \la(t),
  \quad z_n(t) \to z(t), \quad v_n(t) \to v(t),
\]
and
\[
  u(t,x) = e^{i \gat(t)} \la^{-\frac{2}{p-1}}(t) 
  \left[ P \left( \frac{x}{\la(t)} ; z(t),v(t) \right) 
  + \xi \left(t, \frac{x}{\la(t)} \right) \right]
\]
with
\[
  \xi_n(t) \to \xi(t) \; \mbox{ weakly in } H^1 \mbox{ and strongly in } H^\si,
\]
and the estimates 
\[
  |z(t) - 2 \log t |\lec 1, \quad
  |\z(z(t)) - t| \lec \frac{t}{\log^{\frac{1}{2}} t}, \quad
  |v(t)| + |\la(t) - 1| \lec \frac{1}{t},
\]
inherited from~\eqref{paramt}, and
\[
  \| \xi(t,\cdot) \|_{H^1} \lec \frac{1}{t}
\]
inherited from~\eqref{remt}.

The proof of Theorem~\ref{existence} for the subcritical cases $p < 5$
is then finished by observing that
\[
\begin{split}
  \| e^{-i \gat(t)} u(t,\cdot) - P(\cdot; z(t),v(t)) \|_{H^1}
  &\lec \left\| P(\cdot; z(t),v(t)) - 
  \la^{-\frac{2}{p-1}}(t) 
  P \left( \frac{\cdot}{\la(t)}; z(t),v(t) \right) \right\|_{H^1} \\
  & \quad + \left\|  \la^{-\frac{2}{p-1}}(t) \xi \left(t, \frac{\cdot}{\la(t)} \right) \right\|_{H^1} \\
  & \lec |\la(t) - 1| + \| \xi(t,\cdot) \|_{H^1} \lec \frac{1}{t},
\end{split}
\]
and, by~\eqref{Psum},
\[
  \left\| P(\cdot; z(t),v(t)) -   \left[  Q(\cdot - \frac{z(t)}{2}) +  
  Q(\cdot + \frac{z(t)}{2})\right]  \right\|_{H^1}
  \lec |v(t)| + e^{-\frac{z(t)}{2}} \lec \frac{1}{t}, 
\]
and finally by replacing $u(t,\cdot)$ with $u(t_0 + t,\cdot)$
for $t \geq 0$.
$\Box$

%------------------------------------------------------------
\subsection{The construction for supercritical nonlinearities}
\label{supercrit}

Here we explain how to modify the above arguments
for $p < 5$ in order to prove Theorem~\ref{existence} for 
the supercritical powers $p > 5$.

{\it Proof of Theorem~\ref{existence} for $p > 5$}:
the key difference is that the linearized energy
(defined in~\eqref{en}) is no longer coercive under the orthogonality conditions~\eqref{orthos}
obtained by modulating the phase, centre and scale -- a fact very closely related to  
the instability of the individual soliton
for $p > 5$.
Additional control of the unstable direction is required. 
More precisely, there are additional (even) eigenfunctions
\begin{equation} \label{unstable}
  \left( L^+ \re + i L^- \im \right) Y^{\pm} 
  = \pm i e_1 Y^{\pm}, \quad e_1 > 0, \quad
  \bar{Y}^+ = Y^-, \quad |Y^{\pm}(y)| \lec e^{-\sqrt{1 + e_1^2}|y|},
\end{equation}
and an enlarged invariant (for the linearized dynamics around a single soliton with $\Ga = 0$) subspace
\[
  \{ Q, \; yQ, \; i Q', \; i \La Q, \; i Y^+, \; i Y^- \}^{\perp},
\]
on which the linearized energy is coercive by means of (see \cite[eq (1.17)]{Ngu19})
\[
\begin{split}
  p > 5 : \quad \| \eta \|_{H^1}^2 &\lec 
  (\re  \eta, L^+ \re \eta) + 
  (\im \eta, L^- \im \eta) \\
  & \quad + \lan y Q, \eta \ran^2 + \lan i \La Q, \eta \ran^2
  + \lan i Y^+, \eta \ran^2 + \lan i Y^-, \eta \ran^2.
\end{split}
\]

Analogous to~\cite[Sec 6]{Ngu19}, we define the following even (in $y$) functions
\[
\begin{split}
  & \Y^{\pm}(s,y) = \chi(y) \left[ \left( e^{i \frac{v}{2}(\cdot)} Y^\pm \right)
  (y - \frac{z}{2}) +  \left( e^{-i \frac{v}{2}(\cdot)} Y^\pm \right) (y + \frac{z}{2}) \right] \\
  & Z(s,y) = \chi(y) \left[ \left( e^{i \frac{v}{2}(\cdot)} i \La Q \right)
  (y - \frac{z}{2}) +  \left( e^{-i \frac{v}{2}(\cdot)} i \La Q \right) (y + \frac{z}{2}) \right] \\
   & V(s,y) =  \chi(y) \left[ \left( e^{i \frac{v}{2}(\cdot)} i Q' \right)
  (y - \frac{z}{2}) -  \left( e^{-i \frac{v}{2}(\cdot)} i Q' \right) (y + \frac{z}{2}) \right] \\
   & W(s,y) =  \chi(y) \left[ \left( e^{i \frac{v}{2}(\cdot)} (\cdot)Q \right)
  (y - \frac{z}{2}) -  \left( e^{-i \frac{v}{2}(\cdot)} (\cdot)Q \right) (y + \frac{z}{2}) \right],
\end{split}
\]
and take as final conditions, not $\xi(s_{\mathrm f},\cdot) = 0$ (as in~\eqref{ICxi}), but rather: with $z = z_{\mathrm f}$ and $v = v_{\mathrm f}$ given,
\[
  \xi(s_{\mathrm f},\cdot) = b^+ i\Y^+ + b^- i \Y^- + b_1 Z + b_2 V + b_3 W.
\]
\begin{remark} \label{superdomain}
Here the cut-off function $\chi$ removes the origin and so ensures
$\Y^{\pm}$, $Z$, $V$ and $W$, and therefore also $\xi(s_{\mathrm f},\cdot)$,
all lie in the domain $\D^\al$
(for any $\al \in \R$).
\end{remark}
The parameters $\bv = (b^+,b^-,b_1,b_2,b_3)$ are chosen so that
\[
  \bv = \bv(a_{\mathrm f}), \quad a_{\mathrm f} \in [(-s_{\mathrm f})^{-\frac{3}{2}}, (s_{\mathrm f})^{-\frac{3}{2}}], \quad |\bv| \lec |a_{\mathrm f}|,
\]
and the orthogonality conditions
\begin{equation} \label{finalorth}
\begin{split}
  & \qquad \qquad \qquad 
  \lan \eta(s_{\mathrm f},\cdot), i Y^- \ran = a_{\mathrm f}, \\
  &\lan \eta(s_{\mathrm f},\cdot), i Y^+ \ran = \lan \eta(s_{\mathrm f},\cdot), i \La Q \ran
  = \lan \eta(s_{\mathrm f},\cdot), yQ \ran = \lan \eta(s_{\mathrm f},\cdot), i T_z \ran = 0
\end{split}
\end{equation}
hold at $s = s_{\mathrm f}$.
That this can be done is shown just as in~\cite[Lemma 20]{Ngu19},
with the only modifications being the change 
$Q' \mapsto T_z$ in the last condition,
and the introduction of the cut-off
$\chi$ which produces only small error terms.

We let $u$ be the solution of~\eqref{NLS} with
\[
  u(s_{\mathrm f},\cdot) = P(\cdot; z_{\mathrm f},v_{\mathrm f}) + \xi(s_{\mathrm f},\cdot),
\]
with final parameter values
\[
  z_{\mathrm f} \geq \frac{1}{\epsilon_0}, \quad v_{\mathrm f} = \sqrt{F(z_{\mathrm f})}, \quad
  \la_{\mathrm f} = 1, \quad \gat_{\mathrm f} = 0.
\]
By Remark~\ref{superdomain} and~\eqref{approxdomain}, our solution
$u$ satisfies~\eqref{givensol}, and we may modulate the parameters as in Proposition~\ref{modulation} 
to produce the decomposition~\eqref{xi}.
The key uniform estimates are recorded in this
variant of Proposition~\ref{uniform}:
\begin{proposition} \label{uniform2}
There is $s_0 > 0$ such that for all $s_{\mathrm f} > s_0$ there are
$z_{\mathrm f} > 0$ and $a_{\mathrm f} \in [(-s_{\mathrm f})^{-\frac{3}{2}}, (s_{\mathrm f})^{-\frac{3}{2}}]$ 
such that for all $s \in [s_0,s_{\mathrm f}]$,
the parameters satisfy
\begin{equation}
  |z(s) - 2 \log s| \lec 1, \quad
  |\z(z) - s| \lec \frac{s}{\log^{\frac{1}{2}} s}, \quad
  |v(s)| + |\la(s) - 1| \lec \frac{1}{s},
\end{equation}
and the remainder satisfies
\begin{equation}
  \| \xi(s,\cdot) \|_{H^1} \lec \frac{1}{s}.
\end{equation}
\end{proposition}
The proof is outlined in Section~\ref{super}.

With these uniform estimates in hand, the 
proof of Theorem~\ref{existence} for $p > 5$
is completed exactly as above for $p < 5$.
$\Box$

%%%%%%%%%%%%%%%%%%%%%%%%%%%%%%%%%%%%%%%%%%
\section{Uniform estimates} \label{uniformsec}

Here we obtain the uniform estimates on approximate
solutions which are the main ingredient in the construction.
We first consider Proposition~\ref{uniform} (for $p < 5$),
whose proof occupies the next several subsections.
Then in Section~\ref{super}, we explain how these 
arguments may be modified to prove Proposition~\ref{uniform2} (for $p > 5$).  

{\it Proof of Proposition~\ref{uniform}}:
consider the bootstrap estimates
\begin{equation} \label{boot}
  \left| \z(z) - s \right| \leq \frac{s}{\log^{\frac{1}{2}} s}, \qquad \| \xi \|_{H^1} \leq \frac{C^*}{s},
\end{equation}
with $C^* > 1$ to be chosen.
Note that the second of these obviously holds at $s = s_{\mathrm f}$,
by~\eqref{ICxi}.
If~\eqref{boot} hold, then as consequences we have
\begin{equation} \label{consequence}
  |e^{-\frac{z}{2}}| \lec \frac{1}{s}, \qquad 
  |z - 2 \log s| \lec 1, \qquad
  |\dot v| \lec \frac{1}{s^2},
  \qquad 
  |v| \lec \frac{1}{s}.
\end{equation}
The first two use~\eqref{zetabounds2}, the third
uses~\eqref{vdyn2} and~\eqref{fbounds}, and the fourth follows from
time integration of the third, the choice of $v_{\mathrm f}$ in~\eqref{ICs},
and~\eqref{Fbounds}.

We will work on the time interval
\[
  s \in [s^*,s_{\mathrm f}], \qquad
  s^* = s^*(s_{\mathrm f},z_{\mathrm f}) = \inf\{ \tau \in [s_0,s_{\mathrm f}] \; | \; \eqref{boot} 
  \mbox{ holds in } [\tau,s_{\mathrm f}] \}.
\]
Our goal is to show that we may choose $z_{\mathrm f} = z_{\mathrm f}(s_{\mathrm f})$
so that $s^* = s_0$ (independent of $s_{\mathrm f})$, and Proposition~\ref{uniform} will follow. 

%%%%%%%%%%%%%%%%%%%%%%%%%%%%%%%%%%%%%%%%%%
\subsection{Modulation parameter estimates}
\label{modparest}

Here we estimate the time derivatives of the parameters, establishing this analogue of~\cite[Lemma 12]{Ngu19}:
\begin{lemma} \label{mods}
For $s \in [s^*,s_{\mathrm f}]$ (i.e. assuming~\eqref{boot}), we have the following:
\begin{equation} \label{modlam}
  \left| \frac{\dot \la}{\la} \right| \lec \frac{(C^*)^2}{s^2}
\end{equation}
\begin{equation} \label{modgam}
  | \dot \gat - 1 | \lec \frac{(C^*)\log s}{s^2}
\end{equation}
\begin{equation} \label{modz}
  |\dot z - 2 v| \lec \frac{1}{s \log^{\frac{3}{4}} s}
\end{equation}
\begin{equation} \label{modinner}
  |\lan \eta, \; i T_{z} \ran|  \lec \frac{(C^*)^2}{s \log s}.
\end{equation}
\end{lemma}
\begin{remark} \label{logloss}
Compared to \cite[Lemma 12]{Ngu19}, in addition to modifying the
inner product in~\eqref{modinner}, we lose an extra 
(harmless, it turns out) factor of $\log s$ in~\eqref{modgam} -- this comes from interaction between the delta potential (and cut-off) with the tail of $\La Q$.
\end{remark}

{\it Proof of Lemma~\ref{mods}}:
we start with a second bootstrap assumption,
\begin{equation} \label{boot2}
   |\lan \eta, \; i T_{\Ga,z} \ran|  \lec \frac{C^{**}}{s \log s},
\end{equation}
for some $C^{**}$ to be chosen. That is, we will work on the time interval
\[
  s \in [s^{**},s_{\mathrm f}], \qquad
  s^{**} = \inf\{ \tau \in [s^*,s_{\mathrm f}] \; | \; \eqref{boot2} 
  \mbox{ holds in } [\tau,s_{\mathrm f}] \}.
\]
Since~\eqref{boot2} trivially holds at $s=s_{\mathrm f}$ by~\eqref{ICxi},
and so if $s^{*} < s_{\mathrm f}$, then $s^{**} < s_{\mathrm f}$.

The proofs of~\eqref{modlam},~\eqref{modgam} and~\eqref{modz} are based on the equation~\eqref{etaeq} for $\eta$, which we use to show the following analogue of
\cite[(3.29)]{Ngu19}: for smooth, real-valued functions $A, B$ satisfying either 
\[
\begin{split}
  &(a) \quad |f(y)| + |f'(y)| \lec e^{-|y|} \quad \mbox{ or } \\
  &(b) \quad |f(y)| + |f'(y)| \lec (1 + |y|)e^{-|y|}
\end{split}
\]
where $f = A, B$, we have
\begin{equation} \label{modeq}
\begin{split}
  &\left| \frac{d}{ds} \lan \eta, A + i B \ran - \left[
  \lan \eta, i L_z^{-} A - L^{+}_z B \ran
  - \mv \cdot \lan \Mv Q, i A - B \ran \right] \right| \\
  & \qquad \lec \frac{|\mv|}{s} 
  + \frac{1}{s^2} \left( 
  \left\{ \begin{array}{cc} 
  (C^*)^2 & \mbox{ if } (a) \mbox{ holds } \\
  \log s  & \mbox{ if } (b) \mbox{ holds }
  \end{array} \right \} + C^*|\la-1|\log s \right),
\end{split}
\end{equation}
where $L^+_z$ is defined in~\eqref{L+z},
\[
  L^-_z = L^- + \Ga \de_{-\frac{z}{2}}
  = -\p_y^2 + 1 - Q^{p-1} + \Ga \de_{-\frac{z}{2}},
\]
and the inner-products involving $L_z^{\pm}$
are interpreted in quadratic form sense, e.g.
\[
  \lan \eta, L_z^+ B \ran = \lan \eta, L^+ B \ran
  + \Ga  \re \bar{\eta}\left(-\frac{z}{2}\right) B\left(-\frac{z}{2}\right)
\]
(since $A,B \not\in \D^\Ga_{-\frac{z}{2}}$ in general).
\begin{remark} \label{logloss2}
We will apply estimate~\eqref{modeq} for $A$ or $B$ equal to $Q$ (which satisfies (a)) as well as $yQ$ and $\La Q$ (which satisfy (b)).
We must distinguish between these, since otherwise the log loss in case (b) -- which arises from interaction of $A$ or $B$ with the delta potential
$\de_{-\frac{z}{2}}$ -- would destroy the estimate~\eqref{modlam}. 
\end{remark}
The proof of~\eqref{modeq} is essentially the same as in~\cite{Ngu19}, but with the following modifications:
\begin{itemize}
\item 
the contribution $\la \Ga \de_{-\frac{z}{2}} \eta$ to~\eqref{etaeq}
produces the delta potentials in the $L^{\pm}_z$ operators
appearing in~\eqref{modeq}.
We use the fact~\eqref{domains} $\eta(s,\cdot) \in \D^{\la\Ga}_{-\frac{z}{2}}$ to justify, via~\eqref{IBP}, moving the second derivative from $\eta$ onto $A$ or $B$ in the inner-product.
The contribution from the extra term 
$(\la-1) \Ga \de_{-\frac{z}{2}} \eta$ not included in
the leading order in~\eqref{modeq} is estimated as, e.g.,
\[
\begin{split}
  |\lan (\la-1) \Ga \de_{-\frac{z}{2}} \eta, B \ran| & 
  \lec |\la-1| |B(-z/2)||\eta(-z/2)| \lec |\la-1| z e^{-\frac{z}{2}}
  \frac{C^*}{s} \\ &\lec |\la-1| \frac{C^* \log s}{s^2};
\end{split}
\]
\item the cut-off in $P_1$ produces errors such as
\[
\begin{split}
  |\lan (|P_1|^{p-1} - |(P^0)_1|^{p-1}) \eta, B \ran| &\lec
  |\lan |(1 - \chi)(P^0)|^{p-1}(\cdot + \frac{z}{2}) \eta , B \ran | \\
  & \lec e^{-(p-1)\frac{z}{2}} z e^{-\frac{z}{2}} \| \eta \|_{H^1}
  \lec \frac{C^* \log s}{s^{p+1}};
\end{split}
\]
\item the cut-off in $\cE_P$ produces errors such as
\[
\begin{split}
  |\lan \left( (1-\chi) \cE_{P^0} \right)(\cdot + \frac{z}{2}), B \ran| &\lec \left(|\mv| e^{-\frac{z}{2}} + e^{-z} \right) z e^{-\frac{z}{2}}
  \lec \frac{|\mv| \log s}{s^2} + \frac{\log s}{s^3}
\end{split}
\]
and
\[
\begin{split}
  |\lan \left( \p_y^2 \chi P^0 \right)(\cdot + \frac{z}{2}), B \ran| &\lec e^{-\frac{z}{2}} \; 
  \left\{ \begin{array}{cc} 
  e^{-\frac{z}{2}} & \mbox{ if } (a) \\
  z e^{-\frac{z}{2}} & \mbox{ if } (b)
  \end{array} \right\} \\
  &  \lec \left\{ \begin{array}{cc} 
  \frac{1}{s^2} & \mbox{ if } (a) \\
  \frac{\log s}{s^2} & \mbox{ if } (b) 
  \end{array} \right\} . 
\end{split}
\]
\end{itemize}

We will now apply~\eqref{modeq} for several choices of $A$ and $B$,  using the orthogonality conditions~\eqref{orthos}, as well as the relations
\[
  L^- Q = 0, \quad L^+ \La Q = -2 Q,  \quad L^- yQ = -2Q'
\]
and 
\begin{equation} \label{inners3}
  \lan Q', Q \ran = \lan y Q, \La Q \ran = 0, \;\;
  \lan Q', y Q \ran = -M(Q), \;\;
  \lan \La Q, Q \ran = \frac{5-p}{(p-1)}M(Q) 
\end{equation}
(and $p \not= 5$):
\begin{itemize}
\item
applying~\eqref{modeq} with $A=Q$, $B=0$ 
(so $(a)$ holds) gives
\[
  \left|
  \lan \eta, i \Ga \de_{-\frac{z}{2}} Q \ran
  + \frac{5-p}{2(p-1)} \| Q \|_2^2 \frac{\dot \la}{\la} \right| 
  \lec  \frac{|\mv|}{s} + \frac{(C^*)^2}{s^2}
  \left(1 + |\la-1| \log s \right),
\]
and so using~\eqref{boot} and~\eqref{consequence},
\begin{equation} \label{m1}
\begin{split}
  \left| \frac{\dot \la}{\la} \right| &\lec
  e^{-\frac{z}{2}} \| \eta \|_{H^1} + 
  \frac{|\mv|}{s} + \frac{(C^*)^2}{s^2}
  \left(1 + |\la-1| \log s \right) \\
  &\lec  \frac{(C^*)^2}{s^2}
  \left(1 + |\la-1| \log s \right) + \frac{|\mv|}{s};
\end{split}
\end{equation}
\item
applying~\eqref{modeq} with $A=0$, $B = \La Q$ 
(so that (b) holds) gives
\[
\begin{split}
  &\left|
  \lan \eta, i \Ga \de_{-\frac{z}{2}} \La Q \ran
  - \frac{5-p}{2(p-1)} \| Q \|_2^2 (\dot \gat - 1 + \frac{v^2}{4} -
  \frac{\dot \la}{\la} \frac{vz}{4} - \frac{v \dot z}{4}) \right| \\
  & \qquad \lec \frac{|\mv|}{s} + \frac{\log s}{s^2}
  \left( 1 + C^*|\la-1| \right),
\end{split}
\]
and so using~\eqref{boot} and~\eqref{consequence},
\begin{equation} \label{m3}
\begin{split}
  \left| \dot \gat - 1 + \frac{v^2}{4} -
  \frac{\dot \la}{\la} \frac{vz}{4} - \frac{v \dot z}{4}  \right| &\lec
  z e^{-\frac{z}{2}} \| \eta \|_{H^1} + 
   \frac{|\mv|}{s} + \frac{\log s}{s^2}
  \left( 1 + C^*|\la-1| \right) \\
  &\lec  \frac{(C^*) \log s}{s^2}
  \left(1 + |\la-1|\right) + \frac{|\mv|}{s};
\end{split}
\end{equation}
\item
applying~\eqref{modeq} with $A=yQ$, $B = 0$ 
(so (b) holds) gives
\[
  \left|
  2\lan \eta, i Q' \ran + \lan \eta, i \Ga \de_{-\frac{z}{2}} y Q \ran
  - \frac{M(Q)}{2}(\dot z - 2v + \frac{\dot \la}{\la}) \right| 
  \lec \frac{|\mv|}{s} + \frac{\log s}{s^2}
  \left( 1 + C^*|\la-1| \right),
\]
and so, using~\eqref{boot2}, \eqref{boot}, \eqref{consequence},
and~\eqref{esefest},
\begin{equation} \label{m2}
\begin{split}
  |\dot z - 2v + \frac{\dot \la}{\la}| &\lec
  |\lan \eta, i T^{\Ga,z} \ran| + \| \eta \|_{2} \| T^{\Ga,z} - Q' \|_2
  + \| \eta \|_{H^1} z e^{-\frac{z}{2}} \\
  & \qquad +   \frac{|\mv|}{s} + \frac{\log s}{s^2}
  \left( 1 + C^*|\la-1| \right) \\
  & \lec \frac{C^{**}}{s \log s}
  + \frac{C^*}{s} e^{-\frac{z}{2}} + \frac{C^*}{s} \frac{\log s}{s}
  +  \frac{|\mv|}{s} + \frac{\log s}{s^2}
  \left( 1 + C^*|\la-1| \right) \\
  & \lec \frac{C^{**}}{s \log s} + \frac{|\mv|}{s}.
\end{split}
\end{equation}
\end{itemize}

By~\eqref{vdyn2}, \eqref{m1}, and~\eqref{consequence}
\begin{equation} \label{m4}
  \left|\dot v - \frac{\dot \la}{\la} v \right| \lec e^{-z} + | \frac{\dot \la}{\la} |
  \lec \frac{(C^*)^2}{s^2}
  \left(1 + |\la-1| \log s \right) + \frac{|\mv|}{s}.
\end{equation}
Combining~\eqref{m1}-\eqref{m4} yields
\[
  |\mv| \lec \frac{C^{**}}{s \log s} + \frac{|\mv|}{s},
\]
and so (choosing $s_0$ large enough)
\begin{equation} \label{mest}
  |\mv|  \lec \frac{C^{**}}{s \log s}.
\end{equation}
Using~\eqref{mest} in~\eqref{m1} gives
\begin{equation} \label{m1a}
  \left| \frac{\dot \la}{\la} \right| \lec 
  \frac{(C^*)^2}{s^2}
  \left(1 + |\la-1| \log s \right).
\end{equation}
Setting $\Lambda(s) = \sup_{s \leq \tau \leq s_{\mathrm f}} |\la(\tau)-1|$, and using~\eqref{ICs}, we have 
\[
\begin{split}
   |\la(s)-1| &\lec |\log \la(s)| = \left| \int_s^{s_{\mathrm f}} \frac{\dot \la(\tau)}{\la(\tau)} d\tau \right|
 \lec (C^*)^2 \int_s^{s_{\mathrm f}} \frac{d \tau}{\tau^{2}}
 \left( 1 + \Lambda(s) \log \tau \right) \\
  &\lec (C^*)^2 \left(\frac{1}{s} + \frac{\log s}{s} \La(s) \right)
\end{split}
\]
and so (choosing $s_0$ large enough),
\begin{equation} \label{lamest}
  |\la(s) - 1| \lec \frac{(C^*)^2}{s}.
\end{equation}
Inserting this back into~\eqref{m1a} establishes~\eqref{modlam}.

Using~\eqref{mest} and~\eqref{lamest} in~\eqref{m3} shows
\[
  \left| \dot \gat - 1 + \frac{v^2}{4} -
  \frac{\dot \la}{\la} \frac{vz}{4} - \frac{v \dot z}{4} 
  \right| \lec \frac{(C^*) \log s}{s^2}, 
\]
which, together with~\eqref{consequence} and~\eqref{mest}, shows~\eqref{modgam}.

The estimate~\eqref{modz} follows 
from~\eqref{mest} and~\eqref{modlam} as
\[
\begin{split}
  |\dot z - 2v| &\lec |\dot z -2v + \frac{\dot \la}{\la} z| + |\frac{\dot \la}{\la}| z \lec |\mv| + \left| \frac{\dot \la}{\la} \right| 
  \log s \\
  & \lec \frac{C^{**}}{s \log s} +  \frac{(C^*)^2 \log s}{s^2}
  \lec \frac{1}{s \log^{\frac{3}{4}} s}.
\end{split}
\]

At this point, setting  
\[
  \vec{m}^* = 
  \left[ \begin{array}{c}
  m_1 \\ m_3 \\ m_4
  \end{array} \right] =
  \left[ \begin{array}{c} 
  \frac{\dot \la}{\la} \\
  \dot \gat - 1 + \frac{v^2}{4} - \frac{\dot \la}{\la} \frac{v}{2} \frac{z}{2} 
  - \frac{v}{2} \frac{\dot z}{2} \\ \frac{\dot v}{2} - \frac{\dot \la}{\la} \frac{v}{2} \end{array} \right]
\]
($\vec{m}$ with its second component removed),
as an immediate consequence of the first three estimates of
Lemma~\ref{mods}, together with~\eqref{consequence} we have
\begin{equation} \label{mod0}
  |\vec{m}| \lec \frac{1}{s \log^{\frac{3}{4}} s}, \qquad \quad  |\vec{m}^*| \lec \frac{1}{s^{\frac{3}{2}}}.
\end{equation}
\begin{remark}
In fact, we have the stronger estimate
\begin{equation} \label{mod0a}
  |\vec{m}^*| \lec \frac{(C^*)^2 \log s}{s^2},
\end{equation}
which accords with~\cite{Ngu19} (modulo the 
additional log factor).
We will use only the weaker estimate of~\eqref{mod0}
in the arguments below, however, since (as we shall see)
it is all we have available in the supercitical 
($p > 5$) case, and we wish to use the same arguments there. 
\end{remark}

%%%%%%%%%%%%%%%%%%%%%%%%%%%%%%%%%%%%%%%%%%
\subsection{Localized momentum}
\label{momentum}

To complete the proof of Lemma~\ref{mods},
it remains to show the estimate~\eqref{modinner}; 
that is, to improve the bootstrap estimate~\eqref{boot2}. 
For this, following~\cite{Ngu19}, we introduce the localized momentum functional
\begin{equation} \label{locmom}
  \M(s) = \im \int \bar \eta(s,y) \p_y \eta(s,y)
  \chit_s(y) dy, \qquad \chit_s(y) = \chit \left(
  \frac{|y|}{\log s} \right) dy,
\end{equation}
where $\chit : [0,\infty) \to [0,\infty)$ is smooth, 
non-increasing, and satisfies
\[
  \chit \equiv 1 \mbox{ on } [0,\frac{1}{10}], \quad
  \chit \equiv 0 \mbox{ on } [\frac{1}{8},\infty).
\]
Then by~\eqref{consequence},
\begin{equation} \label{support}
  y \in \supp \tilde\chi_s \; \implies \; 
  |y - \frac{z}{2}| \geq \frac{z}{2} - \frac{1}{8} \log s
  \geq \frac{3}{4} \log s.
\end{equation}

We claim, just as in~\cite[(3.38)]{Ngu19}, the estimate
\begin{equation} \label{mom}
  \left| \frac{d}{ds} \M(s) - \left\lan
  \bar \eta \cdot N''(P_1) \cdot \eta, Q'
  \right\ran - (\dot z - v + \frac{\dot \la}{\la}z)
  \lan i Q', \p_y \eta \ran  \right|
  \lec \frac{(C^*)^2}{s^2 \log s},
\end{equation}
where $\frac{1}{2} \bar \eta \cdot N''(P_1) \cdot \eta$
denotes the terms quadratic in $\eta$ in the Taylor expansion of
$N(P_1 + \eta) = |P_1 + \eta|^{p-1}(P_1+\eta)$
(see~\eqref{taylor} for notation).
The proof of~\eqref{mom} begins with
\[
  \frac{d}{ds} \M(s) = \im \int \bar \eta(s,y) \p_y \eta(s,y)
  [\p_s \chit_s(y)] dy + \lan i \p_s \eta, \;
  2 \chit_s \p_y \eta + \eta \p_y \chit_s \ran
\]
and uses the equation~\eqref{etaeq} in the second term.
We omit the argument here as it proceeds exactly as in~\cite{Ngu19} -- since by~\eqref{support}, the contributions in~\eqref{etaeq} from 
the delta potential and the cut-off $\chi$ do not enter the integral.
We remark only that the weaker
estimate of~\eqref{mod0} (rather than~\eqref{mod0a})
easily suffices.

Next we compute the time derivative of 
$\lan \eta, i T_{z} \ran$,
in a way that refines~\eqref{modeq} by also keeping track of terms
of size $\frac{1}{s^2}$ (which were previously
treated as error terms). Start with
\begin{equation} \label{refine}
  \frac{d}{ds} \lan \eta, i T_{z} \ran =
  \lan -i \p_s \eta, T_{z} \ran 
  + \lan \eta, i \p_s T_{z} \ran.
\end{equation}
For the second term:
\[
  |\lan \eta, i \p_s T_{z} \ran|
  \leq \| \eta \|_2 \| \p_z T_{z} \|_2 |\dot z|
  \lec \frac{C^*}{s} e^{-\frac{z}{2}} (|\dot z-2v| + |v|)
  \lec \frac{C^*}{s^3}
\]
using~\eqref{consequence}, \eqref{efz} and~\eqref{modz}.
For the first term, we use~\eqref{etaeq} to find
\begin{equation} \label{refine2}
\begin{split}
  \lan -i \p_s \eta, T_{z} \ran &= 
  -\lan \eta, L^+_{z} T_z \ran \\
  & \quad + \lan |P_1+\eta|^{p-1}(P_1+\eta) - |P_1|^{p-1} P_1
  - \frac{p+1}{2} Q^{p-1} \eta - \frac{p-1}{2} Q^{p-1} \bar \eta, T_{z} \ran \\
  & \quad + \mv \cdot \lan \Mv \eta, T_{z} \ran
  + \lan \cE_{P_1}, T_{z} \ran
  + (\la-1) \Ga \re \bar{\eta}\left(-\frac{z}{2}\right) 
  T_z\left(-\frac{z}{2}\right),
\end{split}
\end{equation}
where the first and last terms arise from integration by parts
as follows:
\[
\begin{split}
  \lan (\p_y^2 - \la \Ga \de_{-\frac{z}{2}} ) \eta, T_z \ran
  &= -q_{-\frac{z}{2}}^{\la \Ga}(\eta, T_z) \\
  &= -q_{-\frac{z}{2}}^{\Ga}(\eta, T_z) +  (1-\la) \Ga \re  \bar{\eta}\left(-\frac{z}{2}\right) 
  T_z\left(-\frac{z}{2}\right) \\
  &=  \lan \eta, (\p_y^2 - \Ga \de_{-\frac{z}{2}}) T_z \ran
   +  (\la-1) \Ga \re \bar{\eta}\left(-\frac{z}{2}\right) 
  T_z\left(-\frac{z}{2}\right) 
\end{split}
\]
since $\eta(s,\cdot), \; T_z \in \D^\Ga_{-\frac{z}{2}}$.
For the first term in~\eqref{refine2}, we have
\[
  |\lan \eta, L^+_{z} T_{z} \ran| = 
  |\lan \eta, \nu_{z} T_{z} \ran|
  \lec \| \eta \|_2 |\nu_{z}| \| T_{z} \|_2
  \lec \frac{C^*}{s} e^{-z} e^{-\frac{z}{2}}
  \lec \frac{C^*}{s^4}
\]
using~\eqref{consequence}, \eqref{nutight} and~\eqref{esefest}.
\begin{remark} \label{replacement}
This is precisely where the need to replace $Q'$ with 
the eigenfunction $T_{z}$ arises. Otherwise, the delta potential 
produces the term $\Ga \eta\left(-z/2\right) Q'\left(-z/2\right)$,
which is the same size, $O(1/s^2)$, as the leading terms.
\end{remark}
For the second term in~\eqref{refine2}:
using the Taylor estimates~\eqref{tay1} and~\eqref{tay2},
\[
\begin{split}
  &\left| \lan |P_1+\eta|^{p-1}(P_1+\eta) - |P_1|^{p-1} P_1
  - \frac{p+1}{2} Q^{p-1} \eta - \frac{p-1}{2} Q^{p-1} \bar \eta, T_{z} \ran \right. \\
  & \left. \qquad  - \lan \frac{1}{2} \bar \eta \cdot N''(P_1) \eta, 
  Q' \ran \right| \lec
  \| \eta \|_{H^1}^2 \|T_{z} - Q'\|_{H^1}
  + \| \eta \|_{H^1}^3 + \| \eta \|_{H^1}^p \\
  & \qquad \qquad \qquad \qquad + \int
  ( ||P_1|^{p-1}-Q^{p-1}| + ||P_1|^{p-3}P_1^2 - Q^{p-1}|) |\eta|
  |T_{z}| dy \\
  & \qquad \lec \frac{(C^*)^2}{s^2} e^{-\frac{z}{2}} + \left( \frac{C^*}{s} \right)^{\min(3,p)} +
  \| \eta \|_\infty \int e^{-|y+z|}(e^{-(p-2)|y+z|} + 
  e^{-(p-2)|y|} ) |T_{z}(y) | dy \\
  & \qquad \lec  \left( \frac{C^*}{s} \right)^{\min(3,p)} + \frac{C^*}{s} e^{-z}
  \lec   \left( \frac{C^*}{s} \right)^{\min(3,p)}.
\end{split}
\]
using~\eqref{boot}, \eqref{esefest}, \eqref{consequence}
and~\eqref{pointwise}.
For the third term in~\eqref{refine2}:
\[
  | \mv \cdot \lan \Mv \eta, T_{z} \ran -
   \mv \cdot \lan \Mv \eta, Q' \ran| \lec
  |\mv| \| \eta \|_2 \| T_{z} - Q' \|_2 \lec 
  \frac{1}{s \log^{\frac{3}{4}} s} \frac{C^*}{s} e^{-\frac{z}{2}}
  \lec \frac{1}{s^3},
\]
and
\[
  |\mv \cdot \lan \Mv \eta, Q' \ran - \frac{1}{2}(\dot z - v + \frac{\dot \la}{\la}z) \lan i Q', \p_y \eta \ran |
  \lec |\mv^*| \| \eta \|_2 \lec \frac{C^*}{s^{\frac{5}{2}}},
\]
using~\eqref{boot}, \eqref{consequence}, 
\eqref{mod0}, and~\eqref{esefest}.
For the fourth term in~\eqref{refine2}, we use Proposition~\ref{Einnerest2}
and~\eqref{vdyn2}, together with~\eqref{consequence} and~\eqref{mod0} to obtain
\[
  |\lan \cE_{P_1}, T_{z} \ran| \lec \frac{1}{s^2} \left(
  \frac{\log^{\frac{5}{4}} s}{s} + \frac{\log^2 s}{s^2} 
  + \frac{1}{s} \right) \lec \frac{\log^{\frac{5}{4}} s}{s^3}.
\]
\begin{remark} \label{modification}
It is precisely here that we use the modification~\eqref{vdyn2} 
to the motion law induced by the delta potential.
\end{remark}
Finally, for the fifth term in~\eqref{refine2}:
\[
  \left| (\la-1) \Ga \re \bar{\eta}\left(-\frac{z}{2}\right)
  T_z\left(-\frac{z}{2}\right) \right|
  \lec |\la-1| \| \eta \|_{H^1} \left| T_z\left(-\frac{z}{2}\right) \right|
  \lec \frac{C^*}{s} \frac{C^*}{s} e^{-\frac{z}{2}}
  \lec  \frac{(C^*)^2}{s^3}
\]
using~\eqref{lamest}, \eqref{pointwise}, \eqref{boot} and~\eqref{consequence}. Using all the above estimates in~\eqref{refine},
we find
\[
  \left| \frac{d}{ds} \lan \eta, i T_z \ran 
  - \lan \frac{1}{2} \bar \eta \cdot N''(P_1) \eta, 
  Q' \ran  - \frac{1}{2}(\dot z - v + \frac{\dot \la}{\la}z) \lan i Q', \p_y \eta \ran \right| \lec 
   \left( \frac{C^*}{s} \right)^{\min(\frac{5}{2},p)}.
\]
Combining this with~\eqref{mom} gives
\[
  \left| \frac{d}{ds} \left[ \frac{1}{2} \M(s) - \lan \eta, i T_{z} \ran
  \right] \right| \lec 
  \left( \frac{C^*}{s} \right)^{\min(\frac{5}{2},p)}
  + \frac{(C^*)^2}{s^2 \log s} \lec \frac{(C^*)^2}{s^2 \log s}.
\]
Both $\M(s)$ and $\lan \eta, i T^{\Ga,z} \ran$
vanish at $s=s_{\mathrm f}$ by~\eqref{ICxi}, so by time integration,
\[
  \left| \frac{1}{2} \M(s) - \lan \eta, i T_{z} \ran
  \right| \lec \frac{(C^*)^2}{s \log s}, 
\]
hence, using~\eqref{boot} again,
\[
  |\lan \eta, i T_{z} \ran| \lec |\M(s)| +  \frac{(C^*)^2}{s \log s}
  \lec \| \xi \|_{H^1}^2 + \frac{(C^*)^2}{s \log s} \lec  \frac{(C^*)^2}{s \log s}.
\]
By choosing $C^{**} \gec (C^*)^2$, we improve the bootstrap estimate~\eqref{boot2},
showing that $s^{**} = s^*$, and so establishing~\eqref{modinner}.
This completes the proof of Lemma~\ref{mods}.
$\Box$

%%%%%%%%%%%%%%%%%%%%%%%%%%%%%%%%%%%%%%%%%%
\subsection{Energy estimate}
\label{enest}

To improve the second estimate in the bootstrap assumption~\eqref{boot},
we use a coercive, almost-conserved functional built from the
localized momentum and a modified linearized energy functional
\begin{equation} \label{en}
\begin{split}
  H(s,\xi) &= \frac{1}{2} \int \left[
  |\p_y \xi|^2 + |\xi|^2 + \frac{2}{p+1} \left(
  -|P+\xi|^{p+1} + |P|^{p+1} + (p+1)|P|^{p-1} \re  (\bar P \xi)  \right) \right] dy
  \\ & \qquad + \frac{1}{2} \la \Ga |\xi(s,0)|^2 
  - \lan \xi, \cE_P^{cut} \ran,
\end{split}
\end{equation}  
with
\[
   \cE_P^{cut} := 2 \p_y \chi \p_y P^0 + P^0 \p_y^2 \chi.
\]
\begin{remark} \label{endef}
This functional is the one used in~\cite[Sec 3.2.3]{Ngu19} 
with the addition of the two terms at the end -- the first is 
simply to incorporate the delta potential, while the second
is needed to handle an effect of the cut-off function
in the approximate solution -- see Remark~\ref{cutpoint} below.
\end{remark}
We will use
\begin{equation} \label{cut}
  |\cE_P^{cut}| \lec e^{-\frac{z}{2}} \mathbbm{1}_{\supp (\p_y \chi)}, \qquad 
  \| \cE_P^{cut} \|_{W^{1,\infty}} \lec e^{-\frac{z}{2}} \lec \frac{1}{s},
\end{equation}
from~\eqref{asy} and~\eqref{consequence}. 

We combine $H$ with the localized momentum functional
\[
  \J(s,\xi) = \frac{v}{2} \left( \M_1 - \M_2 \right),
\]
where
\[
  \M_1 = \M \mbox{ (defined in~\eqref{locmom}) }, \quad \M_2 =   
  \im \int \bar \eta_2 \p_y \eta_2 \chit_s dy, \quad
  \eta_2(s,y) = e^{i \frac{v}{2} y} \xi(s,y-\frac{z}{2}).
\]
\begin{remark} \label{M2}
Since $P$ is even (for any $z$ and $v$) and $u(s,\cdot)$ is even,
so is $\xi(s, \cdot)$. It follows that $\eta_2(y) = \eta_1(-y)$,
and so $\M_2 = -\M_1$, and $\J = v \M$.
\end{remark}
Defining
\[
  \W(s,\xi) = H(s,\xi) - \J(s,\xi),
\]
we have the following analogue of~\cite[Prop 13]{Ngu19}
giving the almost-conservation of $\W$:
\begin{lemma} \label{almost}
We have
\begin{equation} \label{almosteq}
  \left| \frac{d}{ds} \W(s,\xi) \right| \lec  \frac{C^*}{s^3}
\end{equation}
and
\begin{equation} \label{Wbound}
  |\W(s,\xi)| \lec \frac{C^*}{s^2}.
\end{equation}.
\end{lemma}
\begin{proof}
Since~\eqref{Wbound} follows from time integration of~\eqref{almosteq} and~\eqref{ICxi}, it remains to show~\eqref{almosteq}. 
Compute
\begin{equation} \label{endot}
  \frac{d}{ds} H(s,\xi) = \p_s H(s,\xi) + \lan \p_\xi H(s,\xi), \p_s \xi \ran,
\end{equation}
where
\begin{equation} \label{dH}
  \p_\xi H(s,\xi) = \left( -\p_y^2 + 1 + \la \Ga \de \right) \xi
  -|P + \xi|^{p-1}(P + \xi) + |P|^{p-1} P - \cE_P^{cut}.
\end{equation}

For the first term, we have
\[
  \p_s H(s,\xi) = \left\lan \dot z \p_z P + \dot v \p_v P, 
  -\frac{1}{2}\bar \xi \cdot N''(P) \cdot \xi
  + O(|\xi|^3 + |\xi|^{p+1}) \right \ran + \dot \la \Ga |\xi(s,0)|^2
  - \lan \xi, \p_s \cE_P^{cut} \ran
\]
and
\[
  \p_z P = \frac{1}{2} \chi \left[ 
  \left(e^{-i\frac{v}{2} y} Q' \right)(y + \frac{z}{2})
  - \left(e^{i\frac{v}{2} y} Q' \right)(y - \frac{z}{2}) \right] 
  - i \frac{v}{4} P, 
\]
so
\begin{equation} \label{dsen}
\begin{split}
  &\left| \p_s H(s,\xi) -  \frac{\dot z}{2} \left\lan
  \left(e^{i\frac{v}{2} y} Q' \right)(y - \frac{z}{2})
  -\left(e^{-i\frac{v}{2} y} Q' \right)(y + \frac{z}{2}),
  \; \frac{1}{2}\bar \xi \cdot N''(P) \cdot \xi \right\ran \right| 
  \\ & \qquad \lec (|\dot z| e^{-\frac{z}{2}} +
  |\dot v| + v |\dot z|) \| \xi \|_{H^1}^2 + 
  (|\dot z| + |\dot v|) (\| \xi \|_{H^1}^3 + \| \xi \|_{H^1}^{p+1}) \\
  & \qquad \qquad + \left| \frac{\dot \la}{\la} \right| \| \xi \|_{H^1}^2 
  + (|\dot z| + |\dot v|)e^{-\frac{z}{2}} \| \xi \|_{H^1}
  \lec \frac{1}{s^2} \| \xi \|_{H^1}
\end{split}
\end{equation}
using~\eqref{boot}, \eqref{consequence} and~\eqref{mod0}.

For the second term in~\eqref{endot}, we use the equation for $\xi$:
\begin{equation}  \label{xieq}
\begin{split}
  i \p_s \xi &= (-\p_y^2 + \la \Ga \delta + 1 ) \xi
  - |P + \xi|^{p-1}(P + \xi) + |P|^{p-1} P 
  \\ & \qquad \qquad 
  +i \frac{\dot \la}{\la} \La \xi + (\dot \gat-1) \xi - \cE_{P} \\
  &= \p_\xi H(s,\xi)  + i \frac{\dot \la}{\la} \La \xi 
  + (\dot \gat - 1) \xi - \left[ \cE_{P} - \cE_P^{cut} \right],
\end{split}
\end{equation}
by~\eqref{dH}. So
\begin{equation} \label{dxiH}
  \lan \p_\xi H(s,\xi), \p_s \xi \ran =
  \lan \p_\xi H(s,\xi),  \frac{\dot \la}{\la} \La \xi 
  - i(\dot \gat - 1) \xi + i\left[ \cE_{P} - \cE_P^{cut} \right] \ran.
\end{equation}
\begin{remark} \label{cutoffrem}
By~\eqref{domains}, we have 
$\xi \in C_s \D^{\la \Ga} \cap C^1_s L^2$, and so
$\p_s \xi, \p_\xi H(s,\xi) \in C_s L^2$, justifying this derivation. Note that this justification rests on having introduced the cut-off $\chi$ in the approximate solution $P$: not only does this ensure
the regularity used here, but without it the error
term $\cE_P$ would also introduce a term with 
$\delta$ into the right side of this inner product.
\end{remark}
Now by~\eqref{boot},~\eqref{mod0} and~\eqref{cut},
\begin{equation} \label{lamgam}
\begin{split}
  \left| \frac{\dot \la}{\la}  \lan \p_\xi H(s,\xi), \La \xi \ran \right|
  +  \left| (\dot \gat - 1)  \lan \p_\xi H(s,\xi), \xi \ran \right|
  &\lec |\mv^*| 
  \left( \| \xi \|_{H^1}^2 + \| \xi \|_{H^1}^{p+1} + \frac{1}{s} \| \xi \|_{H^1} \right) \\
  &\lec \frac{C^*}{s^{\frac{5}{2}}} \| \xi \|_{H^1} \lec 
  \frac{1}{s^2}  \| \xi \|_{H^1}.
\end{split}
\end{equation}
To handle the last term in~\eqref{dxiH}, we first recall the 
expressions~\eqref{error3} and~\eqref{error} for $\cE_P$. 
For the contribution
\[
  \left\lan \p_\xi H(s,\xi), -i \mv \cdot \left( e^{i \frac{v}{2}(\cdot)} \Mv Q \right)(\cdot-\frac{z}{2}) \right\ran =
  \mv \cdot \left\lan e^{-i\frac{v}{2} (\cdot)} \p_\xi H(s,\xi)(\cdot + \frac{z}{2}), -i\Mv Q \right\ran
\]
we expand
\[
  e^{-i \frac{v}{2} y} \left[ \p_\xi H(s,\xi) + \cE^{cut}_P \right](y + \frac{z}{2}) = 
  L^+_{z} \re  \eta + i L^-_{z} \im \eta 
  + O(|v||\p_y \eta| + v^2 |\eta| + |\eta|^2 + |\eta|^p),
\]
so using
\[
  |\mv \cdot \Mv Q(y)| \lec |\mv^*| (1 + |y|)e^{-|y|} + |\mv| e^{-|y|},
\]
along with~\eqref{mod0},~\eqref{consequence} and~\eqref{cut}:
\[
\begin{split}
  &\left| \left\lan \p_\xi H(s,\xi), 
  -i \mv\cdot\left( e^{i \frac{v}{2}(\cdot)} \Mv Q \right)(\cdot-\frac{z}{2}) \right\ran 
  + \mv \cdot \lan L^+_{z} \re \eta + i L^-_{z} \im  \eta , i \Mv Q \ran \right|
  \\ & \qquad \qquad 
  \lec |\mv| \left[ \left( |v| + \| \xi \|_{H^1} + \| \xi \|_{H^1}^{p-1} 
  \right) \| \xi \|_{H^1} + \frac{1}{s^2} \right]
  +  |\mv^*| \frac{\log s}{s^2} \lec 
  \frac{1}{s^3}.
\end{split}  
\]
Now using also~\eqref{orthos},
\[
\begin{split}
  & \left\lan \p_\xi H(s,\xi), -i \mv \cdot \left( e^{i \frac{v}{2}(\cdot)} \Mv Q \right)(\cdot-\frac{z}{2}) \right\ran =
  m_1 \lan L^+_z  \re \eta, \La Q \ran
  + m_2 \lan L^+_z \re  \eta, Q' \ran \\ & \qquad \qquad - m_3 \lan L^-_z \im  \eta, Q \ran
  - m_4 \lan L^-_z \im  \eta, yQ \ran
  + O \left( \frac{1}{s^3} \right) \\
  & \quad =  m_1 \lan \re \eta, -2Q -\Ga \de_{-\frac{z}{2}} \La Q \ran
  + m_2 \lan L^+_z \re \eta, 
  Q' \ran - m_3 \lan \im \eta,  -\Ga \de_{-\frac{z}{2}} Q \ran \\ & \qquad \qquad 
  - m_4 \lan \im \eta, -2Q' -\Ga \de_{-\frac{z}{2}} yQ \ran
  + O \left( \frac{1}{s^3} \right) \\
  & \quad =  -\Ga m_1 \re  \eta(-\frac{z}{2}) (\La Q)(-\frac{z}{2})
  + m_2 \lan  L_+^z \re \eta, 
   Q' \ran + \Ga m_3 \im \eta(-\frac{z}{2}) Q(-\frac{z}{2}) \\ & \qquad 
  + m_4 \left( 2\lan \im \eta, Q' \ran + \Ga \im \eta(-\frac{z}{2}) (yQ)(-\frac{z}{2})  \right)
  + O \left( \frac{1}{s^3} \right) \\
  & \quad = m_2 \lan 
  L^+_z \re  \eta, Q' \ran + O\left( \left(  |\mv^*| \frac{\log s}{s} + 
  \frac{1}{s^2} \right) \| \xi \|_{H^1}  + \frac{1}{s^3} \right),
\end{split}  
\]
and since 
\[
\begin{split}
  |\lan L^+_z \re  \eta, Q' \ran| &\leq |\lan \re \eta, L^+_z T_z \ran|
  + |\lan L^+_z \re  \eta, T_z - Q' \ran| \lec  |\nu_z| |\lan \re \eta, T_z \ran|
  + \| \eta \|_{H^1} \| T_z - Q' \|_{H^1} \\
  & \lec \left( e^{-z} + e^{-\frac{z}{2}} \right) \| \xi \|_{H^1} 
  \lec \frac{1}{s} \| \xi \|_{H^1},
\end{split}
\]
we have
\[
  \left| \left\lan \p_\xi H(s,\xi), 
   -i \mv \cdot \left( e^{i \frac{v}{2}(\cdot)} \Mv Q \right)(\cdot-\frac{z}{2}) \right\ran \right| \lec
  \left( \frac{|\mv|}{s} + \frac{1}{s^2} \right) \| \xi \|_{H^1}
  + \frac{1}{s^3} \lec \frac{1}{s^2} \| \xi \|_{H^1} + \frac{1}{s^3}.
\]
Then by symmetry, we also have
\[
  \left| \left\lan \p_\xi H(s,\xi), i \Om\mv \cdot \left( e^{-i \frac{v}{2}} \Mv Q \right)(\cdot+\frac{z}{2}) \right\ran \right| 
  \lec \frac{1}{s^2} \| \xi \|_{H^1} + \frac{1}{s^3}. 
\]
Moreover,
\[
\begin{split}
  &\left| \left\lan \p_\xi H(s,\xi), i(1 - \chi) \left[
  - \mv \cdot \left( e^{i \frac{v}{2}(\cdot)} \Mv Q \right)(\cdot-\frac{z}{2}) + \Om\mv \cdot \left( e^{-i \frac{v}{2}} \Mv Q \right)(\cdot+\frac{z}{2}) \right] \right\ran \right| \\
  & \qquad \qquad \lec (\| \xi \|_{H^1} + \| \xi \|_{H^1}^p  + \frac{1}{s})
  \left(|\mv^*| z e^{-\frac{z}{2}} + |\mv| e^{-\frac{z}{2}}\right) \\
  & \qquad \qquad \lec \left( \frac{1}{s^{\frac{3}{2}}} \frac{\log s}{s} +
  \frac{1}{s \log^{\frac{3}{4}} s} \frac{1}{s} \right) 
  \left(\| \xi \|_{H^1} + \frac{1}{s} \right)
  \lec \frac{1}{s^2} \| \xi \|_{H^1} + \frac{1}{s^3}.
\end{split}
\]
Next, by~\eqref{intersize} and~\eqref{consequence},
\[
  \left| \left\lan \p_\xi H(s,\xi), i\chi G
  \right \ran \right| \lec 
  \left( \| \xi \|_{H^1} + \| \xi \|_{H^1}^p  + \frac{1}{s} \right) 
  \frac{1}{s^2} \lec \frac{1}{s^2} \| \xi \|_{H^1} + \frac{1}{s^3}.
\]
Finally, we have
\[
\begin{split}
  &\left| \left\lan \p_\xi H(s,\xi), i\chi_\rho (\chi_\rho^{p-1}-1) 
  |P^0|^{p-1} P^0 \right \ran \right| \lec
  \left( \| \xi \|_{H^1} + \| \xi \|_{H^1}^p  + \frac{1}{s} \right) 
  e^{-p \frac{z}{2}} \\
  & \qquad \qquad \lec  \left( \| \xi \|_{H^1} + \frac{1}{s} \right) \frac{1}{s^p}
  \lec \frac{1}{s^2} \| \xi \|_{H^1} + \frac{1}{s^3},
\end{split}
\]
and
\[
\begin{split}
  \left| \left\lan \p_\xi H(s,\xi), \frac{\dot \la}{\la} 
  (y \p_y \chi_\rho) P^0 \right\ran \right| &\lec
  |\mv^*| \left(\| \xi\|_{H^1} + \| \xi \|_{H^1}^p + \frac{1}{s} \right)
  e^{-\frac{z}{2}} \\ &\lec \frac{1}{s^{\frac{5}{2}}} 
  \left(\| \xi \|_{H^1} + \frac{1}{s} \right)
  \lec  \frac{1}{s^2} \| \xi \|_{H^1} + \frac{1}{s^3}.
\end{split}
\]
Combining this last series of estimates shows
\begin{equation} \label{nocutest}
  | \lan \p_\xi H(s,\xi), i \left[ \cE_P - \cE_P^{cut} \right] \ran | 
  \lec   \frac{1}{s^2} \| \xi \|_{H^1} + \frac{1}{s^3}.
\end{equation}
\begin{remark} \label{cutpoint}
We had to remove $\cE_P^{cut}$ from $\cE_P$ exactly to obtain
this estimate. The price we pay is the extra term added to the 
energy~\eqref{en}, but this will not destroy
our desired coercivity property, by a simple Young's inequality argument -- see below.
\end{remark}

Putting~\eqref{lamgam} and~\eqref{nocutest} into~\eqref{dxiH} produces
\[
  | \lan \p_\xi H(s,\xi), \p_s \xi \ran |
  \lec \frac{1}{s^2} \| \xi \|_{H^1} + \frac{1}{s^3},
\]
and together with~\eqref{dsen} gives
\begin{equation} \label{encomp}
\begin{split}
  &\left| \frac{d}{ds} H(s,\xi) -  \frac{\dot z}{2} 
  \left\lan
  \left(e^{i\frac{v}{2} y} Q' \right)(y - \frac{z}{2})
  -\left(e^{-i\frac{v}{2} y} Q' \right)(y + \frac{z}{2}),
  \; \frac{1}{2}\bar \xi \cdot N''(P) \cdot \xi \right\ran
  \right| \\ & \qquad \qquad \lec \frac{1}{s^2} \| \xi \|_{H^1} + \frac{1}{s^3}.
\end{split}
\end{equation}

Compute
\[
  \frac{d}{ds} \J(s,\xi) = \frac{1}{2} \dot v \left( \M_1 - \M_2 \right)
  + \frac{v}{2} \frac{d}{ds} \left( \M_1 - \M_2 \right).
\]
Now,
\[
  \left| \frac{1}{2} \dot v \left( \M_1 - \M_2 \right)
  \right| \lec |\dot v| \| \xi \|_{H^1}^2 \lec \frac{C^*}{s^3} \| \xi \|_{H^1}.
\]
Moreover, by~\eqref{mom}, and an analogous calculation for $\M_2$
(or just Remark~\ref{M2}),
\[
\begin{split}
  & \left| \frac{d}{ds} \left( \M_1 - \M_2 \right) - 
  \left\lan \left(e^{i\frac{v}{2}y} Q'\right)(\cdot - \frac{z}{2})
  - e^{-i\frac{v}{2}(y+\frac{z}{2})} Q'(\cdot + \frac{z}{2}), \; \bar \xi \cdot N''(P) \cdot \xi   \right\ran \right| \\
  & \qquad \lec \frac{1}{s \log^{\frac{3}{4}} s} \| \xi \|_{H^1}
  + \frac{(C^*)^2}{s^2 \log s},
\end{split}
\]
and so
\begin{equation} \label{momcomp}
\begin{split}
  &\left| \frac{d}{ds} \J(s,\xi) - v \left\lan 
  \left(e^{i\frac{v}{2}y} Q'\right)(\cdot - \frac{z}{2})
  - e^{-i\frac{v}{2}(y+\frac{z}{2})} Q'(\cdot + \frac{z}{2}), \; \frac{1}{2} \xi \cdot N''(P) \cdot \xi \right\ran \right| 
  \\ & \qquad \lec \frac{1}{s^2 \log^{\frac{3}{4}} s} \| \xi \|_{H^1} 
  + \frac{(C^*)^2}{s^3 \log s} + \frac{C^*}{s^3} \| \xi \|_{H^1}
  \lec \frac{1}{s^2} \| \xi \|_{H^1} + \frac{1}{s^3}.
\end{split}
\end{equation}

Using~\eqref{encomp} and~\eqref{momcomp}, 
together with~\eqref{mod0}, we have
\[
\begin{split}
  \left| \frac{d}{ds} \W(s,\xi) \right| &\lec
  |\dot z - 2v| \| \xi \|_{H^1}^2 + \frac{1}{s^2} \| \xi \|_{H^1}
  + \frac{1}{s^3} \\
  &\lec \frac{1}{s^2} \| \xi \|_{H^1} + \frac{1}{s^3}
  \lec \frac{C^*}{s^3},
\end{split}
\]
which completes the proof of Lemma~\ref{almost}.
\end{proof}

%%%%%%%%%%%%%%%%%%%%%%%%%%%%%%%%%%%%%%%%%%
\subsection{Coercivity} \label{coercive}

As in~\cite[Prop. 13]{Ngu19}, due to the orthogonality conditions~\eqref{orthos},
with $p < 5$ we have the standard coercivity property
\[
\begin{split}
  \| \xi \|_{H^1}^2 &\lec \W(s,\xi) 
  + \lan \xi, \cE_P^{cut} \ran - \frac{1}{2} \la \Ga |\xi(s,0)|^2 \\
  &\lec \W(s,\xi) + \lan \xi, \cE_P^{cut} \ran,
\end{split}
\]
so by~\eqref{cut} and Young's inequality,
\[
   \| \xi \|_{H^1}^2 \leq C \W(s,\xi) + C\frac{1}{s} \| \xi \|_2
   \leq  C \W(s,\xi) + \frac{C^2}{2s^2} + \frac{1}{2} \| \xi \|_{H^1}^2,
\]
and so, using~\eqref{Wbound},
\[
   \| \xi \|_{H^1}^2 \lec \W(s,\xi) + \frac{1}{s^2}
   \lec \frac{C^*}{s^2}.
\]
That is,
\[
  \| \xi \|_{H^1} \leq \frac{C \sqrt{C^*}}{s},
\]
and we close the second bootstrap estimate of~\eqref{boot}
by choosing $C^* \geq 4 C^2$, arriving at
\begin{equation} \label{xiest}
  \| \xi(s,\cdot) \| \lec \frac{1}{s}.
\end{equation}

%%%%%%%%%%%%%%%%%%%%%%%%%%%%%%%%%%%%%%%%%%
\subsection{Topological argument} \label{topology}

To complete the proof of Proposition~\ref{uniform}, it remains
to show that for all $s_{\mathrm f} > s_0$ ($s_0$ sufficiently large),
there is a choice of $z_{\mathrm f}$ in~\eqref{ICs}
such that the first bootstrap assumption of~\eqref{boot} 
indeed holds for all $s \in [s_0, s_{\mathrm f}]$
-- i.e., that $s^* = s_0$.
This part proceeds essentially as in~\cite[Sec 3.2.4]{Ngu19}.
Note we may freely use the estimates~\eqref{mod0}, \eqref{consequence}
and~\eqref{xiest}, which have all been established under the
first bootstrap assumption of~\eqref{boot}. 
 
First estimate, using also \eqref{vdyn2},
\[
  |2 v \dot v + f(z) e^{-z} \dot z|
  \leq 2 v | \dot v + f(z)e^{-z}| + f(z) e^{-z} |\dot z - 2v|
  \lec \frac{1}{s} \frac{1}{s^3} + \frac{1}{s^2}\frac{1}{s \log^{\frac{3}{4}} s} \lec \frac{1}{s^3 \log^{\frac{3}{4}} s}
\]
so the energy for the reference classical dynamics satisfies, 
using also the final condition~\eqref{ICs},
\[
  |v^2 - F(z)| \leq \int_s^{s_{\mathrm f}} 2 |v \dot v + f(z) e^{-z} \dot z| ds
  \lec \frac{1}{s^2 \log^{\frac{3}{4}} s}.
\]
Then, factoring,
\[
  |v - \sqrt{F(z)}| = \frac{|v^2 - F(z)|}{v + \sqrt{F(z)}} 
  \lec  \frac{1}{s \log^{\frac{3}{4}} s},
\]
and so
\[
  \left| \frac{d}{ds} \zeta(z) - 1 \right| = \left| \frac{\dot z}{2 \sqrt{F(z)}} - 1 \right|
  \leq \left| \frac{v - \sqrt{F(z)}}{\sqrt{F(z)}} \right|
  +  \left| \frac{\dot z - 2v}{2\sqrt{F(z)}} \right|
  \lec \frac{1}{\log^{\frac{3}{4}} s}.
\]
With this estimate -- which is the analogue of~\cite[eq (3.4.6)]{Ngu19} --
in hand, the topological argument
runs exactly as in~\cite[Sec 3.2.4, Step 2]{Ngu19},
to show that $z(s_{\mathrm f}) = z_{\mathrm f}$ may be chosen so the bounds above, as well as 
\[   
  |\zeta(z) - s| \leq \frac{s}{\log^{\frac{1}{2}} s},
\]
hold for $s \in [s_0,s_{\mathrm f}]$, uniformly in $s_{\mathrm f}$ as $s_{\mathrm f} \to \infty$. 

This estimate, together with its above-established consequences~\eqref{consequence},~\eqref{xiest} and~\eqref{lamest}, 
completes the proof of Proposition~\ref{uniform}. $\Box$

%--------------------------------------------
\subsection{Uniform estimates in the supercritical case} \label{super}

Here we explain how the arguments above may be modified to prove the uniform estimates needed for $p > 5$, Proposition~\ref{uniform2}.

\begin{proof}
We will just highlight the places where the argument differs from the above proof of Proposition~\ref{uniform}.

First, by~\eqref{finalorth} and the
parameter modulation of Proposition~\ref{modulation}, two of the
orthogonality conditions of~\eqref{orthos} still
hold,
\[
  \lan \eta(s,\cdot), i \La Q \ran \equiv 0, 
  \quad
  \lan \eta(s,\cdot), y Q \ran \equiv 0,
\]
but for the remaining one, we now merely have
\[
  \frac{d}{ds} \lan \eta(s,\cdot), Q \ran \equiv 0, \; \implies \;
  |\lan \eta(s,\cdot), Q \ran| \equiv |\lan \eta(s_{\mathrm f},\cdot), Q \ran|
  \lec  |\bv| \lec (s_{\mathrm f})^{-\frac{3}{2}} \lec s^{-\frac{3}{2}}.
\]
This bound suffices, as a replacement
for $\lan \eta, Q \ran = 0$ in~\eqref{m3},
at the cost of weakening~\eqref{modgam} to
\[
  |\dot\gat - 1| \lec \frac{1}{s^{\frac{3}{2}}},
\]
though~\eqref{mod0}, which is what is used
in all subsequent estimates, still holds.

Second, in the estimates involving
the localized momentum, and the linearized energy, time integration produces
a non-zero boundary term, since 
$\eta(s_{\mathrm f},\cdot) \not = 0$. However,
\[
  \| \eta(s_{\mathrm f},\cdot) \|_{H^1} \lec 
  s_{\mathrm f}^{-\frac{3}{2}}
\]
is a sufficient replacement.

Third, to ensure coercivity of the linearized energy, we add the inner-product estimates
\[
 |\lan \eta(s,\cdot), i Y^+ \ran| \leq \frac{1}{s^{\frac{3}{2}}}, \quad 
  |\lan \eta(s,\cdot), i Y^- \ran| \leq \frac{1}{s^{\frac{3}{2}}}
\]
to the bootstrap assumption.
The time derivatives of these inner-products
are estimated just as in~\cite[Lemma 21]{Ngu19}, using~\eqref{modeq} and observing that due to the
exponential decay~\eqref{unstable} of the eigenfunctions $Y^{\pm}$, the contributions from
the delta potentials in $L^{\pm}_z$ are 
lower-order error terms.
The bootstrap is then closed by choosing, simultaneously, $z_{\mathrm f}$ and $a_{\mathrm f}$ exactly as in~\cite[Lemma 23]{Ngu19}. 
\end{proof}

%%%%%%%%%%%%%%%%%%%%%%%%%%%%%%%%%%%%%%%%%%
\section{Nonexistence for stronger potentials}
\label{non}

In this section we give the proof of Theorem~\ref{nonexistence}. 
We focus on the subcritical case $p < 5$,
and refer the reader to Sections~\ref{supercrit} and~\ref{super} for how to modify the
arguments in the supercritcal case.

{\it Proof of Theorem~\ref{nonexistence} for $p < 5$}:
we assume the existence of a solution $u$ of~\eqref{NLS}
satisfying, for $t \geq T_0$ (for some $T_0$)
\[
  \left\| u(t,\cdot) - e^{i \th} \left[ Q(\cdot - \frac{\zz}{2})
  + Q(\cdot + \frac{\zz}{2}) \right] \right\|_{H^1} \leq \frac{c_1}{t},
\]
for some $c_1 > 0$, $\th = \th(t) \in \R$ and $\zz = \zz(t)$ satisfying
\begin{equation} \label{log}
  |\zz(t) - 2 \log t| \leq c_2
\end{equation}
as $t \to \infty$. 
We will derive a contradiction for $\Ga > 2$.

First note we may replace the above estimate with
\begin{equation} \label{2sol}
  \left\| u(t,\cdot) - e^{i \th} P(\cdot; \zz(t),0)
  \right\|_{H^1} \leq \frac{c_1}{t}
\end{equation}
by adjusting $c_1$.

We apply the parameter modulation Lemma~\ref{modulation}
(with Remark~\ref{fullmod}), imposing 
a full set of four orthogonality conditions to determine four parameters, velocity $v$ included,
to obtain:
\[
  \la(s) > 0, \; \gat(s) \in \R, \;
  z(s) > 0, \; v(s) \in \R,
\]
on a time interval
\begin{equation} \label{interval}
  s \in [s_0, \; s_{\mathrm f}], \quad s_{\mathrm f} = 
  t_{\mathrm f} \quad \mbox{ chosen sufficiently large},
\end{equation}
with
\[
  t(s) = t_{\mathrm f} - \int_s^{s_{\mathrm f}} \frac{ds}{\la^2(s)},
\]
\[
  e^{- i \gat(s)} \la^{-\frac{2}{p-1}}(s) u(t(s),\frac{y}{\la(s)})
  = P(y;z(s),v(s)) + \xi(s,y),
\]
and
\[
  \eta(s,y) = e^{-i \frac{v(s)}{2} y} \xi(s, y + \frac{z(s)}{2})
\]
satisfying the orthogonality conditions
\begin{equation}
\label{orthos2}
  0 \equiv \lan \eta, Q \ran \equiv \lan \eta, yQ \ran \equiv
  \lan \eta, i \La Q \ran \equiv \lan \eta, i T_z \ran.
\end{equation}
Note that the final data satisfy
\begin{equation} \label{final}
\begin{split}
  &|\la(s_{\mathrm f}) - 1| + |\gat(s_{\mathrm f}) - \th(t_{\mathrm f})|
  + |z(s_{\mathrm f}) - \zz(t_{\mathrm f})| + |v(s_{\mathrm f})| 
  + \| \xi(s_{\mathrm f},\cdot) \|_{H^1} \\
  & \qquad \lec \left\| u(t_{\mathrm f},\cdot) - e^{i \th(t_{\mathrm f})} P(\cdot; \zz(t_{\mathrm f}),0)
  \right\|_{H^1}
  \leq \frac{c_1}{t_{\mathrm f}} 
\end{split}
\end{equation}
by the assumption~\eqref{2sol}, and
so the hypotheses~\eqref{endparam}
and~\eqref{endclose} of the modulation lemma indeed hold
if $t_{\mathrm f}$ is chosen sufficiently large. 

We make two bootstrap assumptions:
\begin{equation} \label{zboot}
  |z(s) - 2 \log s| \leq 3 c_2
\end{equation}
and
\begin{equation} \label{xiboot}
  \| \xi(s,\cdot) \|_{H^1} \leq \frac{c_4}{s},
\end{equation}
and work on a time interval of the form~\eqref{interval}
on which these both hold. By~\eqref{final}, \eqref{log} and
\[
  |z(s_{\mathrm f}) - 2 \log s_{\mathrm f}| \leq |z(s_{\mathrm f}) - \zz(t_{\mathrm f})|
  + |\zz(t_{\mathrm f}) - 2 \log t_{\mathrm f}| \leq \frac{Cc_1}{t_{\mathrm f}} + c_2,
\]
this time interval is non-trivial if we choose
\[
   c_4 \gec c_1, \qquad 
   t_{\mathrm f} \geq \frac{C c_1}{c_2}.
\]
It follows from~\eqref{zboot} that
\begin{equation} \label{expest}
  e^{-\frac{z}{2}}  = \frac{1}{s} e^{\frac{1}{2}(2\log s - z)}
  \in  [ e^{-\frac{3}{2}c_2} \frac{1}{s}, e^{\frac{3}{2} c_2} \frac{1}{s}].
\end{equation}

Estimates for the time derivatives of the parameters
are obtained just as in Section~\ref{modparest}, by taking suitable inner
products with the PDE, and using~\eqref{expest}, resulting in
\begin{equation} \label{dotest}
  |\frac{\dot \la}{\la}| + |\dot v| \lec
  \frac{c_4^2}{s^2}, \qquad 
  |\dot \gat-1| \lec \frac{c_4 \log s}{s^2}, \qquad 
  |\dot z - 2v| \lec \frac{c_4 \log s}{s^2}.
\end{equation}

The energy estimate is done in the same way as
in Section~\ref{enest}, yielding
\[
  |\frac{d}{ds} \W(s)| \lec \frac{c_4}{s^3} \;\; \implies \;\;
  \W(s) \lec \| \xi(s_{\mathrm f},\cdot) \|^2 + \frac{c_4}{s^2}
  \lec \frac{1}{s_{\mathrm f}^2} + \frac{c_4}{s^2} \lec  \frac{c_4}{s^2}, 
\]
and so the bootstrap~\eqref{xiboot} is closed,
using the coercivity of $\W$ under~\eqref{orthos2},
\[
  \| \xi(s,\cdot) \|_{H^1} \lec \sqrt{\W(s)} \lec \frac{\sqrt{c_4}}{s}
\]
by choosing $c_4$ large enough.

So now we have
\[
  \| \la^{-\frac{2}{p-1}}(s) u(t(s),\frac{\cdot}{\la(s)}) -  e^{i \gat(s)} P(\cdot;z(s),v(s)) \|_{H^1} \lec \frac{1}{s}.
\]
Since
\[
  |\la(s)-1| \lec |\log \la(s)| = |\log \la(s_{\mathrm f}) - \int_{s}^{s_{\mathrm f}}
  \frac{\dot \la(\tau)}{\la(\tau)} d\tau |
  \lec \frac{1}{s_{\mathrm f}} + \frac{1}{s} \lec \frac{1}{s}
\]
by~\eqref{final} and~\eqref{dotest}, we have
\[
  \| \la^{-\frac{2}{p-1}}(s) u(t(s),\frac{\cdot}{\la(s)})
  - u(t(s),\cdot) \|_{H^1} \lec |\la(s)-1| \lec \frac{1}{s},
\]
and so
\[
  \| u(t(s),\cdot) -   e^{i \gat(s)} P(\cdot;z(s),v(s)) \|_{H^1} \lec \frac{1}{s}.
\]
Combining this with~\eqref{2sol}, and using
\begin{equation} \label{stest}
  |t(s) - s| \leq \int_s^{s_{\mathrm f}} |\la^2(\tau)-1| \frac{d \tau}{\la^2(\tau)}
  \lec \int_s^{s_{\mathrm f}} \frac{d \tau}{\tau^2} \lec \frac{1}{s},
\end{equation}
so that
\[
  \frac{1}{t(s)} \lec \frac{1}{s},
\]
we see that
\[
  \| P(\cdot;\zz(t(s)),0) -  e^{i (\gat(s) - \th(t(s))}
  P(\cdot; z(s),v(s)) \|_{H^1} \lec \frac{1}{s},
\]
and so applying Lemma~\ref{Pdiff} 
(shown in Section~\ref{comps}), we arrive at
\[
  |z(s) - \zz(t(s)) | \lec \frac{1}{s},
\]
from which follows
\[
  |z(s) - 2 \log s| \leq |\zz(t(s)) - 2 \log t(s)|
  + 2|\log t(s) - \log s| + \frac{C}{s} \\
  \leq c_2 + \frac{C}{s}
\]
using~\eqref{log} and~\eqref{stest}.
This shows that the bootstrap estimate~\eqref{zboot}
holds at least on the interval
$s \in \left[ \frac{C}{2c_2}, s_{\mathrm f} \right]$.

We repeat the calculations from Section~\ref{momentum} of the time derivatives
of the localized momentum, and of 
$\lan \eta, i T_z \ran \equiv 0$,
to find:
\[
  C_2 \left( \dot v + f(z) e^{-z} \right)
  = \frac{d}{ds} \M(s) + O(\frac{1}{s^2 \log s}), 
\]
and so, integrating in $s$,
\[
  C_2 \left( v(s_{\mathrm f}) - v(s)  + \int_s^{s_{\mathrm f}}
  f(z(\tau)) e^{-z(\tau)} d \tau \right) =
  \M(s_{\mathrm f}) - \M(s) + O(\frac{1}{s \log s}).
\]
Using $\Ga > 2$ and~\eqref{sign}, 
as well as 
\[
  |\M(s)| \lec \| \xi(\cdot,s) \|_{H^1}^2 \lec \frac{1}{s^2},
  \qquad |v(s_{\mathrm f})| \lec \frac{1}{s_{\mathrm f}} \lec \frac{1}{s},
\]
and~\eqref{expest}, we get
\[
  v(s) \lec -\int_s^{s_{\mathrm f}} e^{-z(\tau)} d \tau 
  + \frac{C}{s} \lec
  -\int_s^{s_{\mathrm f}} \frac{d \tau}{\tau} 
  + \frac{C}{s}
  = -\log \frac{s_{\mathrm f}}{s} + \frac{C}{s}.
\]
Thus for some constant $c > 0$, we have
\[
  v(s) \leq 0 \; \mbox{ for } \;
  s \in (s_0, s_1], \qquad s_1 = s_{\mathrm f} - c.
\]
Then
\[
\begin{split}
  z(s_0) - z(s_1) &= -\int_{s_0}^{s_1} \dot z(\tau) d \tau
  = -2\int_{s_0}^{s_1} v(\tau) d \tau + \int_{s_0}^{s_1} (2v(\tau) - \dot z(\tau)) d \tau  \\
  &\gec -\int_{s_0}^{s_1} \frac{\log \tau}{\tau^2} d\tau
  \gec -\frac{\log s_0}{s_0}.
\end{split}
\]
using~\eqref{dotest}. On the other hand, by~\eqref{zboot},
\[
  z(s_0) - z(s_1) \leq 2 \log s_0 - 2 \log s_1
  + 6c_2,
\]
which provides a contradiction by taking $s_{\mathrm f}$ (hence also $s_1$) sufficiently large.
$\Box$

%%%%%%%%%%%%%%%%%%%%%%%%%%%%%%%%%%%%%%%%%%
\section{Properties of the perturbed eigenfunction} 
\label{properties}

Here we prove Proposition~\ref{efprop}.
\begin{proof}
Recall
\[
  L^+ = -\p_x^2 + 1 - p Q^{p-1}(x), \qquad
  L^+_z = L^+ + \Ga \de_{-\frac{z}{2}},
\]
and denote their corresponding quadratic forms by
\[
\begin{split}
   q^+(f,g) &=  \re  \int_\R \left( \bar{f}'(x) g'(x) + 
  (1 - p Q^{p-1}(x)) \bar{f}(x) g(x) \right) dx \\
  q^+_z(f,g) &= q^+(f,g) +  \Ga 
  \re \bar{f}\left(-\frac{z}{2}\right) g \left(-\frac{z}{2}\right), \qquad f, g, \in H^1(\R).
\end{split}
\]
The operator $L^+$ is a classical Schr\"odinger 
operator, and its spectral properties are well known -- 
see, eg \cite{PT33}. We recall here that 
it has ground state
\[
  0 < \phi_0 \in H^2(\R) \mbox{ even, }
  \| \phi_0 \|_2 = 1, \;\; L^+ \phi_0 = -e_0 \phi_0, \;\; e_0 = \frac{(p+1)^2}{4}-1, 
\]
with exponential decay,
\begin{equation} \label{gsexp}
  |\phi_0(y)| + |\phi_0'(y)| \lec e^{-\frac{p+1}{2}|y|},
\end{equation}
and $Q'$ (which is odd) is its first excited state:
\begin{equation} \label{es}
  L^+ Q' = 0, \mbox{ and } 
  H^1(\R) \ni f \perp \phi_0, Q' \; \implies
  (f, \; L^+ f) \gec \| f \|_{H^1}^2.
\end{equation}

Consider first the ground state eigenvalue and eigenfunction of $L^+_{z}$:
\[
  L^+_{z} \phi_{z} = -e_{z} \phi_{z}, \qquad
  \phi_z \in D^\Ga_{-\frac{z}{2}}, \qquad \| \phi_{z} \|_2 = 1,
  \qquad \phi_{z} > 0.
\]
Since $L^+_{z} \geq L^+$, we have $e_{z} < e_0$, and
by variational characterization,
\[
  -e_{z} \leq q^+_z(\phi_0, \phi_0) = -e_0 + \Ga |\phi_0(-z/2)|^2,
\]
so by~\eqref{gsexp},
\begin{equation} \label{gsest}
  e_{z} = e_0 - \tilde e_{z}, \qquad
  0 < \tilde e_{z} \lec \Ga e^{-\frac{p+1}{2}z}.
\end{equation}
Decompose
\begin{equation} \label{gssplit}
  \phi_{z} = (1-a) \phi_0 + b \frac{Q'}{\| Q' \|_2} + h, \qquad
  h \perp \phi_0, Q', \qquad a,b \in \R,
\end{equation}
so
\begin{equation} \label{normsplit}
  1 = \| \phi_{z} \|_2^2 = (1 - a)^2 + b^2 + \| h \|_2^2.
\end{equation}
Since $\phi_0, \; Q' \; \in H^2$,
\[
  q^+(\phi_0,Q') = q^+(\phi_0,h) = q^+(Q',h) = 0,
\]
and so
\[
\begin{split}
  -e_{z} &= q_z^+( \phi_{z}, \phi_z) = q^+(\phi_z,\phi_z) + \Ga |\phi_z(-\frac{z}{2})|^2 \\ 
  &= -(1-a)^2 e_0 + q^+(h, h)  + \Ga |\phi_z(-\frac{z}{2})|^2. 
\end{split}
\]
Then by~\eqref{normsplit} and~\eqref{es},
\[
  \tilde e_{z} - \left(  b^2 + \| h \|_2^2 \right) e_0 
  = q^+(h,h) +  \Ga |\phi_z(-\frac{z}{2})|^2
  \gec \| h \|_{H^1}^2,
\]
hence by~\eqref{gsest},
\[
  \| h \|_{H^1}^2 + b^2
  \lec \Ga  e^{-\frac{p+1}{2}z}.
\]
Then from~\eqref{normsplit}
\[
  |1 - (1-a)^2| \lec \Ga  e^{-\frac{p+1}{2}z},
\]
and by~\eqref{gssplit}, since
$\phi_z, \; \phi_0 \; > 0$ we have
$|a| \lec \Ga  e^{-\frac{p+1}{2}z}$, and so
\begin{equation} \label{gsef}
  \| \phi_{z} - \phi_0 \|_{H^1} \lec |a| + |b| + \| h \|_{H^1} 
  \lec  \sqrt{\Ga}  e^{-\frac{p+1}{4} z}.
\end{equation}

Now we turn to the first excited state eigenvalue and eigenfunction of $L^+_z$. Decompose
\[
  Q' = b \| Q' \|_2 \; \phi_{z} + g, \qquad 
  g \perp \phi_{z}, \qquad b \in \R,
\]
and note
\begin{equation} \label{gnorm}
  \| g \|_2^2 = (1-b^2) \| Q' \|_2^2 = \| Q' \|_2^2 + O(e^{-\frac{p+1}{2}z})
\end{equation}
by~\eqref{gsef}. Compute, using
\[
  q^+(Q',Q') = q^+(Q',\phi_z) = q^+(\phi_z,Q') = 0
\]
(since $Q' \in H^2$, $\phi_z \in H^1$),
\[
\begin{split}
  q_z^+( g, g) &= q^+( Q' - b \| Q' \|_2 \phi_{z} , 
  Q' - b \| Q' \|_2 \phi_{z}) + \Ga |g(-\frac{z}{2})|^2 \\
  &=
  b^2 \| Q'\|_2^2 \; q^+(\phi_z,\phi_z)  + \Ga |g(-\frac{z}{2})|^2 \\
  &= b^2 \| Q'\|_2^2 \left( q_z^+(\phi_z,\phi_z) - \Ga|\phi_z(-\frac{z}{2})|^2  \right) + \Ga |g(-\frac{z}{2})|^2 \\ 
  &\leq -e_z b^2 \|Q'\|_2^2 + \Ga \left( |g(-\frac{z}{2})|^2 
  - |\phi_z(-\frac{z}{2})|^2 \right) \\
  & \leq  -e_z b^2 \|Q'\|_2^2 + \Ga |Q'(-\frac{z}{2})|^2
  + 2 \Ga b^2 \| Q'\|_2 |\phi_z(-\frac{z}{2})|^2 \\
  & \leq \Ga c_p^2 e^{-z} + 
  O(e^{-\frac{p+1}{2} z}) \quad < \inf \sigma_{ess}(L^+_{z}) = 1, 
\end{split}
\]
using~\eqref{asy},~\eqref{gsef} and~\eqref{gsexp}. 
By the variational characterization, this estimate, together with~\eqref{gnorm}, establishes
the existence of a first excited state eigenvalue $\nu_{z}$, satisfying
\begin{equation} \label{esest} 
  0 < \nu_z \leq \frac{q^+_z(g,g)}{\| g \|_2^2} \leq 
  \frac{\Ga c_p^2}{\| Q' \|_2^2} e^{-z} + 
  O(e^{-\frac{p+1}{2} z})
\end{equation}
($\nu_{z} > 0$ follows from $L^+_{z} \geq L^+$),
and eigenfunction  $T_{z}$, i.e. a solution to~\eqref{eigen}. Choose $T_z$ to be normalized as~\eqref{normal} and real-valued. 
Decompose
\begin{equation} \label{esef}
  T_{z} = \al \phi_0 + \be Q'  + f, \qquad f \perp \phi_0, Q',
  \qquad \al, \be \in \R,
\end{equation}
so that
\begin{equation} \label{esnorm}
  \| Q' \|_2^2 = \| T_{z} \|_2^2 = \al^2 + \be^2 \| Q' \|_2^2 + \| f \|_2^2.
\end{equation}
Note
\[
\begin{split}
  0 &= (\phi_{z}, T_{z}) =  \left( (1-a) \phi_0 + b \frac{Q'}{\| Q' \|_2} + h,
  \;  \al \phi_0 + \be Q'  + f \right) \\
  &= \al(1-a) + b \be \| Q' \|_2 + (h,f),
\end{split}
\]
so by~\eqref{gsef},
\begin{equation} \label{alest}
  |\al| \lec \frac{1}{1-a} \left( |b| + \| h \|_2 \right)
  \lec  \sqrt{\Ga}  e^{-\frac{p+1}{4}z}.
\end{equation}
Now
\begin{equation} \label{esquad}
\begin{split}
  \| Q' \|_2^2 \; \nu_{z} &= q^+_z(T_z,T_z)
  = q^+(T_{z}, T_{z}) + \Ga T_{z}^2(-z/2) \\
  &= -\al^2 e_0 + q^+(f, f) + \Ga T_{z}^2(-z/2),
\end{split}
\end{equation}
where we used
\[
  0 = q^+(\phi_0,Q') = q^+(\phi_0,f) = q^+(Q',f)
\]
(since $\phi_0, \; Q' \; \in H^2$, $f \in H^1$),
so by~\eqref{es}, \eqref{alest} and~\eqref{esest}, 
\begin{equation} \label{fest}
  \| f \|_{H^1}^2 \lec q^+(f, f) \lec \al^2 + \nu_z
  \lec \Ga  e^{-\frac{p+1}{2}z} + \Ga e^{-z}
  \lec \Ga e^{-z}.
\end{equation}
From~\eqref{esnorm} and~\eqref{alest}, 
\[
  0 \leq 1 - \be^2 \lec \al^2 + \| f \|_2^2 \lec  \Ga e^{-z},
\]
and so
\[
  \be = \pm 1 + O(\Ga e^{-z}),
\]
and we may replace $T_{z}$ by $-T_{z}$ if needed, to produce 
\[
  |\be - 1| \lec \Ga e^{-z},
\]
which combined with~\eqref{esef}, \eqref{alest} and~\eqref{fest},
yields the $H^1$ bound of~\eqref{esefest}. 

Using~\eqref{esest} and~\eqref{alest}, \eqref{esquad} we also 
get an upper bound
\begin{equation} \label{Tupper}
\begin{split}
  \Ga T_{z}^2\left(-\frac{z}{2}\right) &\leq 
  \| Q' \|_2^2 \; \nu_{z} + \al^2 e_0 \\
  &\leq \Ga c_p^2 e^{-z} \left( 1 + O(e^{-\frac{p-1}{2} z}) \right).
\end{split}
\end{equation}

Next we show the pointwise estimates~\eqref{pointwise}.
Note that the $L^1$ estimate of~\eqref{esefest} then follows directly by integrating~\eqref{pointwise}.
Set
\[
  g := T_{z} - Q',
\]
so that the $H^1$ estimate of~\eqref{esefest} reads
\begin{equation} \label{gbound}
  \| g \|_{H^1} \lec \Ga^{\frac{1}{2}} e^{-\frac{z}{2}},
\end{equation}
and note from~\eqref{normal} that 
\[
\begin{split}
  &0 = \| Q' + g \|_2^2 - \| Q' \|_2^2 =
  2 ( Q', g ) + \| g \|_2^2 \\
  & \;\; \implies \;\;
  0 \leq -2( Q', g ) = \| g \|_2^2 \lec \Ga e^{-z}.
\end{split}
\]
Since
\[
  \left| \int_{-\infty}^{-\frac{z}{2}} Q' \; g \right| \lec
  \left( \int_{-\infty}^{-\frac{z}{2}} e^{2y} \right)^{\frac{1}{2}}
  \| g \|_2 \lec \Ga^{\frac{1}{2}} e^{-z},
\]
we must in fact have
\begin{equation} \label{partorth}
  \left| \int_{-\frac{z}{2}}^\infty Q' \; g \right| \lec
  (\Ga + \Ga^{\frac{1}{2}}) e^{-z}.
\end{equation}
To the right of the singularity, the eigenvalue equation~\eqref{eigen} may be written
\begin{equation} \label{eigen2}
  y > -\frac{z}{2} \; \implies \;
  L^+ g = s, \quad s(y) := \nu_{z}(Q'(y) + g(y)).
\end{equation}
By~\eqref{esest}, \eqref{asy}, and~\eqref{gbound}, we have
\begin{equation} \label{sbound}
  |s(y)| \lec \Ga e^{-z} ( e^{-|y|} + e^{-\frac{z}{2}} ).
\end{equation}
We will express the solution of~\eqref{eigen2} 
by the variation of parameters formula, introducing a second element of
$\ker(L^+)$, complementary to $Q'$:
\begin{equation} \label{Y+}
  L^+ Y = 0, \qquad Y(y) \sim e^y, \;\;
  Y'(y) \sim e^y, \mbox{ as } y \to \infty,
  \qquad Y \mbox{ even}.
\end{equation}
By~\eqref{asy}, the Wronskian of $Y$ and $Q'$ is 
\[
  W[Y, \; Q'] = \left| 
  \begin{array}{cc}  e^y & -c_p e^{-y} \\
  e^y & c_p e^{-y} \end{array} \right| = 2 c_p,
\]
and since $Y$ is even,
\begin{equation} \label{Y-}
  Y(y) \sim e^{-y}, \;\;
  Y'(y) \sim -e^{-y}, \;\; y \to -\infty.
\end{equation}
Since $g \in H^1$ and solves~\eqref{eigen2}, it may be represented
via the variation of parameters formula as
\begin{equation} \label{rep}
  2 c_p g(y) = -Y(y) \int_y^\infty Q' s
  -  Q'(y) \int_{-\frac{z}{2}}^y Y s
  + \zeta Q'(y), \qquad y > -\frac{z}{2}
\end{equation}
for some $\zeta \in \R$. 
Using~\eqref{rep} in~\eqref{partorth} shows
\begin{equation} \label{zetaest1}
   e^{-z} \gec |\zeta| \int_{-\frac{z}{2}}^\infty (Q')^2
  - \left|\int_{-\frac{z}{2}}^\infty Y(y) Q'(y) \int_y^\infty 
  Q' s \right|
  - \left|\int_{-\frac{z}{2}}^\infty  (Q'(y))^2 \int_{-\frac{z}{2}}^y Y s \right|.
\end{equation}
Using~\eqref{asy},~\eqref{sbound},~\eqref{Y+} and~\eqref{Y-}, 
we have
\[
   \int_{-\frac{z}{2}}^\infty (Q')^2
   = \| Q' \|_2^2 - O(e^{-z}),
\]
\begin{equation} \label{term1}
\begin{split}
  \left| \int_y^\infty Q' s \right| &\lec \Ga
  e^{-z} \int_y^\infty (e^{-2|t|} + e^{-\frac{z}{2}} e^{-|t|} ) dt
  \\ &\lec \Ga e^{-z} \left\{  \begin{array}{cc} 
  e^{-2y} + e^{-\frac{z}{2}} e^{-y} & y \geq 0 \\
  1 & y \leq 0 \end{array} \right.,
\end{split}
\end{equation}
\[
\begin{split}
  \left| \int_{-\frac{z}{2}}^\infty Y(y) Q'(y) \int_y^\infty Q' s \right| &\lec \Ga e^{-z} \left(
  \int_{-\frac{z}{2}}^0  dy + \int_0^\infty 
  (e^{-2y} + e^{-\frac{z}{2}} e^{-y}) dy \right) 
  \\ & \lec \Ga ze^{-z},
\end{split}
\]
\begin{equation} \label{term2}
\begin{split}
  &\left| \int_{-\frac{z}{2}}^y Y s \right| \lec \Ga e^{-z} 
  \int_{-\frac{z}{2}}^y (1 + e^{-\frac{z}{2}} e^{|t|} ) dt
  \\ & \quad \lec \Ga e^{-z} \left\{ \begin{array}{cc} 
  y + \frac{z}{2} + 1 - e^{-\frac{z}{2}} e^{-y} 
  \lec z & -\frac{z}{2} \leq y \leq 0 \\
  \frac{z}{2} + 1 - e^{-\frac{z}{2}} 
  + y + e^{-\frac{z}{2}}(e^y-1) \lec z + y + e^{-\frac{z}{2}} e^y & y \geq 0
  \end{array} \right.,
\end{split}
\end{equation}
and
\[
\begin{split}
  \left|\int_{-\frac{z}{2}}^\infty  (Q'(y))^2 \int_{-\frac{z}{2}}^y Y s \right| & \lec \Ga e^{-z} \left( 
  z \int_{-\frac{z}{2}}^0 e^{2y} dy \right. \\
  & \quad \left. + \int_0^\infty e^{-2y}
  \left( z + y + e^{-\frac{z}{2}} e^y \right) dy \right)
  \lec \Ga z e^{-z},
\end{split}
\]
and inserting these estimates into~\eqref{zetaest1} shows that
\begin{equation} \label{zetaest2}
  |\zeta| \lec \Ga z e^{-z}.
\end{equation}
The second and third pointwise estimates of~\eqref{pointwise} then follow directly from~\eqref{asy}, \eqref{Y+}, \eqref{Y-}, \eqref{term1}, \eqref{term2} and~\eqref{zetaest2}.
Moreover, it follows from standard exponential decay estimates
for ode that for $|y| > \frac{z}{2}$, since
\[
  -g'' + (1 - \nu_{z}) g = p Q^{p-1} g  + \nu_z Q'(y), 
  \qquad |y| > \frac{z}{2},
\]
and $|Q^{p-1}(y)| \lec e^{-(p-1)|y|}$, $|Q'(y)| \lec e^{-|y|}$,
we have
\[
\begin{split}
  |g(y)| &\lec \| g\|_{H^1} \left\{
  \begin{array}{cc}
  e^{-\sqrt{1-\nu_{z}} (y-\frac{z}{2})} & y \geq \frac{z}{2} \\
  e^{\sqrt{1-\nu_{z}} (y+\frac{z}{2})} & y \leq -\frac{z}{2} \end{array} \right. \\
  &\lec \Ga^{\frac{1}{2}} e^{-\sqrt{1-\nu_{z}} |y|}, 
  \qquad |y| \geq \frac{z}{2},
\end{split}
\]
using~\eqref{gbound}, and
$e^{\sqrt{1-\nu_{z}} \frac{z}{2}} \leq e^{\frac{z}{2}}$,
which establishes the first estimate of~\eqref{pointwise}.

From the representation formula~\eqref{rep}, we can also
extract a formula relating $T_z\left(-\frac{z}{2}\right)$
and $\nu_z$ which will be used in showing both~\eqref{nutight}
and~\eqref{Ttight}: by~\eqref{zetaest2},
and the asymptotics~\eqref{asy}, \eqref{Y+} and~\eqref{Y-},
we see
\[
  2 c_p g\left(-\frac{z}{2}\right) = -e^{\frac{z}{2}} \; \nu_{z}
  \int_{-\frac{z}{2}}^\infty Q'(Q' + g)
  + O(z e^{-\frac{3}{2} z}),
\]
and since
\[
  \left| \int_{-\frac{z}{2}}^\infty Q' \; g \right|
  \lec \| g \|_{H^1} \lec e^{-\frac{z}{2}},
\]
\[
  2 c_p g\left(-\frac{z}{2}\right) = - \| Q'\|_2^2 \; e^{\frac{z}{2}} \; \nu_{z}
  + O(e^{-z}),
\]
so
\begin{equation} \label{Tatz}
\begin{split}
  T_{z}\left(-\frac{z}{2}\right) &= Q'\left(-\frac{z}{2}\right) 
  + g\left(-\frac{z}{2}\right) \\
  &= c_p e^{-\frac{z}{2}} - \frac{\| Q' \|_2^2}{2 c_p}  e^{\frac{z}{2}} \nu_{z} + O(e^{-z}).
\end{split}
\end{equation}
% Inserting~\eqref{esest} here produces
% \[
% \begin{split}
%   T_{\Ga,z}(-z/2) &\geq c_p e^{-\frac{z}{2}} - \frac{\| \p_y Q\|_2^2}{2 c_p}  e^{\frac{z}{2}}
% \frac{\Ga c_p^2}{\| \p_y Q \|_2^2} e^{-z}
%   \left(1 + O(e^{-\frac{p-1}{2} z} + ze^{-z}) \right) \\
%   & =  c_p e^{-\frac{z}{2}} \left( 1 - \frac{\Ga}{2} 
%   - O(e^{-\frac{p-1}{2} z} + ze^{-z}) \right),
% \end{split}
% \]
% which is the lower bound of~\eqref{pointwise2}.
% Inserting this lower bound into~\eqref{esquad} and 
% using~\eqref{alest} produces the lower bound~\eqref{esest2}.

We turn next to the bounds~\eqref{nutight} and~\eqref{Ttight}.
First we improve the upper bound~\eqref{esest}, which came
from using (the projection orthogonal to the ground state 
$\phi_z$ of) $Q'$ in the quadratic form $q^+_z$,
by designing a better test function: let
\[
  R = R(z), \qquad 1 \leq R \leq \frac{z}{4}
\]
(to be chosen later), and set
\[
  \psi(x) = \left\{ \begin{array}{cc}
  Q'(x) & x \geq -\frac{z}{2} + R \\
  A e^{\be(x+\frac{z}{2})} + B  e^{-\be(x+\frac{z}{2})} &
  -\frac{z}{2} \leq x \leq -\frac{z}{2} + R \\
  (A + B)  e^{\be(x+\frac{z}{2})} & x \leq -\frac{z}{2}
  \end{array} \right\}, \qquad \be = \sqrt{1 - \nu_z} 
  = 1 - O(e^{-z})
\]
where
\[
  A = \frac{\Ga + 2\be}{D} Q'\left(-\frac{z}{2} + R\right), \;\;
  B = -\frac{\Ga}{D} Q'\left(-\frac{z}{2} + R\right),
\]
\[
  D = (\Ga + 2\be) e^{\be R} - \Ga e^{-\be R}
  = (\Ga + 2)e^R - \Ga e^{-R} + O(e^{-z}),
\]
so 
\[
  \psi\left(-\frac{z}{2}\right) = A + B = 
  \frac{2\be}{D} Q'\left(-\frac{z}{2} + R\right).
\]
By construction, $\psi$ is continuous, 
\begin{equation} \label{reg} 
  \psi \in H^1(\R) \cap H^2(\R \backslash \{-\frac{z}{2}, -\frac{z}{2}+R\}), 
\end{equation}  
and at $x=-\frac{z}{2}$ satisfies the 
jump condition of $\D^\Ga_{-\frac{z}{2}}$:
\begin{equation} \label{jump}
  \psi'\left(-\frac{z}{2}+\right) - \psi'\left(-\frac{z}{2}-\right) = \Ga \psi\left(-\frac{z}{2}\right),
\end{equation}
Compute
\begin{equation} \label{testnorm}
\begin{split} 
  \| \psi \|_2^2 &= \int_{-\frac{z}{2}+ R}^\infty (Q'(x))^2 dx
  + \int_{-\frac{z}{2}}^{-\frac{z}{2}+R}
  \left(A^2 e^{\be(2x+z)} + B^2 e^{-\be(2x+z)}
  + 2 AB \right) dx \\
  & \quad + (A+B)^2 \int_{-\infty}^{-\frac{z}{2}} 
  e^{\be(2x+z)} dx = \| Q' \|_2^2 + O(e^{2R-z}).
\end{split}
\end{equation}
Using~\eqref{reg},~\eqref{jump} and
\[
  L^+ Q' = 0, \quad
  L^+ e^{\pm \be(x+\frac{z}{2})} = \left( \nu_z - p Q^{p-1}(x) \right) e^{\pm \be(x+\frac{z}{2})},
\]
we find that
\begin{equation} \label{testquad}
\begin{split}
  q_z^+(\psi,\psi) &= \int_{-\infty}^{\infty}
  \psi(x) L^+ \psi(x) dx + \mu \psi(-\frac{z}{2}+R) \\
  &= \int_{-\infty}^{-\frac{z}{2}} \left( \nu_z - p Q^{p-1}(x) \right) (A + B)^2 e^{\be(2x+z)} dx \\
  & \quad +  \int_{-\frac{z}{2}}^{-\frac{z}{2}+R} \left( \nu_z - p Q^{p-1}(x) \right) \left[A e^{\be(x+\frac{z}{2})} + B  e^{-\be(x+\frac{z}{2})} \right]^2 dx 
  + \mu Q'\left(-\frac{z}{2}+R\right) \\
   &= \mu  Q'\left(-\frac{z}{2}+R\right) + 
   O(e^{2R-2z} + e^{\frac{p+1}{2}(2R-z)}),
\end{split}
\end{equation}
where
\[
  \mu = \psi'\left(-\frac{z}{2} + R - \right) - \psi'\left(-\frac{z}{2} + R + \right)
  = \be (A  e^{\be R} - B e^{-\be R}) - Q''\left(-\frac{z}{2}+R\right).
\]
Estimate, using~\eqref{ode} and~\eqref{relation},
\[
\begin{split}
  \mu &= \frac{\be Q'(-\frac{z}{2} + R)}{D} 
  \left( (\Ga + 2 \be)e^{\be R} + \Ga e^{-\be R} \right)
  - Q''(-\frac{z}{2} + R) \\
  &= \frac{Q'(-\frac{z}{2} + R)}{D} 
   \left( \be \left( (\Ga + 2 \be)e^{\be R} + \Ga e^{-\be R} \right) 
 - D \frac{Q}{Q'}(-\frac{z}{2} + R) \right) + O(e^{\frac{p}{2}(2R- z)}) \\
 &=  \frac{Q'(-\frac{z}{2} + R)}{D} 
   \left( (\be+1) \Ga e^{-\be R} \right) 
   + O(e^{\frac{p}{2}(2R-z)} + e^{\frac{1}{2}(2R-3z)}) \\
  &= \frac{c_p(\be+1) \Ga}{D} e^{-\frac{z}{2}} e^{(1 - \be)R} +
   O(e^{\frac{p}{2}(2R-z)} + e^{\frac{1}{2}(2R-3z)}) \\
  &= \frac{2\Ga c_p}{(\Ga+2)e^R - \Ga e^{-R}} e^{-\frac{z}{2}}  + 
   O(e^{\frac{p}{2}(2R-z)} + e^{\frac{1}{2}(2R-3z)}).
\end{split}
\]
So from~\eqref{testquad},
\begin{equation} \label{testquad2}
  q_z^+(\psi, \psi) =
  \frac{2\Ga c_p^2}{\Ga+2 -\Ga e^{-2R}} e^{-z}  
  + O(e^{\frac{p+1}{2}(2R-z)} + e^{2(R-z)}).
\end{equation}
To use the quadratic form to get an estimate for $\nu_z$,
we need the test function to be orthogonal to the ground state
$\phi_z$. So project:
\[
  \tilde{\psi} = P_{\phi_z^\perp} \psi = 
  \psi - (\phi_z, \psi) \phi_z.
\]
From~\eqref{gssplit} and~\eqref{gsef}, we have
\[
  |(\phi_z,Q')| = \| Q' \|_2 |b| \lec e^{-\frac{p+1}{4}z},
\]
and
\[
  \int_{-\infty}^{-\frac{z}{2}+R} \phi_z^2
  \lec  e^{\frac{p+1}{2}(2R-z)},
\]
and so by H\"older,
\[
\begin{split}
  \left| \int_{-\frac{z}{2}+R}^\infty \phi_z \; \psi \right|
  &\lec \left| \int_{-\infty}^{-\frac{z}{2}+R} Q' \; \phi_z \right| +  |(\phi_z,Q')| \\
  &\lec e^{R-\frac{z}{2}} e^{\frac{p+1}{4} (2R-z)}
  +  e^{-\frac{p+1}{4} z} = 
  e^{\frac{p+3}{4}(2R-z)} + e^{-\frac{p+1}{4} z},
\end{split}
\]
and moreover
\[
  \left| \int_{-\infty}^{-\frac{z}{2}+R} 
  \phi_z \; \psi \right| \lec  e^{-\frac{p+1}{4} z}.
\]
Combining the above yields
\begin{equation} \label{innersmall}
  |( \phi_z, \psi)| \lec  
   e^{\frac{p+3}{4}(2R-z)} + e^{-\frac{p+1}{4} z}.
\end{equation}
So from~\eqref{testnorm},
\begin{equation} \label{testnorm2}
  \| \tilde \psi \|_2^2 = \| \psi \|_2^2 - (\phi_z,\psi)^2
  = \| Q' \|_2^2 - O(e^{2R-z}).
\end{equation}
Now since $\phi_z \in \D_{-\frac{z}{2}}^\Ga$, and using~\eqref{testquad2} and~\eqref{innersmall},
\begin{equation} \label{testquad3}
\begin{split}
  q_z^+(\tilde\psi,\tilde\psi) &= 
  q_z^+(\psi - (\phi_z, \psi) \phi_z, \psi - (\phi_z, \psi) \phi_z)
  \\ &= q_z^+(\psi,\psi) - 2(\phi_z,\psi) q_z^+(\phi_z,\psi)
  + (\phi_z,\psi)^2 q_z^+(\phi_z,\phi_z) \\
  &= q_z^+(\psi,\psi) - 2(\phi_z,\psi) (L^+ \phi_z,\psi)
  + (\phi_z,\psi)^2 (L^+ \phi_z,\phi_z) \\
  &=  q_z^+(\psi,\psi) + e_0 (\phi_z, \psi)^2 \\
  &= \frac{2\Ga c_p^2}{\Ga+2 -\Ga e^{-2R}} e^{-z}  
  + O(e^{\frac{p+1}{2}(2R-z)} + e^{2(R-z)}).
\end{split}
\end{equation}
Using~\eqref{testquad3} and~\eqref{testnorm2} in the
variational principle yields
\[
  \| Q' \|_2^2 \nu_z \leq 
  \frac{\| Q' \|_2^2}{\| \tilde \psi \|_2^2} 
  q_z^+ (\tilde \psi, \tilde \psi)
  = \frac{2\Ga c_p^2}{\Ga+2 - \Ga e^{-2R} } e^{-z}  
  + O(e^{\frac{p+1}{2}(2R-z)} + e^{2(R-z)}).
\]
So
\[
  \nu_z - \frac{2}{\Ga+2}
  \frac{\Ga c_p^2}{\| Q' \|_2^2} e^{-z} \lec
  e^{-(z+2R)} + e^{\frac{p+1}{2}(2R-z)} + e^{2(R-z)}.
\]
To conclude, choose
\[
  R = \left\{ \begin{array}{cc} 
  \frac{p-1}{2(p+3)} z & 2 < p < 5 \\
  \frac{1}{4} z & p > 5 \end{array} \right.
\]
to find
\begin{equation} \label{improved}
  \nu_z \leq \frac{2}{\Ga+2}
  \frac{\Ga c_p^2}{\| Q' \|_2^2} e^{-z}  
  + O( e^{-(1+a_p) z} ), \quad
  a_p = \min \left( \frac{p-1}{p+3},\frac{1}{2} \right).
\end{equation}

Defining the quantities
\[
  \tau_z = \frac{\| Q' \|_2^2 \nu_z}{2 c_p^2 e^{-z}},
  \qquad \rho_z = \frac{T_z(-\frac{z}{2})}{c_p e^{-\frac{z}{2}}},
\]
the upper bound~\eqref{improved} reads
\begin{equation} \label{improved2}
  \tau_z \leq \frac{\Ga}{\Ga+2}
  + O(e^{-a_p z}),
\end{equation}
the relation~\eqref{Tatz} reads
\begin{equation} \label{Tatz2}
  \rho_z = 1 - \tau_z + O(e^{-\frac{z}{2}}),
\end{equation}
while the bound~\eqref{Tupper} reads
\[
  \frac{\Ga}{2} \rho_z^2 \leq \tau_z + 
   O(e^{-\frac{p-1}{2} z}).
\]
Combining these last two gives
\[
  (\tau_z-1)^2 - \frac{2}{\Ga} \tau_z \lec 
  e^{-\frac{z}{2}}.
\]
The quadratic in $\tau_z$ on the left has roots
\[
  1 + \frac{1}{\Ga} \pm \left(\frac{2}{\Ga} + \frac{1}{\Ga^2}  \right)^{\frac{1}{2}},
\]
and so we obtain a lower bound
\[
  \tau_z \geq 1 + \frac{1}{\Ga} - 
  \left(\frac{2}{\Ga} + \frac{1}{\Ga^2}  \right)^{\frac{1}{2}}
  - O( e^{-\frac{z}{2}}).
\]
Combined with~\eqref{improved2}, this gives~\eqref{nutight}.
Inserting~\eqref{nutight} into~\eqref{Tatz2} yields~\eqref{Ttight}.
\begin{remark}
For small $\Ga$, 
\[
  1 + \frac{1}{\Ga} - 
  \left( \frac{2}{\Ga} + \frac{1}{\Ga^2}  \right)^{\frac{1}{2}}
  = \frac{\Ga}{2} - \frac{\Ga^2}{2} + O(\Ga^3), 
\]
while
\[
  \frac{\Ga}{\Ga+2} = \frac{\Ga}{2} - \frac{\Ga^2}{4} + O(\Ga^3),
\]
so the bounds~\eqref{nutight} are nearly sharp.
\end{remark}

Finally, we turn to the $z$-derivative estimates~\eqref{evz} and~\eqref{efz}.
To handle the $z$ dependence in the delta potential,
it is convenient to translate. Set
\[
  \Tt_{z}(y) = T_{z} \left( y-\frac{z}{2} \right),
\]
which satisfies the (translated) eigenvalue equation
\begin{equation} \label{eigtrans}
  \left( -\p_y^2 + 1 - p Q^{p-1}(\cdot-\frac{z}{2}) + \Ga \de - \nu_{z} \right) \Tt_{z} = 0. 
\end{equation}
Differentiating in $z$ gives
\[
\begin{split}
  & \left( -\p_y^2 + 1 - p Q^{p-1}(\cdot-\frac{z}{2}) + \Ga \de - \nu_{z} \right) \p_z \Tt_{z} \\ & 
  \qquad \qquad =  
  \left( -\frac{p}{2}(Q^{p-2} Q')(\cdot-\frac{z}{2}) + \p_z \nu_{z} \right) \Tt_{z}.
\end{split}
\]
Taking inner-product with $\Tt_{z}$ and using~\eqref{eigtrans} gives
\begin{equation} \label{eigtrans2}
  0 = -\frac{p}{2} \int (Q^{p-2}Q')(y-\frac{z}{2})
  \Tt_{z}^2(y) dy + \p_z \nu_{z} \| \Tt_{z} \|_2^2,
\end{equation}
so, since $Q^{p-2}(Q')^3$ is odd, recalling 
$g(y) = T_{z}(y) - Q'(y)$, and using~\eqref{pointwise},
\[
\begin{split}
  \| Q' \|_2^2 &\; \p_z \nu_{z} =
  \frac{p}{2} \int Q^{p-2} Q' T^2_{z} =
  \frac{p}{2} \int Q^{p-2} Q' \left( T^2_{z} - (Q')^2 \right) 
  \\ & \lec \int Q^{p-2} |Q'| 
  \left| T_{z} + Q' \right| |g| \\& \lec  \sqrt{\Ga} \left( 
  \int_{|y| \geq \frac{z}{2}} e^{-(p-1)|y|}(e^{-\frac{z}{2}} + e^{-|y|}) e^{-\frac{z}{2}} dy +
  \int_{-\frac{z}{2}}^0 e^{-z} (e^{-\frac{z}{2}} + e^{y}) e^{(p-2)y} dy \right. \\
  & \left. \quad +
  \int_0^{\frac{z}{2}} e^{-z} (e^{-\frac{z}{2}} + e^{-y})  e^{-(p-1)y} \left((y+z)e^{-y} + e^{-\frac{z}{2}} \right) dy \right) \lec \sqrt{\Ga} e^{-z},
\end{split}
\]
establishing~\eqref{evz}. To show~\eqref{efz}, we consider the translated version of the equation for
\[
  \gt := g(\cdot-\frac{z}{2}) = \Tt_{z} - Q'(\cdot - \frac{z}{2}):
\]
\[
  \left( -\p_y^2 + 1 - p Q^{p-1}(\cdot-\frac{z}{2}) + \Ga \de - \nu_{z} \right) \gt
  = \left(-\Ga \delta + \nu_{z} \right) Q'(\cdot-\frac{z}{2}).
\]
Differentiating here with respect to $z$ produces
\begin{equation} \label{dzgtilde}
\begin{split}
  &\left( -\p_y^2 + 1 - p Q^{p-1}(\cdot-\frac{z}{2}) + \Ga \de - \nu_{z} \right)
  \p_z \gt = \frac{1}{2} \left(\Ga\de - \nu_{z} \right) Q''(\cdot-\frac{z}{2}) \\ & \qquad \qquad + (\p_z \nu_{z}) \left( Q'(\cdot - \frac{z}{2}) + \gt
  \right)  -\frac{p}{2}(Q^{p-2}Q')(\cdot-\frac{z}{2}) \gt,
\end{split}
\end{equation}
and then translating back again, in terms of the function
\begin{equation} \label{ghat}
  \hat g := (\p_z \gt)(\cdot + \frac{z}{2}) = 
  \p_z g - \frac{1}{2} \p_y g = \p_z T_{z} - \frac{1}{2} \p_y g, 
\end{equation}
\[
  \left( L^+ + \Ga \de_{-\frac{z}{2}} - \nu_{z} \right) \hat g = 
  \frac{1}{2} \left(\Ga \delta_{-\frac{z}{2}} - \nu_{z} \right) Q'' 
  + (\p_z \nu_{z}) \left( Q' + g \right)  
  -\frac{p}{2} Q^{p-2} Q' g.
\]
Taking inner product with $\hat g$ yields
\[
\begin{split}
  &( \hat g, \; L^+ \hat g) + \Ga |\hat g(-z/2)|^2 =
  \frac{1}{2} \Ga \hat g(-z/2) Q''(-z/2) - \frac{1}{2} \nu_{z} (\hat g, Q'') \\
  & \qquad \qquad + (\p_z \nu_{z}) (\hat g, Q' + g) - \frac{p}{2}(\hat g, Q^{p-2} Q' g) + \nu_z \| \hat g \|_2^2 \\
  &\qquad \lec \Ga e^{-\frac{z}{2}} |\hat g(-z/2)| + \left( |\nu_{z}|
  + |\p_z \nu_{z}|(1 + \|g\|_{L^2}) + 
  \|Q^{p-2} Q' g\|_{L^2}  \right) \| \hat g \|_{L^2} \\
  &\qquad \lec \Ga e^{-\frac{z}{2}} |\hat g(-z/2)| + \sqrt{\Ga} e^{-\frac{z}{2}} \| \hat g \|_{L^2}
\end{split}
\]
using~\eqref{esest},~\eqref{evz} and~\eqref{esefest}.
Decomposing orthogonally,
\[
  \hat g = (\phi_0, \hat g) \phi_0 + \frac{(Q', \hat g)}{\| Q' \|_2^2} Q' + \hat g^{\perp}, \qquad \hat g^\perp \perp \phi_0, Q',
\]
and using~\eqref{es},
\[
  \| \hat g^\perp \|_{H^1}^2  +  \Ga |\hat g(-z/2)|^2
  \lec |(\phi_0,\hat g)|^2 + \Ga e^{-\frac{z}{2}} |\hat g(-z/2)| + \sqrt{\Ga} e^{-\frac{z}{2}} \| \hat g \|_{L^2}.
\]
The $\hat g(-z/2)$ term on the right can be absorbed into its counterpart on the left by Young's inequality, leaving
\[
 \| \hat g^\perp \|_{H^1}^2 +  \Ga |\hat g(-z/2)|^2
  \lec |(\phi_0,\hat g)|^2 + \sqrt{\Ga} e^{-\frac{z}{2}} \| \hat g \|_{L^2} + \Ga e^{-z},
\]
and so
\[
 \| \hat g \|_{H^1}^2 +  \Ga |\hat g(-z/2)|^2
  \lec |(\phi_0,\hat g)|^2 + |(Q', \hat g)|^2
  + \sqrt{\Ga} e^{-\frac{z}{2}} \| \hat g \|_{L^2} + \Ga e^{-z}.
\]
Then by Young's inequality again,
\begin{equation} \label{ghatbound}
   \| \hat g \|_{H^1}^2 +  \Ga |\hat g(-z/2)|^2
  \lec |(\phi_0,\hat g)|^2 + |(Q', \hat g)|^2 + \Ga e^{-z},
\end{equation}
and it remains to estimate the two inner products on the right.
By~\eqref{esefest},
\begin{equation} \label{gprime}
  |(\phi_0, g')| + |(\p_y Q, g')| \lec \| g'\|_2 \lec \sqrt{\Ga} e^{-\frac{z}{2}}.
\end{equation}
By the normalization~\eqref{normal},
\[
  0 = \p_z \| T_{z} \|_2^2 = 2( T_{z}, \p_z T_{z} )
  = 2 (Q' + g, \p_z T_{z} ),
\]
so
\[
  |( Q',  \p_z T_{z} )| = |(g,  \p_z T_{z} )| =
  |(g,  \hat g - \frac{1}{2} g' )| 
  = |(g, \hat g)| \lec \sqrt{\Ga} e^{-\frac{z}{2}} 
  \| \hat g \|_2 ,
\]
which combined with~\eqref{gprime} gives
\begin{equation} \label{Qprimeorth}
  |( Q', \hat g)| \lec \sqrt{\Ga} e^{-\frac{z}{2}} (1 + \| \hat g \|_2).
\end{equation}
Using the eigenvalue equation~\eqref{eigen},
\[
  \nu_{z} (\phi_0, T_{z} ) = (\phi_0, (L^+ + \Ga \de_{-\frac{z}{2}}) T_{z} ) = -e_0 (\phi_0, T_{z}) + \Ga\phi_0(-z/2) T_{z}(-z/2). 
\]
We differentiate this expression with respect to $z$, observing that
\[
  T_z(-z/2) = Q'(-z/2) + \gt(0) \; \implies \;
  \frac{d}{dz} T_z(-z/2) = -\frac{1}{2} Q''(-z/2) + \hat g(-z/2),
\]
to obtain
\[
\begin{split}
  (e_0 + \nu_{z})(\phi_0, \p_z T_{z}) &= -(\p_z \nu_{z})(\phi_0,T_{z}) \\
  &\quad - \frac{\Ga}{2} \phi_0'(-z/2) T_{z}(-z/2) \\
  &\quad + \Ga \phi_0(-z/2)
  (-\frac{1}{2} Q''(-z/2) + \hat g(-z/2)).
\end{split}
\]
Since $e_0 + \nu_{z} = e_0 + O(e^{-z})$,
\[
  |(\phi_0, \p_z T_{z})| \lec \sqrt{\Ga} e^{-z} + 
  \Ga e^{-\frac{p-1}{2}z}[e^{-\frac{z}{2}} + |\hat g(-z/2)|] 
\]
using~\eqref{asy}, \eqref{esefest}, \eqref{evz} and~\eqref{gsexp}.
Combining this with~\eqref{gprime} yields
\begin{equation} \label{phiorth}
  |( \phi_0, \hat g)| \lec 
  \sqrt{\Ga} e^{-\frac{z}{2}} +
  \Ga e^{-\frac{p-1}{2}z}[e^{-\frac{z}{2}} + |\hat g(-z/2)|].
\end{equation}
Inserting~\eqref{Qprimeorth} and~\eqref{phiorth} into~\eqref{ghatbound} gives
\[
  \| \hat g \|_{H^1}^2 +  \Ga |\hat g(-z/2)|^2
  \lec  \Ga^2 e^{-(p-1) z} [e^{-z} + |\hat g(-z/2)|^2] + 
  \Ga e^{-z}(1 + \|\hat g\|_2^2),
\]
from which follows
\[
  \| \hat g \|_{H^1} \lec \sqrt{\Ga} e^{-\frac{z}{2}}.
\]
Finally, from~\eqref{ghat} and~\eqref{esefest},
\[
  \| \p_z T_{z} \|_2 \lec \| \hat g \|_2 + \| g' \|_2
  \lec \sqrt{\Ga} e^{-\frac{z}{2}},
\]
establishing~\eqref{efz}.
\end{proof}

%%%%%%%%%%%%%%%%%%%%%%%%%%%%%%%%%%%%%%%%%%
\section{Proof of the main force law estimate} 
\label{motionproof}

Here we prove Proposition~\ref{Einnerest2}.
\begin{proof}
Referring to~\eqref{error3} and~\eqref{error}, 
we begin with the contributions
\begin{equation} \label{chiE}
\begin{split}
  \left\lan \chi \cE_{P^0}^{\Ga=0}, \; \left( e^{i \frac{v}{2}(\cdot)}  T_{z} \right) 
  \left(\cdot- \frac{z}{2} \right) \right\ran &=
  \left\lan \chi G , \; 
  \left( e^{i \frac{v}{2}(\cdot)}  T_{z} \right) 
  \left(\cdot- \frac{z}{2} \right) \right\ran \\
  & - \vec{m} \cdot \left\lan \chi \left( \vec{M} Q \right) \left( \cdot - \frac{z}{2} \right),
   \; T_{z} \left(\cdot- \frac{z}{2} \right) \right\ran \\
  &+  \Om\vec{m} \cdot \left \lan \chi
  \left( e^{-i \frac{v}{2}(\cdot)} \vec{M} Q \right)
  \left( \cdot + \frac{z}{2} \right), \; 
  \left( e^{i \frac{v}{2}(\cdot)}  T_{z} \right) 
  \left(\cdot- \frac{z}{2} \right) \right\ran.
\end{split}  
\end{equation}
For the first term in~\eqref{chiE}, we ~\eqref{interform},\eqref{intersize},\eqref{esefest}
and~\eqref{asy} to get
\[
\begin{split}
  &\left| \left\lan \chi G , \; 
  \left( e^{i \frac{v}{2}(\cdot)}  T_{z} \right) 
  \left(\cdot- \frac{z}{2} \right) \right\ran - H(z) \right|
  \lec \left| \left\lan G , \; 
  \left( e^{i \frac{v}{2}(\cdot)}  Q' \right) 
  \left(\cdot- \frac{z}{2} \right) \right\ran - H(z) \right| \\
  & \qquad \quad + 
  \left| \left\lan (1-\chi) G , \; 
  \left( e^{i \frac{v}{2}(\cdot)}  Q' \right) 
  \left(\cdot- \frac{z}{2} \right) \right\ran \right|
  + \left| \left\lan \chi G , \; 
   e^{i \frac{v}{2}(\cdot)} \left( T_{z} - Q' \right) 
  \left(\cdot- \frac{z}{2} \right) \right\ran \right| \\
  & \qquad \lec e^{-z} \left(v^2 z^2 + e^{-\frac{z}{2}} \right) + \| G \|_\infty \|(1-\chi) Q'(\cdot-\frac{z}{2}) \|_1 + \| G \|_{\infty} \| T_z - Q' \|_1 \\
   & \qquad \lec e^{-z} \left(v^2 z^2 + e^{-\frac{z}{2}} \right).
\end{split}
\]
For the remaining terms in~\eqref{chiE} we use the following
inner-product estimates, which include replacements for the exact expressions
\begin{equation} \label{inners}
  \lan i \La Q, Q' \ran = \lan i Q', Q' \ran = \lan Q, Q' \ran = 0,
  \quad \lan -yQ, Q' \ran = M(Q),
\end{equation}
when $Q'$ is replaced with $T_{z}$:
\begin{lemma}
We have the following:
\begin{equation} \label{inners2}
\begin{split}
  &\lan i \La Q, T_{z} \ran = \lan i Q', T_{z} \ran = 0, \\ 
  &|\lan Q, T_{z} \ran| + |\lan -yQ, T_{z} \ran - M(Q)| 
   \lec z^2 e^{-z}, \\
  &\left|  \left\lan  \chi \left(e^{-i \frac{v}{2}(\cdot)} \vec{M} Q \right)(\cdot + z), \;
  e^{i \frac{v}{2}(\cdot)} T_{z}(\cdot) \right\ran \right| \lec z^2 e^{-z}.
\end{split}
\end{equation}
\end{lemma}
\begin{proof}
The first two relations follow immediately from the 
real-valuedness of $T_z$.
In light of~\eqref{asy} and~\eqref{inners}, to bound the 
third and fourth expressions, it suffices to show that
\[
  \int_{-\infty}^\infty (1 + |y|) e^{-|y|} | T_{z}(y) - Q'(y) | dy
  \lec z^2 e^{-z},
\]
which follows directly from~\eqref{pointwise}:
\[
%\begin{split}
  \int_{0}^{\infty} (1 + |y|) e^{-|y|} | T_{z}(y) - Q'(y) | dy
  \lec  \int_0^\infty (1 + y) e^{-y}   e^{-z} \left( (y+z) e^{-y} + e^{-\frac{z}{2}} \right) dy
  % & \lec e^{-z} \left( \int_0^\infty (1+y) y e^{-2y} dy + 
  % z \int_0^\infty (1+y) e^{-2y} dy + e^{-\frac{z}{2}} \int_0^\infty (1+y)e^{-y} dy \right) 
  \lec z e^{-z},
%\end{split}
\]
\[
\begin{split}
  \int_{-\frac{z}{2}}^{0} & (1 + |y|) e^{-|y|} | T_{z}(y) - Q'(y) | dy
  \lec  \int_{-\frac{z}{2}}^0 (1 - y) e^{y}  e^{-z} e^{-y}  dy \\
  &= e^{-z}  \int_{-\frac{z}{2}}^0 (1-y) dy \lec z^2 e^{-z},
\end{split}
\]
and
\[
\begin{split}
  \int_{-\infty}^{-\frac{z}{2}} (1 + |y|) e^{-|y|} & | T_{z}(y) - Q'(y) | dy
  \lec  \int_{-\infty}^{-\frac{z}{2}} (-y) e^{y}  e^{\sqrt{1-\nu_{z}} y} dy \\
  &\lec z e^{-\frac{z}{2}(1 + \sqrt{1-\nu_{z}})} \lec z e^{-z},
\end{split}
\]
since by~\eqref{nutight}, 
\[
  e^{-\frac{z}{2}\sqrt{1-\nu_{z}}}
  = e^{-\frac{z}{2}(1 + O(e^{-z}))} = e^{-\frac{z}{2}} e^{O(z e^{-z})}
  \lec  e^{-\frac{z}{2}}.
\]
To bound the final quantity of~\eqref{inners2}, we combine the direct computation
\[
   \left|  \left\lan  \chi \left(e^{-i\frac{v}{2}(\cdot)} \vec{M} Q \right)(\cdot + z), \;
  e^{i \frac{v}{2}(\cdot)} Q'(\cdot) \right\ran \right| \lec 
  \int_{-\infty}^\infty (1 + |y+z|) e^{-|y+z|} e^{-|y|} dy \lec z^2 e^{-z}
\]
with the following consequence of~\eqref{pointwise}:
\[
\begin{split}
 &\left|  \left\lan  \left(e^{-i \frac{v}{2} (\cdot)} \vec{M} Q \right)(\cdot + z), \;
  e^{i \frac{v}{2}(\cdot)} (T_{z} - Q')(\cdot) \right\ran \right| \\ & \qquad \lec  
  \int_{-\infty}^\infty (1 + |y+z|) e^{-|y+z|}   | T_{z}(y) - Q'(y) | dy
   \\ & \qquad \lec  
  \int_{-\infty}^{-z} (1-y-z) e^{y+z} e^{\sqrt{1-\nu_{z}} y }  dy
   \\  & \qquad \quad +  
  \int_{-z}^{-\frac{z}{2}} (1 + y+z) e^{-y-z} e^{\sqrt{1-\nu_{z}} y } dy
  \\ & \qquad \quad +  
  \int_{-\frac{z}{2}}^0 (1 + y+z) e^{-y-z} e^{-z} e^{-y} dy
   \\ & \qquad \quad +  
  \int_{0}^\infty (1 + y+z) e^{-y-z}  e^{-z} \left( (y+z) e^{-y} + e^{-\frac{z}{2}} \right)   dy
  \\ & \qquad \lec z^2 e^{-z}
\end{split}
\]
by direct computation.
\end{proof}
Returning to the second term of~\eqref{chiE}, we have, from~\eqref{inners2},
\[
   \left| \left\lan \vec{M} Q, \; T_{z} \right \ran 
   + M(Q) \hat e_4 \right| \lec z^2 e^{-z},
\]
and since
\[
\begin{split}
   &\left| \left\lan (1 - \chi) \left( \vec{M} Q \right) \left(\cdot - \frac{z}{2} \right),
   T_z\left(\cdot- \frac{z}{2} \right) \right \rangle \right| \\
   & \quad \lec  \left| \left \langle (1 - \chi) \left( \vec{M} Q \right) \left(\cdot - \frac{z}{2} \right),
   Q'\left(\cdot- \frac{z}{2} \right) \right \rangle \right| + 
   \left\|  (1 - \chi) \left( \vec{M} Q \right) \left(\cdot - \frac{z}{2} \right)
   \right\|_2 \| T_z - Q' \|_2
   \\ & \quad \lec \int_{-\frac{z}{2}-2}^{-\frac{z}{2}+2} (1+|y|) e^{-2|y|} dy 
   + \left( \int_{-\frac{z}{2}-2}^{-\frac{z}{2}+2} (1+|y|) e^{-|y|} dy \right)^{\frac{1}{2}} e^{-\frac{z}{2}}
   \\ & \quad \lec z e^{-z}
\end{split}
\]
by~\eqref{esefest} and~\eqref{asy},
we have
\[
  \left| \mv \cdot \left\lan \chi \left( \vec{M} Q \right) \left( \cdot - \frac{z}{2} \right),
   \; T_{z} \left(\cdot- \frac{z}{2} \right) \right\ran 
   + \frac{M(Q)}{2} \left( \dot v - \frac{\dot \la}{\la} v \right) \right|
   \lec |\mv| z^2 e^{-z}.
\]
The last term of~\eqref{chiE} is estimated directly
using the last estimate of~\eqref{inners2}:
\[
\begin{split}
  &\left| \Om\vec{m} \cdot  \left \lan \chi
  \left( e^{-i \frac{v}{2}(\cdot) } \vec{M} Q \right)
  \left(\cdot + \frac{z}{2} \right), \; 
  \left( e^{i \frac{v}{2}(\cdot)}  T_{z} \right) 
  \left(\cdot- \frac{z}{2} \right) \right\ran \right| \\
  & \qquad \lec |\mv|
  \left|  \left\lan  \chi \left(e^{-i \frac{v}{2}(\cdot)} \vec{M} Q \right)(\cdot + z), \;
  e^{i \frac{v}{2}(\cdot)} T_{z}(\cdot) \right\ran \right| 
  \lec |\mv| z^2 e^{-z}.
\end{split}
\]
Collecting the last few estimates into~\eqref{chiE} produces
\begin{equation} \label{chiEcont}
\begin{split}
  &\left| \left\lan \chi \cE_{P^0}^{\Ga=0}, \; \left( e^{i \frac{v}{2}(\cdot)}  T_{z} \right) 
  \left(\cdot- \frac{z}{2} \right) \right\ran - \left[ H(z) 
  + \frac{M(Q)}{2} \left( \dot v - \frac{\dot \la}{\la} v \right) \right]
  \right| \\ & \qquad \lec e^{-z} \left(v^2 z^2 + e^{-\frac{z}{2}} 
  + |\mv| z^2 \right).
\end{split}
\end{equation}
Continuing the proof of Proposition~\eqref{Einnerest2}, 
the next terms to consider are
\begin{equation} \label{chipcont}
\begin{split}
  &\left| \left \langle \chi \left(\chi^{p-1} - 1 \right) |P^0|^{p-1} P^0,   \; 
  \left( e^{i \frac{v}{2}(\cdot)} T_z \right) \left(\cdot - \frac{z}{2} \right) \right \rangle \right| \\
  & \quad \lec \int_{-2}^2 |P^0(y;z,v)|^p dy 
  \sup_{z \in [-2,2]} | T_z |
  \lec e^{-p \frac{z}{2}} e^{-\frac{z}{2}} = e^{-z} e^{-\frac{p-1}{2} z}
\end{split}
\end{equation}
by~\eqref{asy} and~\eqref{pointwise}, and
\begin{equation} \label{chidchicont}
  \begin{split}
  &\left| \frac{\dot \la}{\la} \left \langle (y \p_y \chi) P^0,   \; 
  \left( e^{i \frac{v}{2}(\cdot)} T_{z} \right) \left(\cdot - \frac{z}{2} \right) \right \rangle \right| \\
  & \quad \lec |\mv| \int_{-2}^2 |P^0(y;z,v)| dy 
  \sup_{z \in [-2,2]} | T_z |
  \lec |\mv| e^{-\frac{z}{2}} e^{-z},
\end{split}
\end{equation}
again by~\eqref{asy} and~\eqref{esefest}.
To complete the proof of Proposition~\ref{Einnerest2},
it remains to show:
\begin{lemma}  \label{deltacomp}
\[
\begin{split}
  &  \left| \left\lan (2 \p_y \chi \p_y
  + \p_y^2 \chi_\rho) P^0, \;
  \left( e^{i \frac{v}{2}(\cdot)}  T_{z} \right)\left(\cdot - \frac{z}{2} \right)
  \right\ran  + 2 \Ga c_p e^{-\frac{z}{2}} T_{z}\left(-\frac{z}{2} \right) \right| \\
  & \qquad \lec e^{-z} \left( e^{-z} + e^{-\frac{p-1}{2} z} + v^2 \right).
\end{split}
\]
\end{lemma}
\begin{proof}
Recalling the definition~\eqref{P} of $P^0$, we start with the
contribution from its first term.
Changing variable and integrating by parts twice, using
\[
  \chi'(-2) = \chi'(2) = 0, \qquad
  \chi(-2) = \chi(2) = 1, \qquad
  \chi(0) = 1,
\]
and that $T_z$ is real-valued, we find
\[
\begin{split}
  &\left\lan (2 \p_y \chi_\rho \p_y
  + \p_y^2 \chi_\rho) \left( e^{i \frac{v}{2}(\cdot)} Q \right)
  \left(\cdot - \frac{z}{2} \right), \;
  \left( e^{i \frac{v}{2}(\cdot)}  T_{z} \right) \left(\cdot - \frac{z}{2} \right)
  \right\ran \\
  &= \int_{-\frac{z}{2}-2}^{-\frac{z}{2} + 2 } 
  \left( 2 \chi'(y + \frac{z}{2}) Q'(y) + 
  \chi''(y + \frac{z}{2}) Q(y) \right) T_{z}(y) dy \\
  &= \int_{-\frac{z}{2}-2}^{-\frac{z}{2} + 2} 
  (\chi- 1)'(y + \frac{z}{2}) \left( Q'(y)  T_{z}(y) - 
  Q(y) T_{z}'(y) \right) dy \\
   &=  \int_{-\frac{z}{2}-2}^{-\frac{z}{2} + 2} 
  (\chi - 1)(y + \frac{z}{2}) \left( Q(y)  T_{z}''(y) - 
  Q''(y) T_{z}(y) \right) dy.
\end{split}
\]
Using the equations~\eqref{ode} and~\eqref{eigen}, this integral
becomes
\[
\begin{split}
  &\int_{-\frac{z}{2}-2}^{-\frac{z}{2} + 2} 
  (\chi-1)(y + \frac{z}{2}) \left( Q(y)  T_{z}''(y) - 
  Q''(y) T_{z}(y) \right) dy \\
  & \qquad = \int_{-\frac{z}{2}-2}^{-\frac{z}{2} + 2} 
  (1-\chi) (y + \frac{z}{2}) \left( 
  -\Ga \delta_{-\frac{z}{2}} + \nu_{z} +
  (p-1) Q^{p-1}(y))  \right) Q(y) T_{z}(y)  dy \\
  & \qquad = -\Ga Q \left(-\frac{z}{2}\right) T_{z}\left(-\frac{z}{2}\right) +
  O\left( (e^{-(p-1)\frac{z}{2}} + e^{-z}) e^{-z} \right),
\end{split}
\]
which using~\eqref{asy} shows
\begin{equation} \label{P1cont}
\begin{split}
  &\left| \left\lan (2 \p_y \chi_\rho \p_y
  + \p_y^2 \chi_\rho) , \left( e^{i \frac{v}{2}(\cdot)} Q \right)
  \left(\cdot - \frac{z}{2} \right), \;
  \left( e^{i \frac{v}{2}(\cdot)}  T_{z} \right) \left(\cdot - \frac{z}{2} \right)
  \right\ran \right. \\
  & \qquad \left. + \Ga c_p e^{-\frac{z}{2}}  T_{z}\left(-\frac{z}{2}\right) \right| \lec e^{-z} (e^{-\frac{p-1}{2}z} + e^{-z} + v^2),
\end{split}
\end{equation}
using~\eqref{asy}, \eqref{nutight} and~\eqref{pointwise}.
For the contribution from the second term of $P^0$, we have
\[
\begin{split}
  &\left\lan (2 \p_y \chi_\rho \p_y
  + \p_y^2 \chi_\rho) \left( e^{-i \frac{v}{2}(\cdot)} Q \right)
  \left(\cdot + \frac{z}{2} \right), \;
  \left( e^{i \frac{v}{2}(\cdot)}  T_{z} \right) \left(\cdot - \frac{z}{2} \right) \right\ran  \\
  &= \re \int_{-\frac{z}{2}-2}^{-\frac{z}{2} + 2 } e^{i\frac{v}{2}(2y+z)}
  \left( 2 \chi'(y + \frac{z}{2}) \left[Q'(y+z) + i\frac{v}{2} Q(y+z) \right] + 
  \chi''(y + \frac{z}{2}) Q(y+z) \right) T_{z}(y) dy \\
  &=  \int_{-\frac{z}{2}-2}^{-\frac{z}{2} + 2 } \cos( v(y+\frac{z}{2}) )
   \left( 2 \chi'(y + \frac{z}{2}) Q'(y+z)  + 
  \chi''(y + \frac{z}{2}) Q(y+z) \right) T_{z}(y) dy \\
  & \quad - v \int_{-\frac{z}{2}-2}^{-\frac{z}{2} + 2 } \sin( v(y+\frac{z}{2}) )
  \chi'(y + \frac{z}{2}) Q(y+z) T_z(y) dy \\
   &=  \int_{-\frac{z}{2}-2}^{-\frac{z}{2} +2}
   \left( 2 \chi'(y + \frac{z}{2}) Q'(y+z)  + 
  \chi''(y + \frac{z}{2}) Q(y+z) \right) T_{z}(y) dy
  + O(v^2 e^{-z})
\end{split}
\]
by~\eqref{asy} and~\eqref{pointwise}. Integrating by parts and
using the equations~\eqref{ode} and~\eqref{eigen} just 
as above, the integral here becomes
\[
\begin{split}
  &\int_{-\frac{z}{2}-2}^{-\frac{z}{2} +2}
  \left( 2 \chi'(y + \frac{z}{2}) Q'(y+z)  + 
  \chi''(y + \frac{z}{2}) Q(y+z) \right) T_{z}(y) dy \\
  &=  \int_{-\frac{z}{2}-2}^{-\frac{z}{2} + 2} 
  (1-\chi) (y + \frac{z}{2}) \left( 
  -\Ga \delta_{-\frac{z}{2}} + \nu_{z} +
  (p-1) Q^{p-1}(y+z))  \right) Q(y+z) T_{z}(y)  dy \\
  &= -\Ga Q \left(\frac{z}{2}\right) T_{\Ga,z}\left(-\frac{z}{2}\right) +
  O\left( (e^{-(p-1)\frac{z}{2}} + e^{-z}) e^{-z} \right),
\end{split}
\]
and so
\begin{equation} \label{P2cont}
\begin{split}
  &\left| \left\lan (2 \p_y \chi_\rho \p_y
  + \p_y^2 \chi_\rho) , \left( e^{-i \frac{v}{2}(\cdot)} Q \right)
  \left(\cdot + \frac{z}{2} \right), \;
  \left( e^{i \frac{v}{2}(\cdot)}  T_{z} \right) \left(\cdot - \frac{z}{2} \right)
  \right\ran \right. \\
  & \qquad \left. + \Ga c_p e^{-\frac{z}{2}}  T_{\Ga,z}\left(-\frac{z}{2}\right) \right| \lec e^{-z} (e^{-\frac{p-1}{2}z} + e^{-z}).
\end{split}
\end{equation}
Adding together~\eqref{P1cont} and~\eqref{P2cont} 
completes the proof of Lemma~\ref{deltacomp}.
\end{proof}
Proposition~\ref{Einnerest2} now follows from collecting the estimates~\eqref{chiEcont},~\eqref{chipcont},~\eqref{chidchicont}
and Lemma~\ref{deltacomp}.
\end{proof}

%---------------
\appendix
%---------------

%%%%%%%%%%%%%%%%%%%%%%%%%%%%%%%%%%%%%%%%%%
\section{Proof of the modulation lemma} 
\label{modproof}

Here we prove Lemma~\ref{modulation}.
\begin{proof}
Given a solution $u$ as in~\eqref{givensol}, satisfying~\eqref{endparam} and~\eqref{endclose}, we express the three differentiated orthogonality conditions~\eqref{difforthos} by using the 
equation~\eqref{eta} as follows:
\begin{equation} \label{contributions}
\begin{split}
  0 &= \frac{d}{ds} \left\lan \left[ \begin{array}{c} 
  Q \\ y Q \\ i \La Q \end{array} \right],
  \eta(s,\cdot) \right\ran = 
  \left\lan \left[ \begin{array}{c} 
  iQ \\ iy Q \\ -\La Q \end{array} \right],
  i \p_s \eta(s,\cdot) \right\ran \\
  &=  \left\lan \left[ \begin{array}{c} 
  iQ \\ iy Q \\ -\La Q \end{array} \right],
  (-\p_y^2 + \la \Ga \de_{-\frac{z}{2}} + 1) \eta
  - |P_1+\eta|^{p-1}(P_1+\eta) + |P_1|^{p-1} P_1 \right. \\
  & \qquad \qquad \qquad \qquad 
  \left. - \mv \cdot \Mv \eta - \cE_{P_1} \right\ran. 
\end{split}
\end{equation}
Re-expressing $\eta$
in terms of $u$ (and the parameters $\vec{q}$) via~\eqref{xi}-\eqref{eta} to get
\begin{equation} \label{etau}
  \eta(s,y) = e^{-i\frac{v}{2} y} \left[
  e^{-i \gat} \la^{\frac{2}{p-1}} u(t,\la(y + \frac{z}{2})) - P(y + \frac{z}{2}; z, v) \right],
\end{equation}
and adding in the motion law~\eqref{vdyn2},
produces, we claim, a system of four equations of the form
\begin{equation} \label{odesys}
  \Upsilon \; \frac{d}{ds} \vec{p} = \vec{f}, \qquad
  \vec{p} = \left[ \begin{array}{c} \log \la \\
  \frac{z}{2} \\ \gat \\ \frac{v}{2} \end{array} \right]
\end{equation}
where 
\begin{equation} \label{regularity}
  \Upsilon = \Upsilon(\vec{p},t) \mbox{ and } 
  \vec{f} = \vec{f}(\vec{p},t) \mbox{ are } C^1 
  \mbox{ in } \vec{p} \mbox{ and continuous in } t,
\end{equation}
and
\begin{equation} \label{invert}
  0 < \frac{1}{z} +  |v| z +
  \| \eta \|_{H^1} < 2\epsilon_0   \;\; \implies \;\; 
  \| \Upsilon^{-1} \| \lec 1
\end{equation}
(if $\epsilon_0$ chosen small enough).
We give just a rough sketch of this claim.

First, the main part of $\Upsilon$ comes from the
following part of the contribution of $\cE_{P_1}$ to~\eqref{contributions}:
\[
\begin{split}
  &\left\lan \left[ \begin{array}{c} 
  iQ \\ iy Q \\ -\La Q \end{array} \right],
  -\chi(\cdot + \frac{z}{2}) \mv \cdot \Mv Q \right\ran
  = -\mv \cdot \left( \left\lan \left[ \begin{array}{c} 
  iQ \\ iy Q \\ -\La Q \end{array} \right],
  \Mv Q \right\ran + O(z^2 e^{-z}) \right).
\end{split}
\]
Using the inner-product relations~\eqref{inners}
and~\eqref{inners3}, and adding in the 
motion law~\eqref{vdyn2}, $\dot v = - \Ht(z)$,
these terms produce
\[
  \left( A + O(z^2 e^{-z}) \right) \mv
  + \left[ \begin{array}{c} 
  0 \\ 0 \\ 0 \\ \frac{\Ht(z)}{2} \end{array} \right],
  \qquad A =  M(Q) \left[ \begin{array}{cccc} 
  \frac{p-5}{p-1} & 0 & 0 & 0 \\ 
  0 & -1 & 0 & 0 \\ 
  0 & 0 & \frac{5-p}{p-1} & 0 \\ 
  \frac{v}{2} & 0 & 0 & 1 \end{array} \right] \mv
\]
in the right side of~\eqref{contributions}. 
Observing that
\[
  \mv = B \frac{d}{ds} \vec{p} +
  \left[ \begin{array}{c} 0 \\ -v \\ \frac{v^2}{4} - 1 \\ 0 \end{array} \right], \qquad
  B = \left[ \begin{array}{cccc} 
  1 & 0 & 0 & 0 \\ \frac{z}{2} & 1 & 0 & 0 \\
  \frac{vz}{4} & -\frac{v}{2} & 1 & 0 \\ 
  -\frac{v}{2} & 0 & 0 & 1 \end{array} \right],
\]
we see that these terms produce
\[
  \left( AB + O(z^2 e^{-z}) \right) \frac{d}{ds} \vec{p} + \left( A + O(z^2 e^{-z}) \right) 
   \left[ \begin{array}{c} 0 \\ -v \\ \frac{v^2}{4} - 1 \\ 0 \end{array} \right] + 
   \left[ \begin{array}{c} 
  0 \\ 0 \\ 0 \\ \frac{\Ht(z)}{2} \end{array} \right]
\]
in~\eqref{contributions}, with the first term contributing to the matrix $\Upsilon$ on the left side of~\eqref{odesys}, and the rest to its right side.
The uniform invertibility~\eqref{invert} of $\Upsilon$
will be a consequence of the easily verified
(uniform) invertibility of the matrix $AB$.
The regularity~\eqref{regularity} of these 
contributions is straightforward to verify, 
except for the non-trivial fact that 
the force law $\tilde H(z)$ 
(in particular the contribution $T_z(-\frac{z}{2})$ to it) is $C^1$ in $z$, which nonetheless follows from the
estimates in the proof of~\eqref{efz} given
in Section~\ref{properties}. More precisely:
differentiating~\eqref{eigtrans2} in $z$
and using~\eqref{efz} shows the boundedness of
$\p_z^2 \nu_z$. Then differentiating~\eqref{dzgtilde} in $z$ and estimating $\p_z^2 \tilde g$ just as in the
subsequent estimates of $\hat g$ (hence $\p_z \tilde g$) shows the boundedness
of $\| \p_z^2 \tilde g \|_{H^1}$, from 
which $T_z(-z/2) \in C^1_z$ follows.

The remaining contributions to $\Upsilon$ in~\eqref{odesys} come from the terms
\[
  \left\lan \left[ \begin{array}{c} 
  iQ \\ iy Q \\ -\La Q \end{array} \right],
  -\chi(y+\frac{z}{2}) e^{i \frac{v}{2} z} e^{-i v y } \Om \mv \cdot \Mv Q(y+z) \right\ran
  = O(z^2 e^{-z}) |\mv|
\]
and
\[
  \left\lan \left[ \begin{array}{c} 
  iQ \\ iy Q \\ -\La Q \end{array} \right],
  -\frac{\dot \la}{\la} e^{-i\frac{v}{2}y} (y \p_y \chi)(y+\frac{z}{2}) P^0(y + \frac{z}{2}) \right\ran
  = O(z^2 e^{-z}) |\mv|
\]
from $\cE_{P_1}$, and
\[
  \left\lan \left[ \begin{array}{c} 
  iQ \\ iy Q \\ -\La Q \end{array} \right],
  -\mv \cdot \Mv \eta \right\ran
  = O(\|\eta\|_2 |\mv|)
\]
which all contribute error terms to $\Upsilon$
which do not harm the invertibility~\eqref{invert},
and for which the regularity~\eqref{regularity} is 
easily verified (in the case of the last term, for example, it is ensured by $u \in C_t H^1$).

The remaining contributions to~\eqref{contributions}
from $\cE_{P_1}$ add to $\vec{f}$ in~\eqref{odesys}
and are easily seen to verify~\eqref{regularity}.

The terms in~\eqref{contributions} coming from
$(-\p_y^2 + \la \Ga \de_{-\frac{z}{2}} + 1) \eta$
contribute to $\vec{f}$ in~\eqref{odesys}, and
are handled by observing that since
$\eta(s,\cdot) \in \D_{-\frac{z}{2}}^{\la \Ga}$,
\[
\begin{split}
  &\left\lan \left[ \begin{array}{c} 
  iQ \\ iy Q \\ -\La Q \end{array} \right],
  (-\p_y^2 + \la \Ga \de_{-\frac{z}{2}} + 1) \eta
  \right\ran = 
  \left\lan   (-\p_y^2 + 1) \left[ \begin{array}{c} 
  iQ \\ iy Q \\ -\La Q \end{array} \right],
  \eta \right\ran \\ & \qquad \qquad \qquad - \la \Ga
  \re  \left( \eta(s,-\frac{z}{2}) 
  \left[ \begin{array}{c} 
  iQ \\ iy Q \\ \La Q \end{array} \right]
  (-\frac{z}{2}) \right)
\end{split}
\]
and then expressing $\eta$ in terms of $u$ via~\eqref{etau}, to verify the regularity~\eqref{regularity}. For example,
the last term becomes
\[
  -\la \Ga
  \re  \left(e^{i \frac{vz}{4}} \la^{\frac{2}{p-1}} u(t,0)
  \left[ \begin{array}{c} 
  iQ \\ iy Q \\ \La Q \end{array} \right]
  (-\frac{z}{2}) \right),
\]
whose continuity in $t$ is ensured by $u \in C_t H^1$. 
\begin{remark} \label{modproofrem}
The proof of~\cite[Lemma 5]{Ngu19} treats~\eqref{odesys}, together with the relation~\eqref{ts}, that is 
\begin{equation} \label{ts2}
  \frac{dt}{ds} = \la^2,
\end{equation}
as an autonomous system
for $(\vec{p},t)$ (with independent variable $s$)
which requires the ODE coefficients to be $C^1$
in $t$, In the presence of the delta potential,
however, we obtain terms as above containing $u(t,0)$,
which is continuous in $t$, but not necessarily $C^1$,
since we have only $u \in C_t \D_\Ga \cap C^1_t L^2$. 
Therefore, our argument proceeds differently:
we will first ignore $s$,
and solve a non-autonomous system for $\vec{\hat{p}}(t) = \vec{p}(s)$
(for which mere continuity in $t$ suffices), and then 
define $s = s(t)$ by inverting~\eqref{ts}, and setting
$\vec{p}(s) = \vec{\hat{p}}(t(s))$.
\end{remark}
The regularity~\eqref{regularity} of the 
contributions to $\vec{f}$ in~\eqref{odesys} 
coming from the remaining terms of~\eqref{contributions} is easily checked.  

Using~\eqref{ts2}, and writing
\[
  \vec{\hat p}(t) = \vec{p}(s), \qquad
  \frac{d}{dt} \vec{\hat p} = \la^2(s) 
  \frac{d}{ds} \vec{p} = e^{2 \hat{p}_1(t)} \vec{p}, 
\]
as long as $\| \eta \|_{H^1} < 2\epsilon_0$ and
$z > \frac{2}{\epsilon_0}$ (which holds on some interval by assumptions~\eqref{endparam}
and~\eqref{endclose}) on the final conditions at $t=s_{\mathrm f}$,
we may use~\eqref{invert} to rewrite~\eqref{odesys}
as the system
\[
  \frac{d}{dt} \vec{\hat p} = e^{2 \hat{p}_1} \left(\Upsilon(\vec{\hat p}, t) \right)^{-1} \vec{f}(\vec{\hat p},t).
\]
The regularity~\eqref{regularity} ensured that the Picard-Lindel\"of theorem applies to this system
to provide a unique solution
\[
  \vec{\hat p} \in C^1(\hat I), \qquad
  \vec{\hat p}(s_{\mathrm f}) = \left[ \begin{array}{c} \log \la_{\mathrm f} \\
  \frac{z_{\mathrm f}}{2} \\ \gat_{\mathrm f} \\ \frac{v_{\mathrm f}}{2} \end{array} \right]
\]
on some open interval $\hat I \ni s_{\mathrm f}$, on which also
(by continuity and~\eqref{endparam}-\eqref{endclose})
\[
  |{\hat p}_1(t) = \log \hat \la(t)| +
  \frac{1}{\hat z(t)} + |\hat v(t)| < 2 \epsilon_0.
\]
Finally, we invert~\eqref{ts} to define
\[
  I \ni s(t) = s_{\mathrm f} - \int_t^{s_{\mathrm f}} \frac{d \tau}{\hat\la^2(\tau)}
\]
and set $\vec{p}(s) = \vec{\hat p}(t)$ 
to produce the parameters $\vec{q}(s)$
satisfying the conditions of Lemma~\ref{modulation}.
\end{proof}

%%%%%%%%%%%%%%%%%%%%%%%%%%%%%%%%%%%%%%%%%%%%
\section{Some computations}
\label{comps}

\begin{lemma} \label{Ip}
With $H(z)$ defined as in~\eqref{Hdef} and $p > 2$, we have
\begin{enumerate}
\item
$|H(z) - c_p I_p e^{-z}| \lec e^{-2z} + e^{-\frac{p+1}{2} z}$,
where $I_p = \int_{-\infty}^\infty Q^p(x) e^{-x} dx$;
\item
$I_p = 2 c_p$.
\end{enumerate}
\end{lemma}
\begin{proof}
We use~\eqref{asy} throughout. First,
\[
\begin{split}
 & p\int_{-\frac{z}{2}}^\infty Q^{p-1}(y) Q'(y) Q(y + z) dy
 = \int_{-\frac{z}{2}}^\infty(Q^p(y))' Q(y+z) dy \\
 & \qquad = -Q^p\left(-\frac{z}{2}\right) Q\left(\frac{z}{2}\right)
 -  \int_{-\frac{z}{2}}^\infty Q^p(y) Q'(y+z) dy \\
 & \qquad = \int_{-\frac{z}{2}}^\infty Q^p(y) \left(c_p e^{-(y+z)}
 + O(e^{-p(y+z)}) \right) dy + O(e^{-\frac{p+1}{2}z}) \\ 
 & \qquad = c_p e^{-z} \left( I_p - \int_{-\infty}^{-\frac{z}{2}} Q^p(y) e^{-y} dy \right) + 
 \int_{-\frac{z}{2}}^\infty O(e^{-p|y|} e^{-p(y+z)}) dy + O(e^{-\frac{p+1}{2}z}) \\
 & \qquad = c_p I_p e^{-z} + O\left( e^{-\frac{p+1}{2}z} + z e^{-pz} \right).
\end{split}
\]
Second,
\[
   \left| \int_{-\infty}^{-\frac{z}{2}} Q^{p-1}(y+z) Q'(y) Q(y) dy \right| \lec
   \int_{-\infty}^{-\frac{z}{2}} e^{-p|y+z|} e^{2y} dy
   \lec e^{-2z} + z e^{-pz},
\]
and so we have the first statement of Lemma~\ref{Ip},
using $p > 2$.

From~\eqref{Qform} and~\eqref{asy}, and substituting $y = \frac{p-1}{2} x$:
\[
\begin{split}
  \frac{I_p}{2 c_p} &= 2^{-\frac{p+1}{p-1}}\frac{p+1}{2} \int_\R
  e^{-x} \cosh^{-\frac{2p}{p-1}}((\frac{p-1}{2}) x) dx \\
  &=  2^{-\frac{p+1}{p-1}}\frac{p+1}{p-1} \int_\R
  e^{-\frac{2}{p-1} y} \cosh^{-\frac{2p}{p-1}}(y) dy \\
  &= 2\frac{p+1}{p-1} \int_\R  e^{-\frac{2}{p-1} y}(e^y + e^{-y})^{-\frac{2p}{p-1}} dy
  = 2\frac{p+1}{p-1} \int_\R  e^{2y}(e^{2y} + 1)^{-\frac{2p}{p-1}} dy \\
  &= \frac{p+1}{p-1} \int_0^\infty (z+1)^{-\frac{2p}{p-1}} dz
  = -(z+1)^{-\frac{p+1}{p-1}} |^\infty_0 = 1
\end{split}
\]
substituting $z = e^{2y}$ in the last step, which
gives the second  statement of Lemma~\ref{Ip}.
\end{proof}

\begin{lemma} \label{Pdiff}
For positive $z_1, z_2$ sufficiently large, 
\[
  \frac{|z_1 - z_2|^2}{1 + |z_1 - z_2|^2} \lec \inf_{\omega \in \R, |v| \leq 1} \| P(\cdot; z_1,0) -
  e^{i \omega} P(\cdot; z_2, v) \|^2_{2}
  + e^{-\min(z_1,z_2)}.
\]
\end{lemma}
\begin{proof}
Assuming (without loss of generality) 
$z_1 < z_2$, We have
\[
  \| P(\cdot; z_1,0) -
  e^{i \omega} P(\cdot; z_2, v) \|^2_{2} \geq
  \int_{\frac{z_1}{2}-1}^\frac{z_1}{2}
  |Q(y-\frac{z_1}{2}) - e^{i \Phi}Q(y - \frac{z_2}{2})|^2 dy -  O(e^{-z_1})
\]
where $\Phi = \omega + \frac{v}{2}(y-\frac{z_2}{2})$. Now
\[
  \left| Q(y-\frac{z_1}{2}) - e^{i \Phi}Q(y - \frac{z_2}{2}) \right|^2
  \geq \left( Q(y-\frac{z_1}{2}) - Q(y - \frac{z_2}{2}) \right)^2,
\]
so
\[
\begin{split}
  &\| P(\cdot; z_1,0) -
  e^{i \omega} P(\cdot; z_2, v) \|^2_{2} \geq
  \int_{\frac{z_1}{2}-1}^\frac{z_1}{2}
  \left( Q(y-\frac{z_1}{2}) - Q(y - \frac{z_2}{2}) \right)^2 dy
  -  O(e^{-z_1}) \\
  & \qquad = g(\tilde{z}) -  O(e^{-z_1}), \quad
  g(\tilde{z}) = \int_{-1}^0
  \left( Q(y) - Q(y - \tilde{z}) \right)^2 dy,
  \quad \tilde{z} = \frac{z_2-z_1}{2} > 0.
\end{split}
\]
Now compute
\[
  g(0) = g'(0) = 0, \quad
  g''(0) = 2 c, \quad  c = \int_{-1}^0 (Q'(y))^2 dy > 0 ,
\]
so by Taylor expansion
\[
  g(\tilde{z}) = c \tilde{z}^2 + O(\tilde{z}^3),
  \qquad \tilde z \geq 0,
\]
and since
\[
  g(\tilde z) > 0 \mbox{ for } \tilde z > 0
  \;\; \mbox{ and } \;\;
  \lim_{\tilde z \to \infty} g(\tilde z) = 
  \int_{-1}^0 Q^2(y) dy > 0,
\]
we have
\[
  \frac{\tilde z^2}{1 + \tilde z^2} \lec g(\tilde z), \qquad \tilde z \geq 0.
\]
The lemma follows.
\end{proof}

%%%%%%%%%%%%%%%%%%%%%%%%%%%%%%%%%%%%%%%%%%%%
\section*{Acknowledgment}

The first author is partially supported by NSERC Discovery Grant.
The second author is partially supported by a JSPS Overseas Research Fellowship.

%%%%%%%%%%%%%%%%%%%%%%%%%%%%%%%%%%%%%%%%%%


\begin{thebibliography}{99}
%
% The \bibitem commands: 
% Please follow "Notice to Authors" for referencing.  You 
% must specify bold and italic fonts yourself. 
%


% \bibitem{AkNa13} T. Akahori, H. Nawa, 
% \textit{Blowup and scattering problems for the nonlinear Schr\"{o}dinger equations}, 
% Kyoto J. Math. {\bf 53} (2013), no. 3, 629--672.

% \bibitem{ArIn21}
% A. Ardila, T. Inui,
% \emph{Threshold scattering for the focusing NLS with a repulsive Dirac delta potential}, preprint, arXiv:2108.00248.

% \bibitem{ADM20}
% A. K. Arora, B. Dodson, J. Murphy, 
% \emph{Scattering below the ground state for the 2d radial nonlinear Schr\"{o}dinger equation},
% Proc. Amer. Math. Soc. {\bf 148} (2020), no. 4, 1653--1663.

\bibitem{Ary}
S. Aryan, 
\emph{Existence of two-solitary waves with logarithmic distance for the nonlinear Klein--Gordon equation},
Comm. Contemp. Math. {\bf 24} (2022) no. 1, 2050091.

%\bibitem{BW95} Thomas Bartsch, Zhi Qiang  Wang, 
%\textit{Existence and multiplicity results for some superlinear elliptic problems on $\R^N$}, 
%Comm. Partial Differential Equations 
%{\bf 20} (1995), no. 9-10, 1725--1741.


%\bibitem{BL83}Ha\"{i}m Br\'{e}zis, Elliott H. Lieb, 
%\textit{A relation between pointwise convergence of functions and convergence of functionals}, 
%Proc. Amer. Math. Soc. 
%{\bf 88} (1983), no. 3, 486--490.

% \bibitem{CFR20}
% L. Campos, L. G. Farah, S. Roudenko
% \emph{Threshold solutions for the nonlinear Schr\"odinger equation},
% preprint, arXiv:2010.14434. 


\bibitem{Caz03} 
T. Cazenave, 
\textit{Semilinear Schr\"{o}dinger Equations}, 
Courant Lecture Notes in Mathematics, vol. 10, American Mathematical Society, Courant Institute of Mathematical Sciences, 2003.

% \bibitem{CW92}  
% T. Cazenave, F. B. Weissler, 
% \textit{Rapidly decaying solutions of the nonlinear Schr\"{o}dinger equation}, 
% Comm. Math. Phys. 
% {\bf 147} (1992), no. 1, 75--100.

\bibitem{Com}
V. Combet,
\textit{Multi-existence of multi-solitons for the supercritical
nonlinear {S}chr\"{o}dinger equation in one dimension}, 
Discrete Contin. Dyn. Syst. \textbf{34} (2014), no.~5, 1961--1993.

\bibitem{CMM}
R. C\^ote, Y. Martel, F. Merle,
\textit{Construction of multi-soliton
solutions for the {$L^2$}-supercritical g{K}d{V} and {NLS} equations}, 
Rev. Mat. Iberoam. \textbf{27} (2011), no.~1, 273--302.

%\bibitem{Dod14pre}
%Benjamin Dodson,
%\textit{Global well - posedness and scattering for the focusing, energy - critical nonlinear Schr\"{o}dinger problem in dimension $d = 4$ for initial data below a ground state threshold}, 
%preprint, arXiv:1409.1950.

%\bibitem{Dod15}
%Benjamin Dodson, 
%\textit{Global well-posedness and scattering for the mass critical nonlinear Schr\"{o}dinger equation with mass below the mass of the ground state}, 
%Adv. Math. 
%{\bf 285} (2015), 1589--1618.

% \bibitem{DoMu17}
% B. Dodson, J. Murphy, 
% \textit{A new proof of scattering below the ground state for the 3D radial focusing cubic NLS},
% Proc. Amer. Math. Soc. {\bf 145} (2017), no. 11, 4859--4867.

% \bibitem{DoMu18}
% B. Dodson, J. Murphy, 
% \textit{A new proof of scattering below the ground state for the non-radial focusing NLS}
% Math. Res. Lett. {\bf 25} (2018), no. 6, 1805--1825.

% \bibitem{DHR08} T. Duyckaerts, J. Holmer, S. Roudenko, 
% \textit{Scattering for the non-radial 3D cubic nonlinear Schr\"{o}dinger equation}, 
% Math. Res. Lett. {\bf 15} (2008), no. 6, 1233--1250.

% \bibitem{DLR20}
% T. Duyckaerts, O. Landoulsi, and S. Roudenko, 
% \textit{Threshold solutions in the focusing 3D cubic NLS equation outside a strictly convex obstacle}, preprint, arXiv:2010.07724.

% \bibitem{DuMe09}
% T. Duyckaerts, F. Merle, 
% \textit{Dynamic of threshold solutions for energy-critical NLS}, 
% Geom. Funct. Anal. {\bf 18} (2009), no. 6, 1787--1840.

% \bibitem{DuRo10}
% T. Duyckaerts, S. Roudenko,
% \emph{Threshold solutions for the focusing 3D cubic Schr\"odinger equation},
% Rev. Mat. Iberoamericana {\bf 26} No.1, (2010), 1--56.

%\bibitem{DWZ16} Dapeng Du, Yifei Wu, Kaijun Zhang,
%\textit{On Blow-up criterion for the Nonlinear Schr\"{o}dinger Equation},
%Discrete Contin. Dyn. Syst., 
%{\bf 36} (2016), 3639--3650.


%\bibitem{DHR08} Thomas Duyckaerts, Justin Holmer, Svetlana Roudenko, 
%\textit{Scattering for the non-radial 3D cubic nonlinear Schr\"{o}dinger equation}, 
%Math. Res. Lett. 
%{\bf 15} (2008), no. 6, 1233--1250.

%\bibitem{DJKM16pre}
%Thomas Duyckaerts, Hao Jia, Carlos Kenig, Frank Merle,
%\textit{Soliton resolution along a sequence of times for the focusing energy critical wave equation},
%preprint, arXiv:1601.01871. 

%\bibitem{DKM13}
%Thomas Duyckaerts, Calros E. Kenig, Frank Merle, 
%\textit{Classification of radial solutions of the focusing, energy-critical wave equation}, 
%Camb. J. Math. 
%{\bf 1} (2013), no. 1, 75--144. 

%\bibitem{DM09}
%Thomas Duyckaerts, Frank Merle, 
%\textit{Dynamic of threshold solutions for energy-critical NLS}, 
%Geom. Funct. Anal. 
%{\bf 18} (2009), no. 6, 1787--1840.

%\bibitem{DR10}
%Thomas Duyckaerts, Svetlana Roudenko, 
%\textit{Threshold solutions for the focusing 3D cubic Schr\"{o}dinger equation}, 
%Rev. Mat. Iberoam. 
%{\bf 26} (2010), no. 1, 1--56.

%\bibitem{DR15}
%Thomas Duyckaerts, Svetlana Roudenko, 
%\textit{Going beyond the threshold: scattering and blow-up in the focusing NLS equation}, 
%Comm. Math. Phys. 
%{\bf 334} (2015), no. 3, 1573--1615.

% \bibitem{FXC11} 
% %DaoYuan Fang, Jian Xie, Thierry Cazenave, 
% D. Y. Fang, J. Xie, T. Cazenave, 
% \textit{Scattering for the focusing energy-subcritical nonlinear Schr\"{o}dinger equation}, 
% Sci. China Math. 
% {\bf 54} (2011), no. 10, 2037--2062. 

\bibitem{Fib}
G. Fibich,
\textit{The Nonlinear Schr\"odinger Equation}.
Springer (2015).

\bibitem{FJ}
R.~Fukuizumi, L.~Jeanjean, 
\textit{Stability of standing waves for a nonlinear {S}chr\"{o}dinger equation with a repulsive {D}irac delta potential},
Disc. Cont. Dyn. Syst. \textbf{21} (2008), no.~1, 121--136.

\bibitem{GiVe79}
J. Ginibre, G. Velo, 
\textit{On a class of nonlinear Schr\"{o}dinger equations. I. The Cauchy problem, general case,} 
J. Funct. Anal., {\bf 32} (1979), 1--32.

% \bibitem{GHW04} 
% R. H. Goodman, P. J. Holmes, M. I. Weinstein,
% \emph{Strong NLS soliton-defect interactions}, 
% Physica D {\bf 192} (2004), 215--248.

%\bibitem{GV79} Jean Ginibre, Giorgio Velo, 
%\textit{On a class of nonlinear Schr\"{o}dinger equations. I. The Cauchy problem, general case}, 
%J. Funct. Anal. 
%{\bf 32} (1979), no. 1, 1--32.

%\bibitem{Gue14}
%Cristi Darley Guevara, 
%\textit{Global behavior of finite energy solutions to the d-dimensional focusing nonlinear Schr\"{o}dinger equation}, 
%Appl. Math. Res. Express. AMRX 2014, no. 2, 177--243.

%\bibitem{Gla77} Robert T. Glassey, 
%\textit{On the blowing up of solutions to the Cauchy problem for nonlinear Schr\"{o}dinger equations}, 
%J. Math. Phys. 
%{\bf18} (1977), no. 9, 1794--1797.

\bibitem{GI1}
S. Gustafson, T. Inui,
\textit{Threshold even solutions to the nonlinear Schr\"odinger equation with delta potential at high frequencies},
Preprint (2023). 

\bibitem{GI2}
S. Gustafson, T. Inui.
\textit{Scattering and blow-up for threshold even solutions to the nonlinear Schr\"odinger equation with repulsive delta potential at low frequencies},
Preprint (2023).

\bibitem{GIS}
S. Gustafson, T. Inui, I. Shimizu,
\textit{Multi-solitons for the nonlinear 
Schr\"odinger equation with repulsive Dirac delta potential},
Preprint (2023).

% \bibitem{HoRo08} J. Holmer, S. Roudenko,
% \textit{A sharp condition for scattering of the radial 3D cubic nonlinear Schr\"{o}dinger equation}, Comm. Math. Phys. {\bf 282} (2008), no. 2, 435--467. 

% \bibitem{HoRo10}
% J. Holmer, S. Roudenko,
% \textit{Divergence of infinite-variance nonradial solutions to the 3D NLS equation},
% Comm. Partial Differential Equations {\bf 35} (2010), no. 5, 878--905. 

\bibitem{II}
M. Ikeda, T. Inui, 
\textit{Global dynamics below the standing waves for the focusing semilinear Schr\"{o}dinger equation with a repulsive Dirac delta potential}, 
Anal. PDE 1{\bf 0} (2017), no. 2, 481--512.

% \bibitem{Inui17}
% T. Inui
% \textit{Global dynamics of solutions with group invariance for the nonlinear Schr\"{o}dinger equation}, 
% Commun. Pure Appl. Anal. 
% {\bf 16} (2017), no. 2, 557--590.

% \bibitem{Inui18}
% T. Inui, 
% \textit{Remarks on the global dynamics for solutions with an infinite group invariance to the nonlinear Schr\"{o}dinger equation},
% Harmonic analysis and nonlinear partial differential equations, 1--32, RIMS K\^{o}ky\^{u}roku Bessatsu, {\bf B70}, Res. Inst. Math. Sci. (RIMS), Kyoto, 2018.

\bibitem{IN}
K. Ishizuka, K. Nakanishi
\emph{Global dynamics around 2-solitons for the nonlinear damped Klein-Gordon Equations} (2022)

\bibitem{Jen}
J. Jendrej,
\textit{Construction of two-bubble solutions for the energy-critical NLS},
Anal. PDE  {\bf 10} (2017) no. 8, 1923--1959.

% \bibitem{KeMe06} C. E. Kenig, F. Merle,
% \textit{Global well-posedness, scattering and blow-up for the energy-critical, focusing, non-linear Schr\"{o}dinger equation in the radial case}, 
% Invent. Math. {\bf 166} (2006), no. 3, 645--675.

\bibitem{KMR}
J. Krieger, Y. Martel, P. Rapha\"el,
\textit{Two-soliton solutions to the three-dimensional gravitational Hartree equation},
Comm. Pure Appl. Math. {\bf 62} (2009), no. 11, 1501--1550.

\bibitem{LiPo15}
F. Linares, G Ponce,
\textit{Introduction to nonlinear dispersive equations}, Second edition, Universitext. Springer, New York, 2015. xiv+301 pp.

%\bibitem{KM08}
%Carlos E. Kenig, Frank Merle, 
%\textit{Global well-posedness, scattering and blow-up for the energy-critical focusing non-linear wave equation}, Acta Math. 
%{\bf 201} (2008), no. 2, 147--212.

%\bibitem{KMMV16pre}
%Rowan Killip, Satoshi Masaki, Jason Murphy, Monica Visan,
%\textit{Large data mass-subcritical NLS: critical weighted bounds imply scattering},
%preprint, arXiv:1606.01512. 

%\bibitem{KTV09}
%Rowan Killip, Terence Tao, Monica Visan, 
%\textit{The cubic nonlinear Schr\"{o}dinger equation in two dimensions with radial data},
%J. Eur. Math. Soc. (JEMS) 
%{\bf 11} (2009), no. 6, 1203--1258.

%\bibitem{KV10}
%Rowan Killip, Monica Visan, 
%\textit{The focusing energy-critical nonlinear Schr\"{o}dinger equation in dimensions five and higher}, 
%Amer. J. Math. 
%{\bf 132} (2010), no. 2, 361--424.

%\bibitem{KVZ08}
%Rowan Killip, Monica Visan,  Xiaoyi Zhang,
%\textit{The mass-critical nonlinear Schr\"{o}dinger equation with radial data in dimensions three and higher}, 
%Anal. PDE 
%{\bf 1} (2008), no. 2, 229--266.

%\bibitem{LZ09}
%Dong Li, Xiaoyi Zhang,  
%\textit{Dynamics for the energy critical nonlinear Schr\"{o}dinger equation in high dimensions}, 
%J. Funct. Anal. 
%{\bf 256} (2009), no. 6, 1928--1961.

%\bibitem{LL01}
%Elliott H. Lieb, Michael  Loss, 
%\textit{Analysis. Second edition.} 
%Graduate Studies in Mathematics, 
%{\bf 14}. American Mathematical Society, Providence, RI, 2001. xxii+346 pp.

%\bibitem{LP15} Felipe Linares, Gustavo Ponce, 
%\textit{Introduction to nonlinear dispersive equations. Second edition}, 
%Universitext. Springer, 
%New York, 2015. xiv+301 pp.



%\bibitem{Mar97} Yvan Martel,
%\textit{Blow-up for the nonlinear Schr\"{o}dinger equation in nonisotropic spaces}, 
%Nonlinear Anal. 
%{\bf 28} (1997), no. 12, 1903--1908.

\bibitem{MM}
Y Martel, F. Merle,
\textit{Multi-solitary waves for nonlinear Schr\"odinger equations},
Ann. Inst. Henri Poincar\'e {\bf 23}
(2006), 849-864.

\bibitem{MM2}
Y Martel, F. Merle,
\textit{Inelastic interaction of nearly equal solitons for the quartic gKdV equation},
Invent. Math {\bf 183} (2011) no. 3, 563--648.

\bibitem{MR} 
Y. Martel, P Rapha\"{e}l,
\textit{Strongly interacting blow up bubbles for the mass critical NLS},
Ann. Sci. ENS {\bf 51} (2018), 701--737.

%\bibitem{Mas15}
%Satoshi  Masaki, 
%\textit{A sharp scattering condition for focusing mass-subcritical nonlinear Schr\"{o}dinger equation}, 
%Commun. Pure Appl. Anal. 
%{\bf 14} (2015), no. 4, 1481--1531.

%\bibitem{Mas16pre}
%Satoshi Masaki, 
%\textit{Two minimization problems on non-scattering solutions to mass-subcritical nonlinear Schr\"{o}dinger equation},
%preprint, arXiv:1605.09234.

%\bibitem{NR15pre}
%Kenji Nakanishi, Tristan Roy,
%\textit{Global dynamics above the ground state for the energy-critical Schr\"{o}dinger equation with radial data},
%preprint, arXiv:1510.04479. 

%\bibitem{NS12CVPDE}
%Kenji Nakanishi, Wilhelm Schlag,
%\textit{Global dynamics above the ground state energy for the cubic NLS equation in 3D}, 
%Calc. Var. Partial Differential Equations 
%{\bf 44} (2012), no. 1-2, 1--45.

%\bibitem{NS12ARMA}
%Kenji Nakanishi, Wilhelm Schlag, 
%\textit{Global dynamics above the ground state for the nonlinear Klein-Gordon equation without a radial assumption}, 
%Arch. Ration. Mech. Anal. 
%{\bf 203} (2012), no. 3, 809--851. 

\bibitem{NR}
I. Naumkin, P. Rapha\"el,
\emph{On travelling waves of the non linear Schr\"{o}dinger equation escaping a potential well},
Ann. Henri Poincar\'e {\bf 21} (2020), 1677--1758.

%\bibitem{OgTs91} Takayoshi Ogawa,  Yoshio Tsutsumi, 
%\textit{Blow-Up of $H^1$ solution for the Nonlinear Schr\"{o}dinger Equation}, 
%J. Diff. Eq. 
%{\bf 92}, 317--330 (1991).

\bibitem{Olm}
E. Olmedilla,
\textit{Multiple pole solutions of the 
nonlinear Schr\"odinger equation},
Physica D {\bf 25} (1987), 330--346.

\bibitem{PT33}
G. P\"oschl, E. Teller,
\textit{Bemerkungen zur Quantenmechanik des anharmonischen Oszillators}
Zeitschrift f\"{u}r Physik
{\bf 83}, 143--151 (1933).

%\bibitem{Str77}
%Walter A. Strauss,
%\textit{Existence of Solitary Waves in Higher Dimensions},
%Commun. math. Phys. 
%{\bf 55}, (1977) 149--162.

%\bibitem{Str81_1} 
%Walter A. Strauss, 
%\textit{Nonlinear scattering theory at low energy}, 
%J. Funct. Anal. 
%{\bf 41} (1981), no. 1, 110--133.
%
%\bibitem{Str81_2} Walter A. Strauss, 
%\textit{Nonlinear scattering theory at low energy: sequel}, 
%J. Funct. Anal. 
%{\bf 43} (1981), no. 3, 281--293.

\bibitem{SS}
C. Sulem, P.-L. Sulem,
\textit{The Nonlinear Schr\"odinger Equation}.
Springer (1999).

%\bibitem{Tao06}
%Terence Tao, 
%\textit{Nonlinear dispersive equations. Local and global analysis},
%CBMS Regional Conference Series in Mathematics, 
%{\bf106}. Published for the Conference Board of the Mathematical Sciences, Washington, DC; by the American Mathematical Society, Providence, RI, 2006. xvi+373 pp.

% \bibitem{MMZ21}
% C. Miao, J. Murphy, J. Zheng,
% \emph{Threshold scattering for the focusing NLS with a repulsive potential}, preprint, arXiv:2102.07163. 

\bibitem{Ngu19}
%Ti\'{\^e}n Vinh Nguy\tilde{\^e}n,  
T. V. Nguy\~{\^e}n,  
\textit{Existence of multi-solitary waves with logarithmic relative distances for the NLS equation},
C. R. Math. Acad. Sci. Paris {\bf 357} (2019), no. 1, 13--58. 

% \bibitem{Str81}
% W. A. Strauss, 
% \textit{Nonlinear scattering theory at low energy}, 
% J. Functional Analysis {\bf 41} (1981), no. 1, 110--133.

\bibitem{Wein}
M. Weinstein, 
\textit{Modulational stability of ground states of nonlinear Schr\"odinger equations},
SIAM J. Math. Anal. {\bf 16} (1985) no. 3,

\bibitem{ZS}
T. Zhakarov, A.B. Shabat,
\textit{Exact theory of two-dimensional self-focusing and one-dimensional self-modulation of waves in nonlinear media},
Sov. Phys. JETP {\bf 34} (1972), 62--29.

\end{thebibliography}
\end{document}